\documentclass[final,1p,times]{elsarticle}

\usepackage{amssymb}
\usepackage{amsmath} \usepackage{amsthm}
\usepackage{xcolor}
\usepackage{soul}

\usepackage[utf8]{inputenc}
\usepackage[english]{babel}
\usepackage{tikz}
\usetikzlibrary{shapes.geometric}
\usetikzlibrary{arrows.meta,positioning,fit,calc}

\usepackage{graphicx}
\usepackage{caption, subcaption, float}
\usepackage[colorlinks=true, allcolors=blue]{hyperref}
\usepackage{verbatim}
\usepackage{algpseudocode}
\usepackage[linesnumbered,ruled,vlined]{algorithm2e} % Options for style
\usepackage{soul}
\usepackage{enumitem}

\theoremstyle{plain}
\newtheorem{theorem}{Theorem}

\theoremstyle{definition}
\newtheorem{definition}{Definition}

\theoremstyle{remark}
\newtheorem{remark}{Remark}          % numbered
\newtheorem*{remark*}{Remark}        % unnumbered
\usepackage[colorlinks=true, allcolors=blue]{hyperref}
\usepackage{algorithmicx}
\usepackage{algcompatible}
\usepackage{multirow}
\usepackage{booktabs}

\newlist{computesteps}{description}{1}
\setlist[computesteps]{
    style=sameline,
    leftmargin=!,
    widest=Step 10.,
    labelsep=0.75em,
    font=\normalfont\ttfamily
}
\newcommand{\Step}[2]{\item[Step #1.] #2}

\newcommand{\bx}{\mathbf{x}}

\newcommand{\bv}{\mathbf{v}}
\newcommand{\bK}{\mathbf{K}}
\newcommand{\bc}{\mathbf{c}}

\journal{journal name}

\begin{document}
\begin{frontmatter}

\title{A Parallel and Adaptive Mesh-Free Method for Discontinuous Coefficient Fields in Heterogeneous Porous Media}

\author[fsu]{Kapil Chawla}
\author[fsu]{Sanghyun Lee}
\author[ncsu]{Yeonjong Shin}

\affiliation[fsu]{organization={Department of Mathematics, Florida State University},
            addressline={1017 Academic Way}, 
            city={Tallahassee},
            postcode={32306}, 
            state={FL},
            country={USA}}

\affiliation[ncsu]{organization={Department of Mathematics, North Carolina State University},
             addressline={2311 Stinson Dr},
             city={Raleigh},
             postcode={27607},
             state={NC},
             country={USA}}

\begin{abstract}

Discontinuous coefficient fields arise in many computational physics problems and are often represented as cellwise constant data tied to a given spatial discretization. Such representations are inherently mesh-dependent, requiring interpolation or projection whenever they are transferred to a different discretization. 
In this work, we develop \emph{Parallel and Adaptive Mesh-Free Approximation (PAM)}, a mesh-independent framework that approximates discontinuous data by a continuous, closed-form function. The resulting approximation can be evaluated consistently across different geometries and numerical discretizations, while preserving sharp interface features.

The proposed PAM framework employs radial basis functions (RBFs) to construct continuous approximations of discontinuous data. To accurately capture discontinuities, we incorporate Shepard-normalization, which stabilizes the approximation near sharp interfaces. The coefficients of the RBF expansion are determined via sparse regression, enabling automatic selection of the most relevant basis functions and promoting robust representations. In addition, we develop a novel adaptive refinement approach which further enriches the approximation in regions of rapid spatial variation. We provide a theoretical analysis showing that the proposed normalized RBF framework achieves arbitrarily small $L^1$ error in approximating discontinuous step functions.

To enhance computational efficiency, the domain is partitioned into subdomains, and the reconstruction problem is solved independently on each subdomain in parallel.
Numerical experiments demonstrate the accuracy, adaptivity, scalability, and downstream impact of the proposed method on Darcy flow simulations, including tests on heterogeneous permeability fields, mesh-transfer settings, and the SPE10 benchmark.
\end{abstract}

%%Research highlights
%\begin{highlights}
%\item Research highlight 1
%\item Research highlight 2
%\end{highlights}

%% Keywords

\end{frontmatter}

%\linenumbers

\section{Introduction and Background}

Field-scale numerical simulation in heterogeneous porous media is central to many applications in environmental, civil, and hydraulic engineering, including groundwater management, contaminant remediation, oil and gas production, geothermal energy, and rare earth mining \cite{bear1972dynamics,dagan1989flow,durlofsky1991numerical}. A key challenge in these applications is the representation of highly heterogeneous subsurface material properties (e.g., permeability, porosity, hydraulic conductivity, etc.), which may vary by several orders of magnitude and contain sharp discontinuities \cite{hou1997multiscale,efendiev2009multiscale,lee2016pressure,choo2018enriched,lee2026}. In practice, permeability fields are often described as discontinuous, piecewise constant data tied to a given spatial discretization. When performing simulations on different meshes or discretizations, the
permeability field must be interpolated or projected each time to fit the new
mesh. Such situations arise, for example, in mesh refinement studies, adaptive
local refinement, multiscale simulations, or when transferring material data
between different numerical solvers; however, interpolation or projection may
smear sharp interfaces and introduce mesh-dependent artifacts for discontinuous
piecewise constant fields.
Thus, in this work, we propose a novel approach for approximating discontinuous
data by a closed-form continuous function using parallel and adaptive radial
basis functions. Once learned, this function can be utilized on any given
discretization.

Radial basis function (RBF) approximation provides a natural mesh-free framework for multivariate scattered-data approximation~\cite{powell1987rbf,buhmann2003rbf,wendland2005scattered}. Unlike tensor-product polynomial interpolants defined on structured grids, RBF approximants are constructed from radial kernels centered at scattered locations, offering geometric flexibility and the ability to handle irregularly distributed data. These properties make RBFs particularly attractive for representing fields independently of an underlying mesh. In addition, strictly positive definite kernels, such as the Gaussian, yield well-posed interpolation systems for distinct data sites~\cite{micchelli1986interpolation}, and RBF networks possess universal approximation capabilities on compact domains~\cite{park1991universal}.

However, classical global RBF interpolation constructs a single smooth surrogate over the entire domain, which can be restrictive for highly heterogeneous or discontinuous coefficient fields. Accurately representing sharp transitions within a globally smooth expansion may require a large number of RBF basis functions and can lead to localized oscillations near interfaces~\cite{fornberg2007runge}. Furthermore, global formulations result in dense linear systems whose size scales with the number of data points, limiting scalability for large problems~\cite{wendland2005scattered,majdisova2017biggeo}.

To address these challenges, we propose \emph{Parallel and Adaptive Mesh-Free Approximation} (PAM), a localized and sparse RBF framework for constructing a mesh-independent surrogate of discontinuous permeability fields. This method combines radial basis functions with Shepard normalization to yield stable, locally weighted approximations that remain robust in the presence of sharp transitions.

To improve accuracy, PAM incorporates an adaptive refinement strategy that targets regions of rapid variation in the permeability field. In particular, local approximation quality is assessed through error indicators, and the RBF representation is selectively enriched in regions where the error is large. This enrichment is achieved by increasing the number of RBF centers and adjusting the shape parameters of the basis functions, allowing the method to resolve sharp interfaces while avoiding unnecessary basis enrichment in smooth regions. To improve efficiency, the computational domain is partitioned into nonoverlapping subregions, and independent sparse RBF regressions are performed within each subregion. The resulting local approximations are then combined into a continuous coefficient field that can be evaluated on arbitrary finite element or finite-volume discretizations, while the subdomain structure enables efficient parallel computation.

We also provide theoretical analysis showing that the proposed normalized RBF framework achieves arbitrarily small $L^1$ error in approximating discontinuous step functions. The performance of PAM is then assessed through a sequence of numerical experiments designed to isolate and demonstrate its key components. We first construct a baseline RBF approximation of the permeability field using a uniform distribution of centers across the domain. In this setting, one RBF center is assigned per computational element (e.g., at cell centers), and fixed shape parameters are used. Next, we apply the adaptive refinement strategy to the baseline approximation and quantify its effect on resolving sharp interfaces. We then apply the parallel domain decomposition scheme that solves independent local RBF problems across subregions to highlight its computational speedup and demonstrate its scalability for large-scale problems. To assess the impact of the reconstructed permeability, we solve the Darcy equation using the learned coefficient field and compare the resulting pressure solutions against reference solutions. Finally, we apply PAM to the SPE10 benchmark dataset \cite{christie2001tenth} to demonstrate its performance on realistic, high-contrast heterogeneous permeability fields.

The remainder of this paper is organized as follows. Section~2 introduces the PAM framework, including the problem formulation and main algorithmic components. Section~3 presents the adaptive RBF construction and the associated sparse regression strategy based on the Elastic Net. Section~4 provides the theoretical analysis, including approximation of step functions using Shepard-normalized RBFs and the associated error analysis. Section~5 contains numerical experiments that validate the accuracy, adaptivity, and scalability of the proposed method, including applications to high-contrast heterogeneous permeability fields such as the SPE10 benchmark.

\section{Overview of the proposed PAM algorithm}
\label{sec:2}

The main objective is to construct a continuous approximation of a permeability field from discontinuous piecewise constant data. 
As a motivating example, and for simplicity of presentation, we consider
single-phase flow in a heterogeneous porous medium governed by the Darcy equation
\begin{equation}
-\nabla \cdot \bigl(K(\bx)\nabla p(\bx)\bigr) = f(\bx),
\qquad \bx \in \Omega \subset \mathbb{R}^d,
\end{equation}
where $p:\Omega\to\mathbb{R}$ denotes the pressure field, $f$ is a source term,
and $K$ is a scalar-valued permeability field. In this work, we consider
$d=1,2$.

Let $\mathcal{T}_h=\{T\}$ be a shape-regular partition of $\Omega$. For
$d=1$, the elements are intervals, whereas for $d=2$, the elements are triangles
or quadrilaterals. We denote by
\[
h := \max_{T\in\mathcal{T}_h} h_T,
\]
where $h_T$ is the diameter of the element $T$.

Throughout this paper, the permeability field $K(\bx)$ is assumed to be given
as discontinuous piecewise constant data on the underlying mesh $\mathcal{T}_h$,
allowing jumps across element interfaces. More precisely, we define
\[
K_T := K|_T, \qquad T\in\mathcal{T}_h,
\]
and assume that
\[
0<K_{\min}\le K_T\le K_{\max}<\infty
\qquad \text{for all } T\in\mathcal{T}_h .
\]
Examples of such permeability fields are shown in Figure~\ref{fig:example_of_K}.

\begin{figure}[!h]
    \centering
    \includegraphics[width=0.34\linewidth]{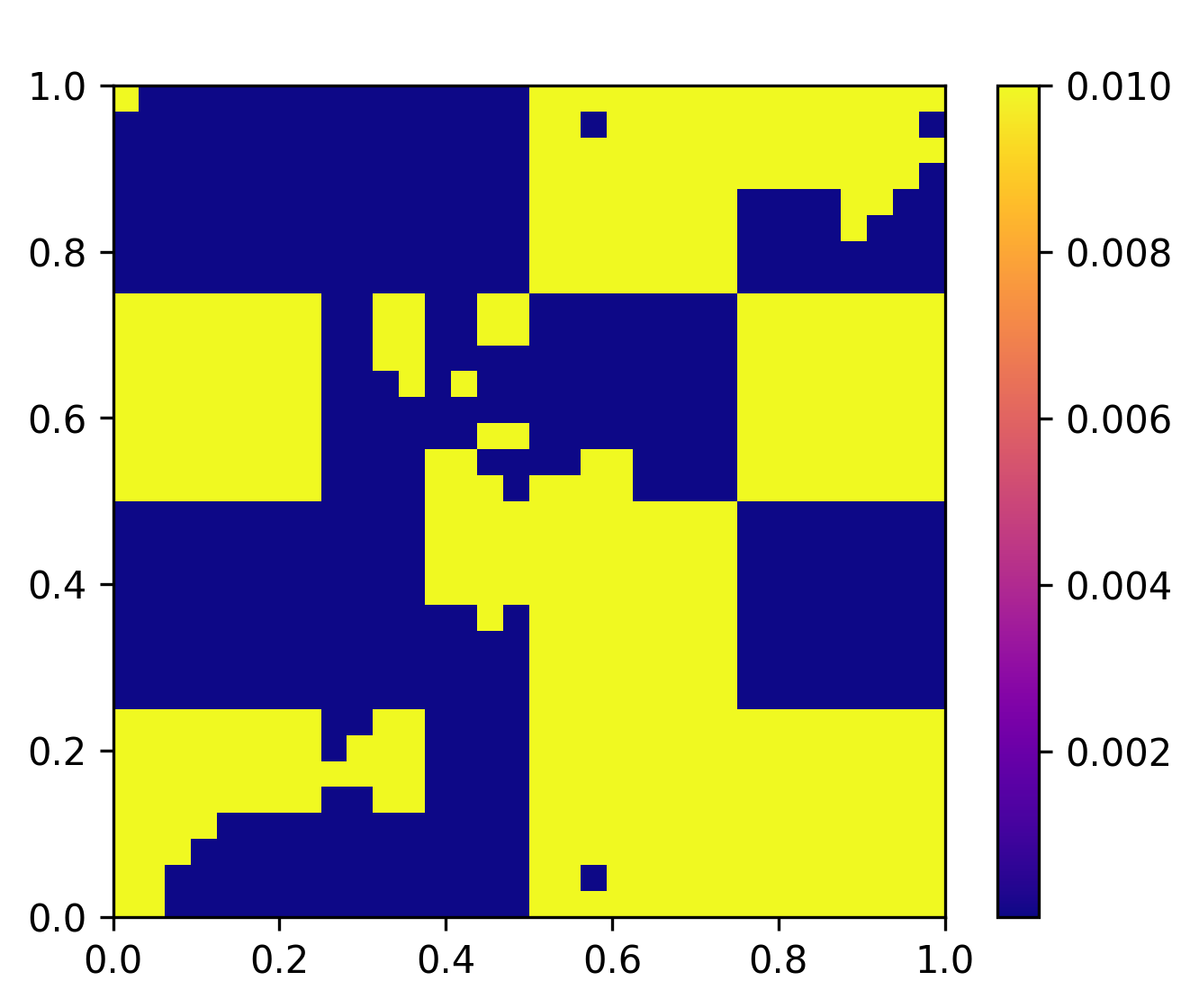}
    \includegraphics[width=0.35\linewidth]{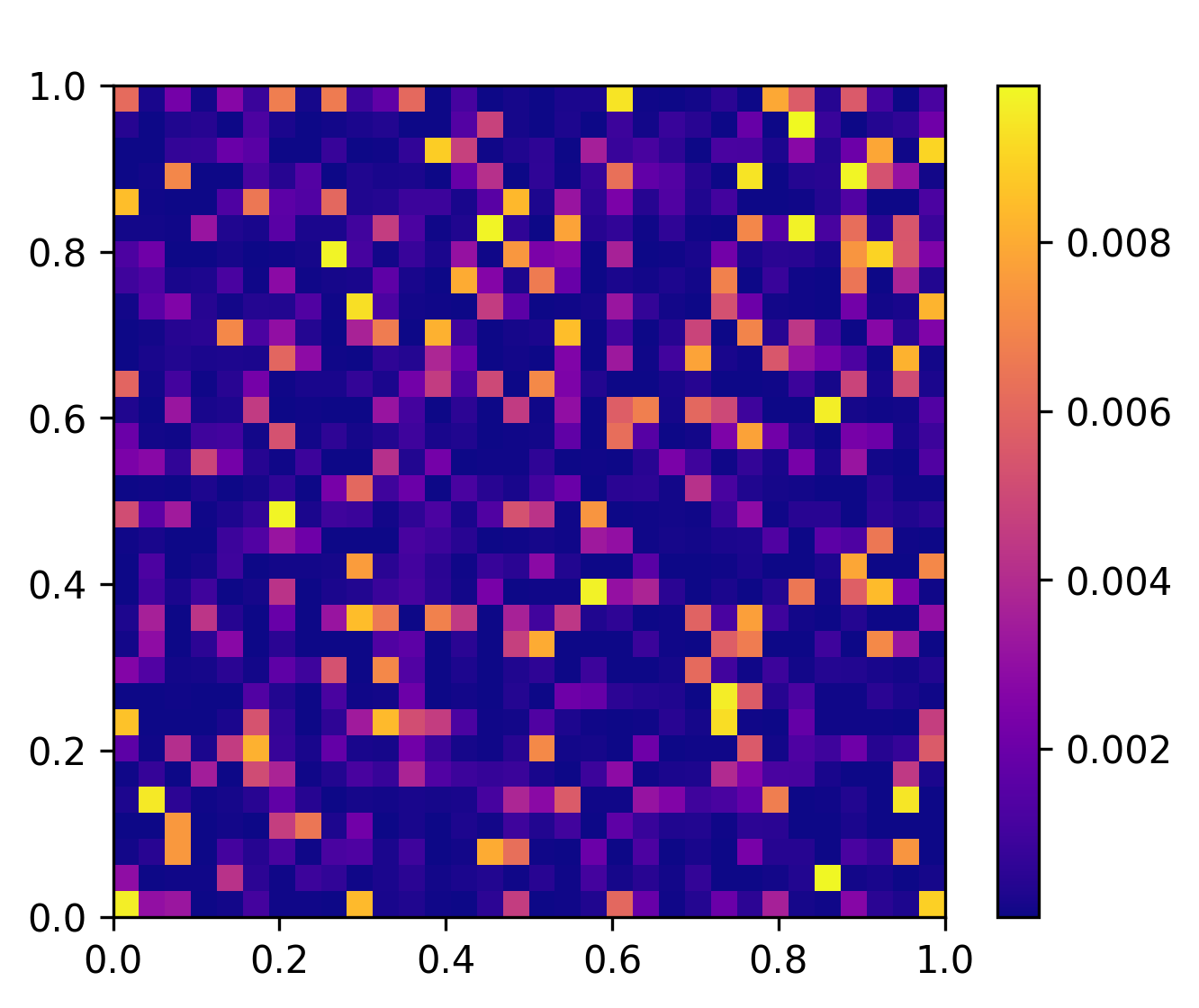}
    \caption{Examples of heterogeneous and discontinuous permeability values ($K(\bx)$) on a $32\times 32$ mesh in a two-dimensional domain.}
    \label{fig:example_of_K}
\end{figure}

The PAM framework consists of four conceptual stages: domain decomposition, local mesh-free regression, adaptive refinement, and global reconstruction.

\begin{computesteps}
\Step{1}{\textit{Subdivide the domain $\Omega$.} The computational domain $\Omega$ is partitioned into $N$ nonoverlapping Lipschitz subdomains $\{\omega_i\}_{i=1}^N$ such that
$$
\overline{\Omega} = \bigcup_{i=1}^N \overline{\omega_i},
\qquad
\omega_i \cap \omega_j = \emptyset \quad (i\neq j).
$$
This decomposition enables independent and parallel computation on each subdomain.}
    
\Step{2}{\textit{Local mesh-free permeability approximation.}
For each subdomain $\omega_i$, we construct a local approximation of the permeability by solving a sparse regression problem based on the underlying piecewise constant data. We assume that each subdomain $\omega_i$ is a union of elements of $\mathcal{T}_h$. Let
$\mathcal{T}_h(\omega_i) := \{ T \in \mathcal{T}_h \;:\; T \subset \overline{\omega_i} \}$
denote the set of mesh elements contained in $\omega_i$, and let $N_i := |\mathcal{T}_h(\omega_i)|$ be the number of such elements.

For each element $T \in \mathcal{T}_h(\omega_i)$, we associate a sampling point $\bx_T \in T$ (e.g., the centroid of $T$). Since the permeability field is piecewise constant, its value on $T$ is denoted by $K_T$.

We define the local data vector
$$
\bK_{\omega_i}
=
\begin{pmatrix}
K_{T_1} \\
\vdots \\
K_{T_{N_i}}
\end{pmatrix}
\in \mathbb{R}^{N_i},
$$
where $\{T_j\}_{j=1}^{N_i}$ are the elements in $\mathcal{T}_h(\omega_i)$. Let $\{\varphi_m^{(i)}\}_{m=1}^{M_i}$ be a set of local basis functions defined on $\omega_i$. The corresponding feature matrix
$
\Phi_{\omega_i} \in \mathbb{R}^{N_i \times M_i}
$
is defined by
$$
(\Phi_{\omega_i})_{j,m} = \varphi_m^{(i)}(\bx_{T_j}),
\qquad
1 \le j \le N_i,\;\; 1 \le m \le M_i.
$$
For simplicity, we choose one RBF basis per element in $\omega_i$, i.e., $M_i = N_i$. We compute the local coefficient vector $\boldsymbol{\beta}^{(i)} \in \mathbb{R}^{M_i}$ by solving the Elastic Net problem~\cite{zou2005elastic}
\begin{equation}\label{elastic}
    \min_{\boldsymbol{\beta}^{(i)}\in\mathbb{R}^{M_i}}
    \frac12\|\bK_{\omega_i}-\Phi_{\omega_i}\boldsymbol{\beta}^{(i)}\|_2^{2}
    +\lambda_1\|\boldsymbol{\beta}^{(i)}\|_1
    +\frac{\lambda_2}{2}\|\boldsymbol{\beta}^{(i)}\|_2^{2},
\end{equation}
where $\lambda_1,\lambda_2>0$ are regularization parameters controlling sparsity and stability. Here, $\lambda_1$ and $\lambda_2$ denote
the parameters used in the local Elastic Net problem on a given subdomain
$\omega_i$. In the parallel implementation, these parameters may be selected
independently for different subdomains, although we keep the notation
$\lambda_1$ and $\lambda_2$ for simplicity. Details of the basis construction and regression procedure are given in Section~\ref{sec:adaptive_rbf}. The resulting coefficients define a continuous local surrogate
\begin{equation}
K^*_{\omega_i}(\bx)
=
\sum_{m=1}^{M_i} \beta_m^{(i)} \,\varphi_m^{(i)}(\bx),
\qquad \bx\in\omega_i,
\label{eq:local_Kstar}
\end{equation}
whose evaluation at the sampling points yields the reconstructed vector
$
\bK^*_{\omega_i} = \Phi_{\omega_i}\boldsymbol{\beta}^{(i)}.
$
}

\Step{3}{\textit{Residual-driven adaptive refinement.}
We assess the local approximation quality on each subdomain $\omega_i$ using an elementwise indicator. For each element $T \in \mathcal{T}_h(\omega_i)$, we define
\begin{equation}
R_T
=
\sum_{\ell=1}^{n_q}
w_T^\ell
\bigl(K^*_{\omega_i}(\bx_T^\ell) - K_T\bigr)^2,
\end{equation}
where $\{\bx_T^\ell, w_T^\ell\}_{\ell=1}^{n_q}$ denote the quadrature points and weights on the element $T$, and $n_q$ is the number of quadrature points associated with the chosen quadrature rule. This quantity provides an approximation of the elementwise $L^2$ error between the reconstructed permeability and the piecewise constant data.

Elements with the largest indicators are marked for refinement, and the local RBF basis is enriched in the corresponding regions. Steps~2 and~3 are repeated until a prescribed stopping criterion is satisfied.
}

\Step{4}{\textit{Global reconstruction.}
Using the learned coefficients from the local regressions, we construct on each subdomain $\omega_i$ a continuous local surrogate $K^*_{\omega_i}(\bx)$ in closed form. The local surrogates $\{K^*_{\omega_i}\}_{i=1}^N$ are then assembled into a global permeability field $K^*(\bx)$ on $\Omega$ by
$$
K^*(\bx) = K^*_{\omega_i}(\bx), \qquad \bx \in \omega_i.
$$
Thus, $K^*$ is a mesh-independent, piecewise continuous, closed-form representation of the permeability field with respect to the subdomain partition. It can therefore be evaluated at arbitrary spatial locations and directly used in the numerical solution of PDEs on arbitrary discretizations and spatial resolutions. Figure~\ref{fig:overall} illustrates the overview of the proposed algorithm.
}

\end{computesteps}

\begin{remark}
When multiple subdomains are used, the piecewise definition above does not in general guarantee continuity of $K^*$ across subdomain interfaces, since each local surrogate $K^*_{\omega_i}$ is constructed independently. If a globally continuous reconstruction is desired, one may introduce an additional interface treatment, such as a gluing or averaging procedure across neighboring subdomains.
\end{remark}

\begin{figure}[!h]
    \centering
    \includegraphics[width=\linewidth]{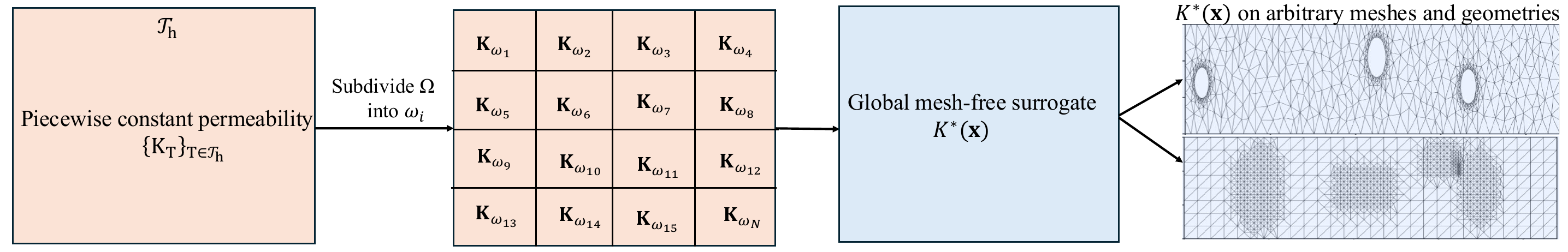}
    \caption{Overview of the PAM framework. Piecewise-constant permeability data on $\mathcal T_h$ are partitioned into subdomains $\{\omega_i\}_{i=1}^N$, where local adaptive RBF approximations are constructed and assembled into a global mesh-free representation $K^*(\mathbf{x})$. The surrogate enables evaluation on arbitrary meshes and geometries.}
    \label{fig:overall}
\end{figure}

%\clearpage
%\newpage

\section{Adaptive Selection of Radial Basis Functions and Elastic Net Algorithm}\label{sec:adaptive_rbf}

This section provides the algorithmic realization of
Step 2 (local mesh-free regression) and Step 3
(residual-driven adaptive refinement) of the PAM framework introduced in Section~\ref{sec:2}.
We specify the construction of the local basis functions, the Elastic Net regression problem \cite{zou2005elastic}, and the adaptive enrichment strategy used to approximate the permeability data on each subdomain.

\subsection{Radial Basis Functions}

On each subdomain $\omega_i$, let $\phi:[0,\infty)\to\mathbb{R}$ denote a radial kernel. A radial basis function (RBF) centered at $\bc\in\mathbb{R}^d$ is defined by
\begin{equation}
\varphi(\bx;\bc) := \phi(\|\bx-\bc\|_2),
\end{equation}
where $\bx\in\mathbb{R}^d$ and $\|\cdot\|_2$ denotes the Euclidean norm. In this work, we employ Gaussian radial kernels, which yield smooth and spatially localized basis functions well suited for approximating heterogeneous permeability fields~\cite{buhmann2003rbf}.

Specifically, the $m$-th local basis function on $\omega_i$ is given by
\begin{equation}
\varphi_m^{(i)}(\bx)
=
\exp\!\left(-\frac{\|\bx-\bc_m^{(i)}\|_2^2}{2(\sigma_m^{(i)})^2}\right),
\label{eq:gaussian_rbf}
\end{equation}
where  {$\bc_m^{(i)}\in\omega_i$} denotes the center and $\sigma_m^{(i)}>0$ is the width parameter controlling spatial localization. The collection $\{\bc_m^{(i)},\sigma_m^{(i)}\}_{m=1}^{M_i}$ defines the local RBF dictionary on $\omega_i$.

The local permeability surrogate $K^*_{\omega_i}(\bx)$, defined in Step~2 of Section~\ref{sec:2}, is given by
\begin{equation}
K^*_{\omega_i}(\bx)
=
\sum_{m=1}^{M_i} \beta_m^{(i)} \,\varphi_m^{(i)}(\bx),
\qquad \bx\in\omega_i,
\label{eq:local_rbf_surrogate}
\end{equation}
where the coefficient vector $\boldsymbol{\beta}^{(i)}\in\mathbb{R}^{M_i}$ is computed by solving the elastic-net problem~\eqref{elastic}.

Evaluating~\eqref{eq:local_rbf_surrogate} at the sampling points $\{\bx_{T_j}\}_{j=1}^{N_i}\subset\omega_i$ yields the discrete relation
\begin{equation}
\bK^*_{\omega_i}
=
\Phi_{\omega_i}\boldsymbol{\beta}^{(i)},
\end{equation}
where $\Phi_{\omega_i}\in\mathbb{R}^{N_i\times M_i}$ is the feature matrix with entries
$$
(\Phi_{\omega_i})_{j,m}=\varphi_m^{(i)}(\bx_{T_j}).
$$

\subsection{Shepard Normalization and Partition of Unity RBFs}
\label{subsec:shepard}

Although standard RBFs provide powerful approximation properties, their raw evaluations can produce ill-conditioned feature matrices, especially when basis functions vary in amplitude.
To mitigate these issues, we employ a Shepard-normalized RBF construction, which forms a smooth partition of unity and is defined as

\begin{definition}[Shepard Normalization~\cite{Shepard1968, MelenkBabuska1996, BabuskaMelenk1997}]
On each subdomain $\omega_i$, given nonnegative local RBFs $\{\varphi_m^{(i)}(\bx)\}_{m=1}^{M_i}$ centered at $\{\bc_m^{(i)}\}_{m=1}^{M_i}$, the normalized weights are defined pointwise as
\begin{equation}
    w_m^{(i)}(\bx)
    \;=\;
    \frac{\varphi_m^{(i)}(\bx)}{\sum_{k=1}^{M_i}\varphi_k^{(i)}(\bx)},
    \qquad m = 1, \dots, M_i.
\end{equation}
Whenever $\sum_{k=1}^{M_i} \varphi_k^{(i)}(\bx) > 0$, the weights satisfy
$$
\sum_{m=1}^{M_i} w_m^{(i)}(\bx) = 1,
$$
and therefore form a partition of unity.
\end{definition}

In practice, Shepard normalization is applied to the sampled evaluations of the local basis functions. Given the feature matrix $\Phi_{\omega_i}\in\mathbb{R}^{N_i\times M_i}$ %with entries
%$$
%(\Phi_{\omega_i})_{j,m} = \varphi_m^{(i)}(\bx_{T_j}),
%$$
evaluated at representative points $\{\bx_{T_j}\}_{j=1}^{N_i}\subset\omega_i$, the normalized matrix $W_{\omega_i}$ is obtained by row-wise scaling,
\begin{equation}\label{eqn:shepard_rbf}
    (W_{\omega_i})_{j,m}
    =
    \frac{{(\Phi_{\omega_i})_{j,m}}}{\sum_{k=1}^{M_i}{(\Phi_{\omega_i})_{j,k}}},
    \qquad j=1,\dots,N_i.
\end{equation}
Each row of $W_{\omega_i}$ sums to unity, which mitigates scaling disparities among features and can improve the numerical conditioning of the Elastic Net regression.

\begin{figure}[!h]
  \centering

  \begin{subfigure}[!h]{0.32\linewidth}
    \centering
    \includegraphics[width=\linewidth]{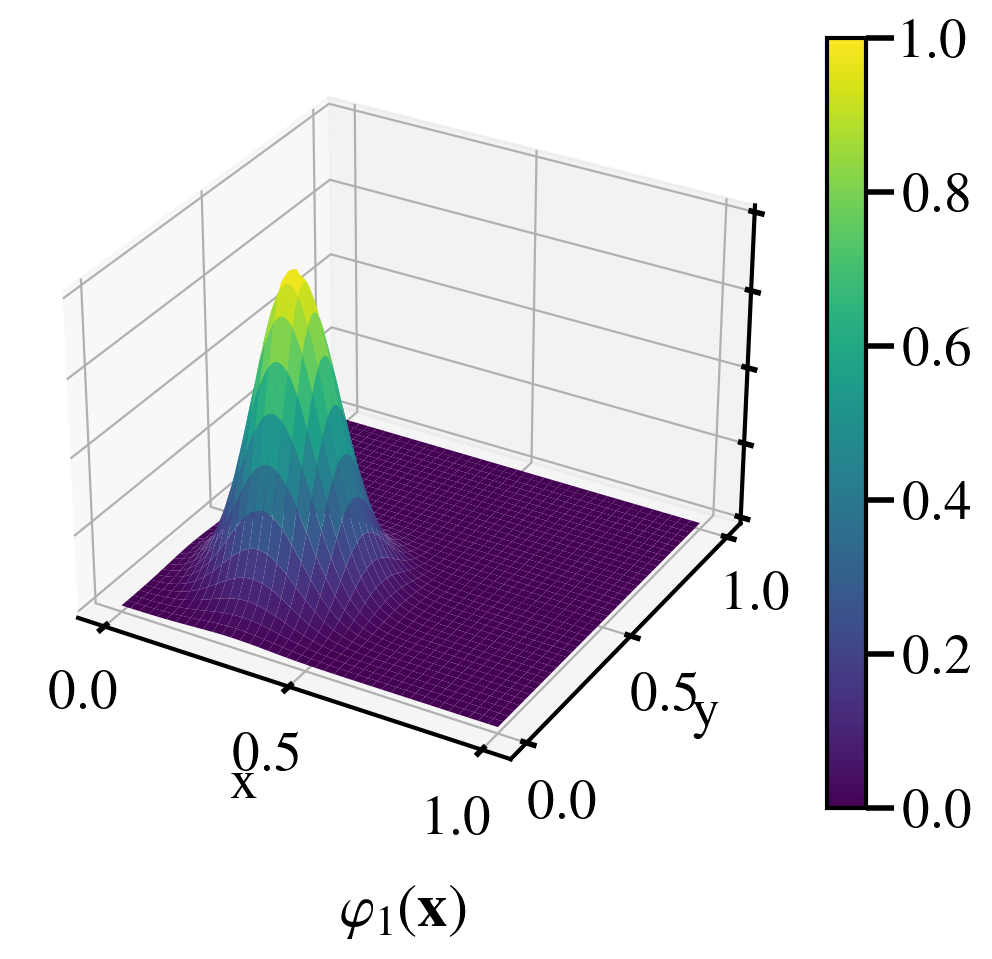}
    \caption{}
  \end{subfigure}
  \hfill
  \begin{subfigure}[!h]{0.32\linewidth}
    \centering
    \includegraphics[width=\linewidth]{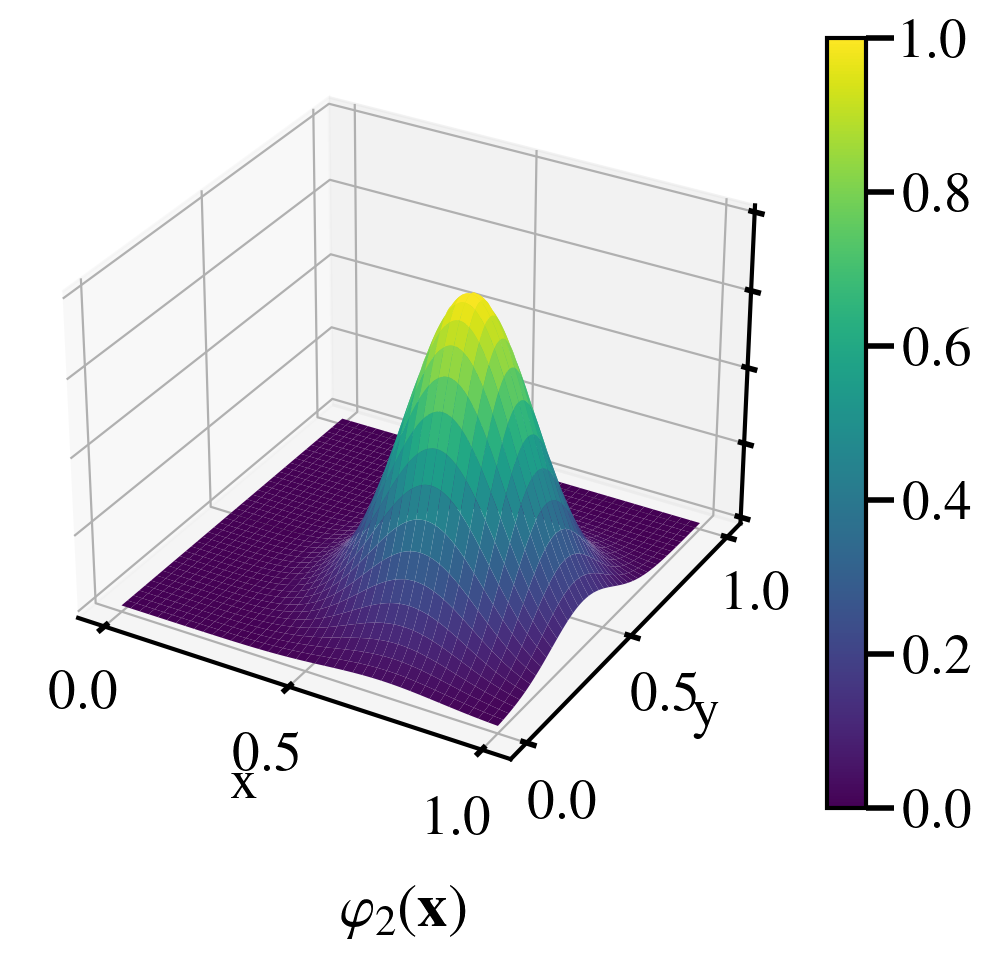}
    \caption{}
  \end{subfigure}
  \hfill
  \begin{subfigure}[!h]{0.32\linewidth}
    \centering
    \includegraphics[width=\linewidth]{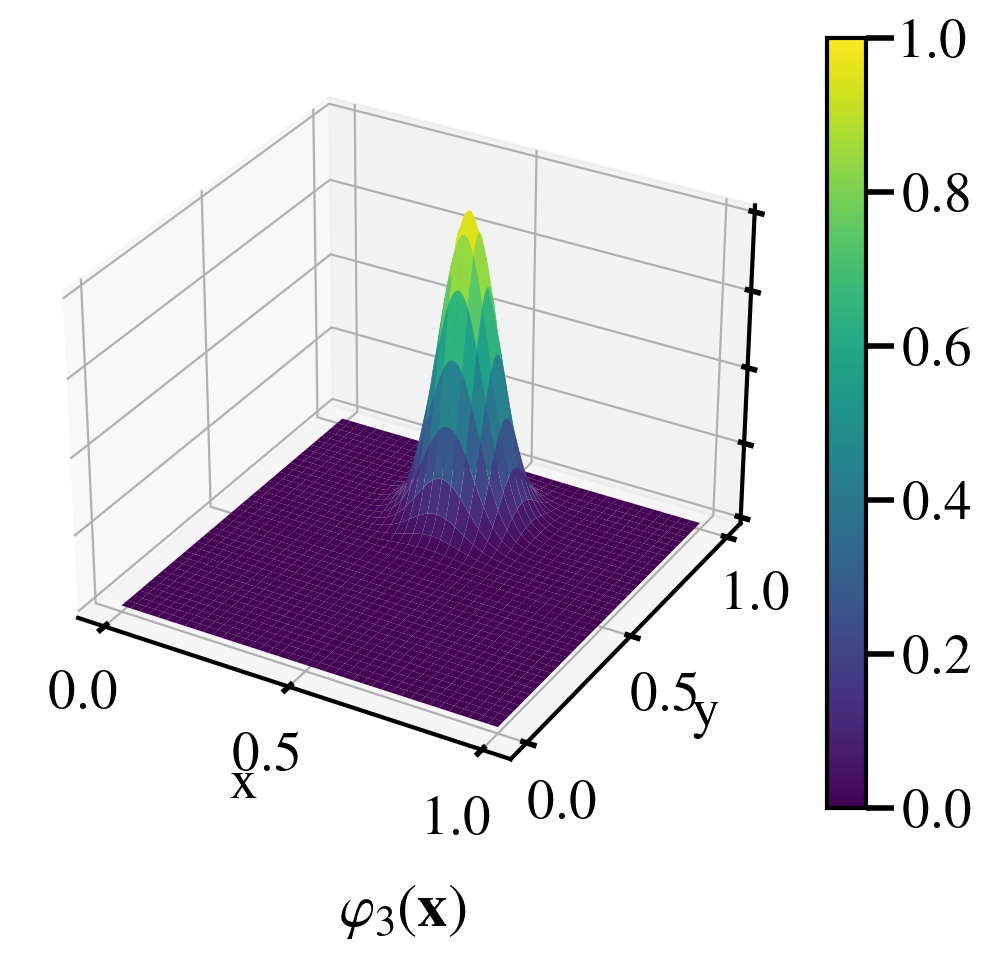}
    \caption{}
  \end{subfigure}

  \caption{Individual Gaussian RBFs: (a) $\varphi_1(\bx)$ centered at $\bc_1$, (b) $\varphi_2(\bx)$ centered at $\bc_2$, and (c) $\varphi_3(\bx)$ centered at $\bc_3$.}
  \label{fig:rbf_example_1}
\end{figure}

To illustrate the effect of Shepard normalization, we consider a simple example on the domain $\Omega = [0,1]^2$ involving three Gaussian RBFs $\varphi_1(\bx)$, $\varphi_2(\bx)$, and $\varphi_3(\bx)$ with centers $\bc_1=(0.3,\,0.3)$, $\bc_2=(0.7,\,0.4)$, and $\bc_3=(0.5,\,0.75)$, and widths $\sigma_1=0.10$, $\sigma_2=0.15$, and $\sigma_3=0.07$, as shown in Figure~\ref{fig:rbf_example_1}. The unnormalized superposition $\sum_{m=1}^3 \beta_{m} \varphi_{m}(\bx)$,
shown in Figure~\ref{fig:rbf_pu_blend}(a-b), is smooth but unbalanced, with regions of overlap exhibiting amplified magnitudes. However, after applying Shepard normalization, the resulting weights
$$
w_{m}(\bx)=\frac{\varphi_{m}(\bx)}{\sum_{k=1}^3 \varphi_{k}(\bx)},
$$
(Figure~\ref{fig:rbf_pu_example}) form a proper partition of unity, and the field
$K^*(\bx)=\sum_{m=1}^3 \beta_{m} w_{m}(\bx)$
(Figure~\ref{fig:rbf_pu_blend}(c)) becomes balanced and bounded between $\min_{m} \beta_{m}$ and $\max_{m} \beta_{m}$.  
\begin{figure}[!h]
  \centering

  \begin{subfigure}[!h]{0.32\linewidth}
    \centering
    \includegraphics[width=\linewidth]{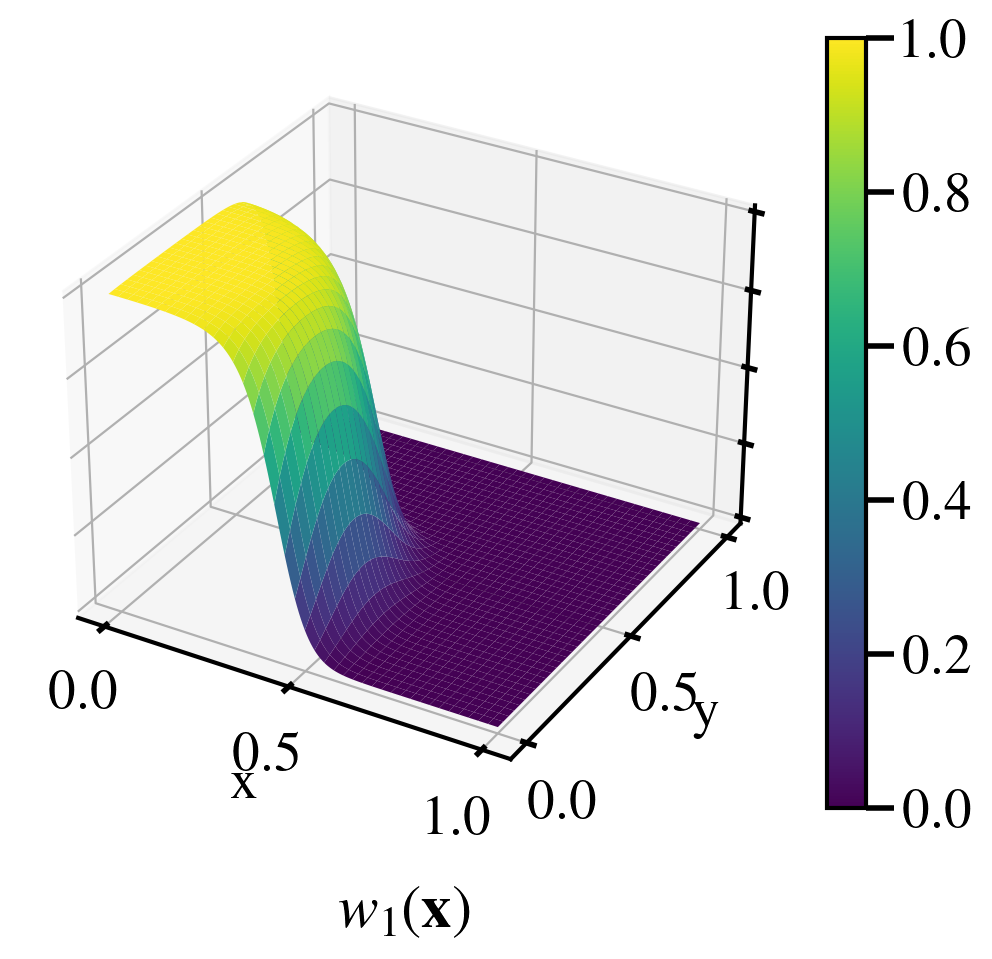}
    \caption{}
  \end{subfigure}
  \hfill
  \begin{subfigure}[!h]{0.32\linewidth}
    \centering
    \includegraphics[width=\linewidth]{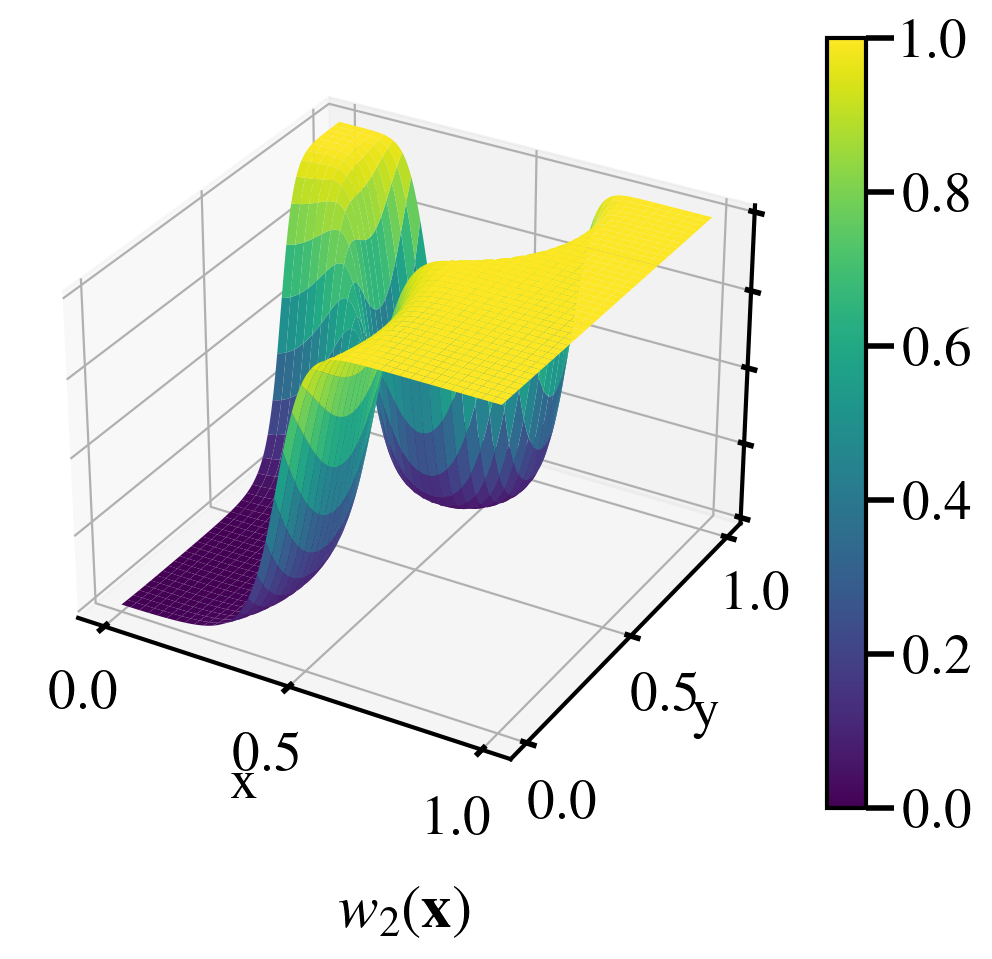}
    \caption{}
  \end{subfigure}
  \hfill
  \begin{subfigure}[!h]{0.32\linewidth}
    \centering
    \includegraphics[width=\linewidth]{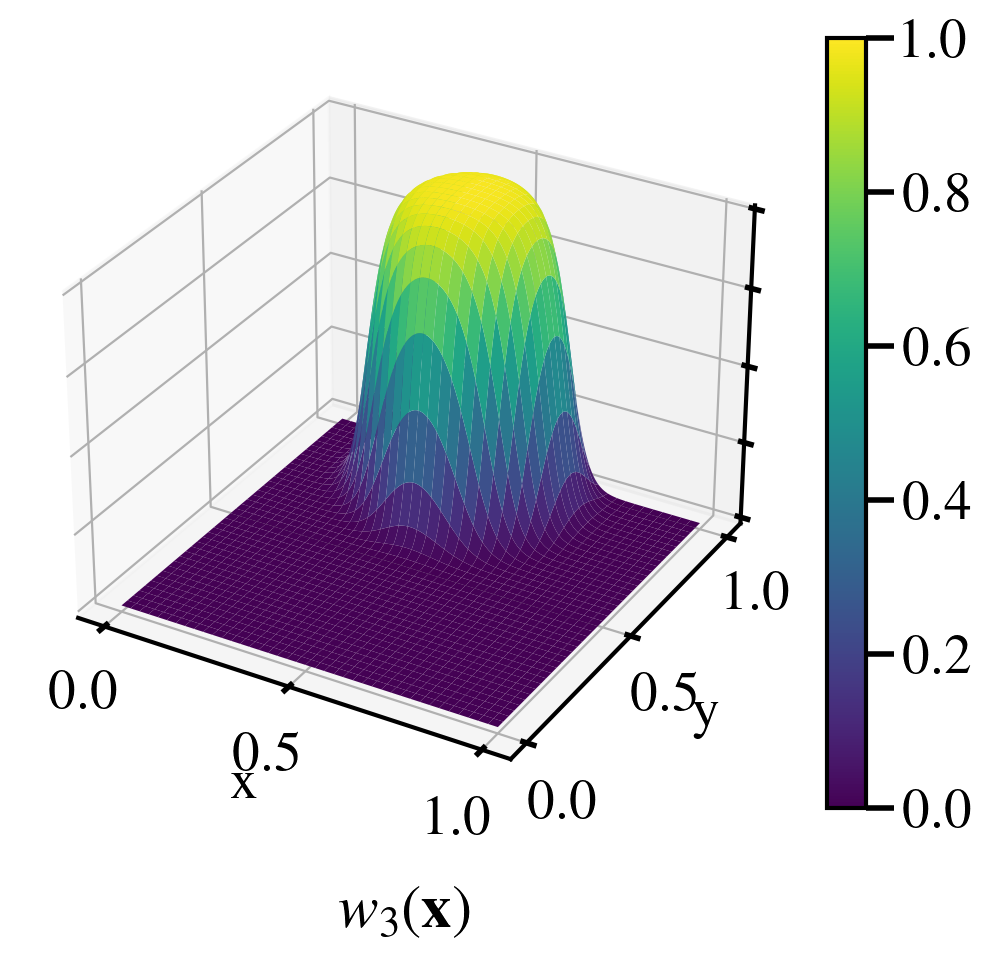}
    \caption{}
  \end{subfigure}

  \caption{Shepard-normalized weights:
(a) $w_1(\bx)$, (b) $w_2(\bx)$, and (c) $w_3(\bx)$,
each satisfying $0 \le w_m(\bx) \le 1$ and $\sum_{m=1}^3 w_m(\bx)=1$.}
  \label{fig:rbf_pu_example}
\end{figure}

\begin{figure}[!h]
  \centering
  \includegraphics[width=\linewidth]{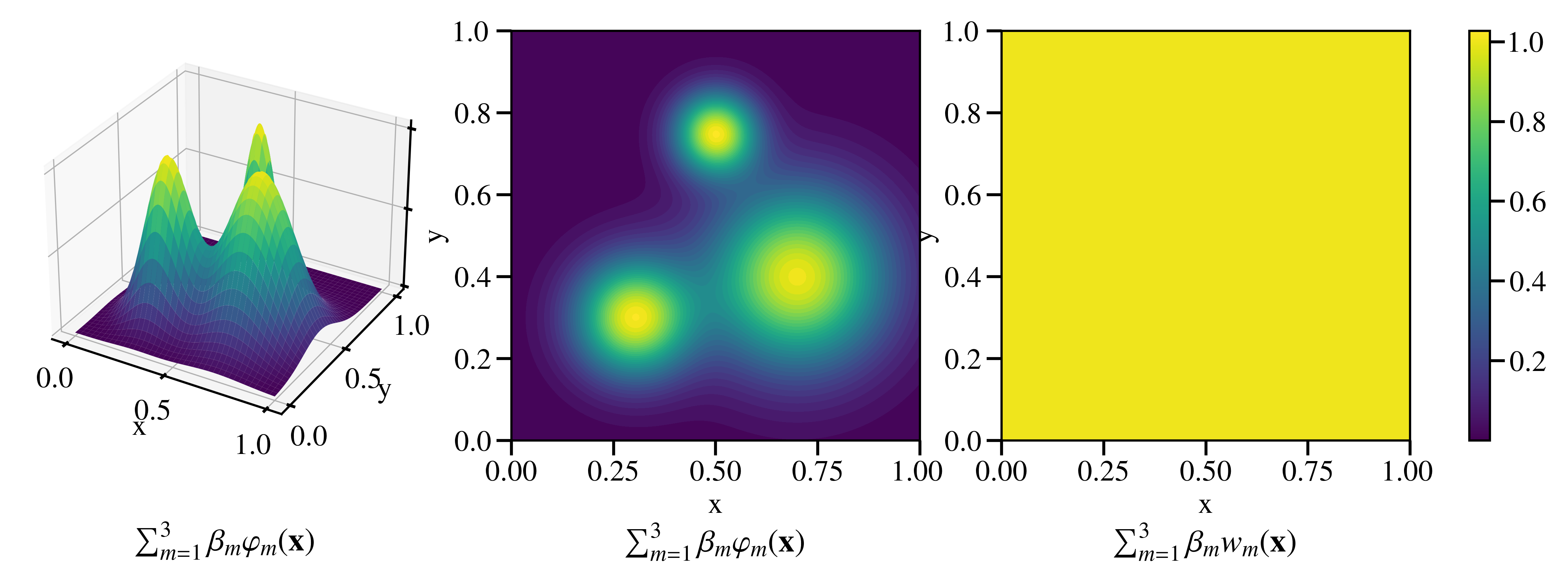}
  \caption{Comparison of RBF reconstructions:
(a)-(b) unnormalized superpositions $\sum_{m=1}^3 \beta_m \varphi_m(\bx)$ and
(c) Shepard-normalized reconstruction $\sum_{m=1}^3 \beta_m w_m(\bx)$. The normalization removes amplitude amplification in overlap regions and yields a smooth, well-balanced surface.}
  \label{fig:rbf_pu_blend}
\end{figure}

\subsection{Elastic Net Regression}
\label{sec:elastic_net}

On each subdomain $\omega_i$, the coefficient vector $\boldsymbol{\beta}^{(i)} \in \mathbb{R}^{M_i}$ determines the weights of the local RBF basis functions in the permeability surrogate. We employ Elastic Net regularization~\cite{zou2005elastic} to compute these coefficients, combining the sparsity-promoting $\ell_1$ penalty of the Lasso~\cite{tibshirani1996lasso} with the stability of ridge regression~\cite{hoerl1970ridge}.

The resulting optimization problem~\eqref{elastic} balances three components: a least-squares data fidelity term enforcing agreement with the sampled permeability values, an $\ell_1$ penalty promoting sparse selection of relevant basis functions, and an $\ell_2$ penalty improving numerical stability. The regularization parameters $\lambda_1,\lambda_2>0$ control the trade-off between sparsity and stability.

Since the objective function is convex, the problem can be solved efficiently using coordinate descent~\cite{friedman2010regularization}, which updates one coefficient at a time while keeping the others fixed. This approach is well suited for moderately large local dictionaries and guarantees convergence to a global minimizer~\cite{friedman2010regularization, wright2015coordinate}.

\begin{remark}
\label{rem:logK}
To enforce positivity of the permeability, all regressions are performed on the logarithm of the data. Since the permeability is given as elementwise constant values $K_T>0$, we define the transformed data $\widetilde{K}_T := \log(K_T)$,
and construct an approximation $\widetilde{K}^*(\bx)$ using the proposed RBF framework. The reconstructed permeability is then recovered as
$K^*(\bx) = \exp\!\big(\widetilde{K}^*(\bx)\big)$.
This transformation guarantees that $K^*(\bx)>0$ and prevents nonphysical negative values that may arise from the regression procedure.
\end{remark}
 
\begin{comment}
\begin{remark}
\label{rem:logK}
To enforce positivity of the permeability, all regressions are performed on the logarithm of the data. Since the permeability is given as element wise constant values $K_T>0$, we define the transformed data $\widetilde{K}_T := \log(K_T)$,
and construct an approximation $\widetilde{K}^*(\bx)$ using the proposed RBF framework. The reconstructed permeability is then recovered as
$K^*(\bx) = \exp\!\big(\widetilde{K}^*(\bx)\big)$.
This transformation guarantees that $K^*(\bx)>0$ and prevents nonphysical negative values that may arise from the regression procedure.
\end{remark}
\end{comment}

\subsection{Adaptive RBF Enrichment with Elastic Net}

We propose an adaptive radial basis function (RBF) framework, coupled with Elastic Net regression, to efficiently capture sharp transitions in heterogeneous permeability fields. On each subdomain $\omega_i$, the method starts from an initial set of RBF basis functions and iteratively enriches the representation in regions where the approximation error is large.

Let $\{\varphi_m^{(i)}\}_{m=1}^{M_i}$ denote the current set of local RBF basis functions on $\omega_i$, with centers $\{\bc_m^{(i)}\}_{m=1}^{M_i}$ and widths $\{\sigma_m^{(i)}\}_{m=1}^{M_i}$. At each adaptive iteration, the coefficients $\boldsymbol{\beta}^{(i)}$ are computed by solving the Elastic Net regression problem described in Section~\ref{sec:elastic_net}. Using the resulting approximation, we evaluate elementwise residual indicators $R_T$ (cf. Step~3) and identify a subset of elements $\mathcal{T}_i^{\mathrm{ref}} {\subset \mathcal{T}_h(\omega_i)}$ corresponding to the largest errors.

For each selected element $T \in \mathcal{T}_i^{\mathrm{ref}}$, the approximation space is locally enriched by introducing additional RBFs in a neighborhood of the associated center locations. Specifically, for each basis function centered near $T$, we define a local neighborhood $\mathcal{N}(\bc_m^{(i)})$ and introduce a finite set of new centers $\{\bc_m^{(i,q)}\}_{q=1}^{M_q} \subset \mathcal{N}(\bc_m^{(i)})$. The widths of the new bases are chosen as
$$
\sigma_m^{(i,q)} = \eta_m^{(q)} \, \sigma_m^{(i)}, \qquad \eta_m^{(q)} \in (0,1),
$$
yielding additional basis functions $\{\varphi_m^{(i,q)}(\bx)\}_{q=1}^{M_q}$ that provide increased resolution near $\bc_m^{(i)}$.

This enrichment increases the local basis density in regions of large approximation error while maintaining a compact representation elsewhere. After enrichment, the Elastic Net regression is solved again on the updated dictionary, resulting in an increased number of basis functions $M_i$. The adaptive process is repeated until a prescribed stopping criterion is satisfied, such as a residual tolerance or a maximum number of basis functions. The complete procedure is summarized in Algorithm~\ref{algo:adaptive}.

To illustrate the effect of adaptive enrichment, we consider a two-dimensional domain $\Omega = [0,1]^2$ with an initially uniform set of Gaussian RBFs. Selected basis functions are refined according to the above strategy, and the resulting
configurations are shown in Figure~\ref{fig:adaptivity_cases}.

\begin{figure}[!h]
    \centering
    \includegraphics[width=\linewidth]{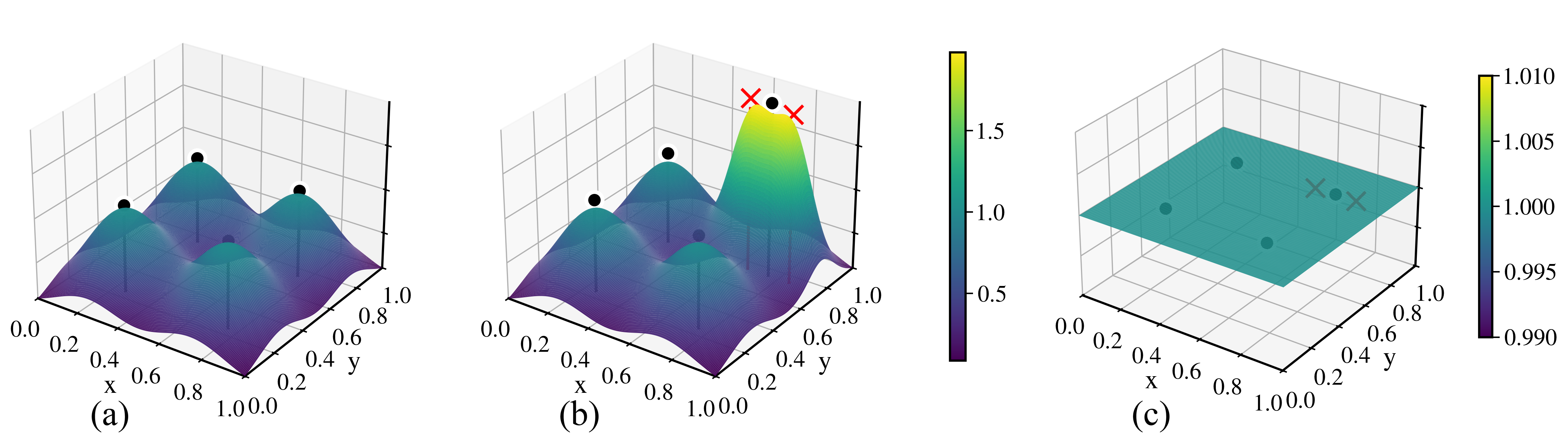}
    \caption{Adaptive refinement and Shepard normalization.Black filled circles denote existing RBF centers, while red crosses indicate newly added neighboring centers.
    (a) Unnormalized RBF superposition before refinement;
    (b) Unnormalized RBF superposition after adaptive refinement;
    (c) verification of the partition-of-unity property under Shepard normalization, showing $\sum_m w_m \approx 1$.
}

    \label{fig:adaptivity_cases}
\end{figure}

\begin{algorithm}[!h]
\small
\caption{Adaptive RBF Enrichment with Elastic Net Regression (local solve on each subdomain)}
\label{algo:adaptive}

\For{each subdomain $\omega_i \subset \Omega$, $i=1,\dots,N$}
{
\KwIn{
mesh partition $\mathcal{T}_h(\omega_i)$ with elementwise permeability values $\{K_{T_j}\}_{j=1}^{N_i}$, and corresponding data vector $\widetilde{\bK}_{\omega_i} = (\log K_{T_j})_{j=1}^{N_i}$\;
initial RBF dictionary $\mathcal{D}_{\omega_i}=\{(\bc_m^{(i)},\sigma_m^{(i)})\}_{m=1}^{M_i}$ with feature matrix $\Phi_{\omega_i}$\;
Elastic Net parameters $\lambda_1,\lambda_2$\;
number of marked cells per adaptive round $K_{\mathrm{top}}$\;
maximum number of added RBFs $M_{\max}$\;
stopping tolerance $\varepsilon_{\mathrm{tol}}$
}

\BlankLine
Initialize $M_i' \gets M_i$ and $\boldsymbol{\beta}^{(i)}\gets\boldsymbol{0}\in\mathbb{R}^{M_i'}$\;

\BlankLine
\Repeat{$\max_{T\in\mathcal{T}_h(\omega_i)} R_T < \varepsilon_{\mathrm{tol}}$ \textbf{or} $M_i' \ge M_i + M_{\max}$}
{
Initialize $\mathcal{B}_{\mathrm{new}}\gets\emptyset$\;

\BlankLine
\textbf{(1) Elastic Net regression.}\;
Solve
$$
\boldsymbol{\beta}^{(i)}
=
\arg\min_{\boldsymbol{\beta}\in\mathbb{R}^{M_i'}}
\frac12\|{\widetilde{\bK}_{\omega_i}}-\Phi_{\omega_i}\boldsymbol{\beta}\|_2^2
+\lambda_1\|\boldsymbol{\beta}\|_1
+\frac{\lambda_2}{2}\|\boldsymbol{\beta}\|_2^2.
$$
Define
$$
{\widetilde{K}^*_{\omega_i}}(\bx)=\sum_{m=1}^{M_i'} \beta_m^{(i)}\,\varphi_m^{(i)}(\bx),
\qquad \bx\in\omega_i,
$$
and recover
$$
{K^*_{\omega_i}(\bx)=\exp\!\bigl(\widetilde{K}^*_{\omega_i}(\bx)\bigr).}
$$

\BlankLine
\textbf{(2) Compute elementwise residual indicators.}\;
\For{each element $T\in\mathcal{T}_h(\omega_i)$}
{
Compute
$$
R_T
=
\sum_{\ell}
w_T^\ell
\bigl(K^*_{\omega_i}(\bx_T^\ell)-K_T\bigr)^2,
$$
where $\{\bx_T^\ell,w_T^\ell\}$ are quadrature points and weights on $T$.\;
}

\BlankLine
\textbf{(3) Mark elements for refinement.}\;
Let $\mathcal{E}_{\mathrm{worst}}\subset\mathcal{T}_h(\omega_i)$ denote the set of the $K_{\mathrm{top}}$ elements with the largest indicators $R_T$\;
If $\mathcal{E}_{\mathrm{worst}}=\emptyset$, \textbf{break}\;

\BlankLine
\textbf{(4) Adaptive enrichment.}\;
\For{each $T\in\mathcal{E}_{\mathrm{worst}}$}
{
Add new RBFs according to the refinement rule in Section~3.4 and append them to $\mathcal{B}_{\mathrm{new}}$\;
}
{If $\mathcal{B}_{\mathrm{new}}=\emptyset$, \textbf{break}\;}

\BlankLine
\textbf{(5) Update dictionary.}\;
Update $\mathcal{D}_{\omega_i} \gets \mathcal{D}_{\omega_i}\cup \mathcal{B}_{\mathrm{new}}$\;
Update $M_i' \gets M_i' + |\mathcal{B}_{\mathrm{new}}|$\;
Recompute the feature matrix $\Phi_{\omega_i}$ for the updated dictionary and extend $\boldsymbol{\beta}^{(i)}$ by zeros\;
}

\KwOut{Final coefficient vector $\boldsymbol{\beta}^{(i)}$ and refined RBF dictionary $\mathcal{D}_{\omega_i}$\;}
}
\end{algorithm}

\clearpage
\newpage

\subsection{One-Dimensional Examples: Uniform and Adaptive RBF Approximation}

Here, we provide several examples to illustrate the details of our approach. First, we present the general performance of RBFs in capturing sharp transitions in a one-dimensional setup. Since this example is one-dimensional, we write the spatial variable as
$x\in\Omega\subset\mathbb{R}$ instead of the vector notation $\bx$.

\subsubsection{Example 1. Baseline RBF approximation of a step function}
Here, we consider the one-dimensional domain $\Omega = [0,0.03125]$ where the permeability is defined as a step function:
\begin{equation}
   K(x) \;=\;
   \begin{cases}
      10^{-4}, & x < x_0, \\[0.3em]
      10^{-1}, & x \ge x_0,
   \end{cases}
   \qquad x \in \Omega,
\label{eqn:2.4_step}
\end{equation}
with a sharp transition at $x_0 = 0.015625$. The function $K(x)$ is approximated using two Gaussian RBFs, one centered in each half of the domain, with common width $\sigma=0.0019$. Coefficients are computed via Elastic Net regression. 

Figure~\ref{fig:base_combined} shows that the unnormalized RBF approximation is unbalanced and produces overshoot near the jump, whereas Shepard normalization enforces $\sum_m w_m(x)=1$ and yields a stable, bounded transition. While normalization improves robustness, two bases are insufficient to resolve the discontinuity; additional RBFs are introduced next to improve accuracy.

\begin{figure}[!h]
    \centering
    \begin{subfigure}[b]{0.48\textwidth}
        \centering
        \includegraphics[width=\linewidth]{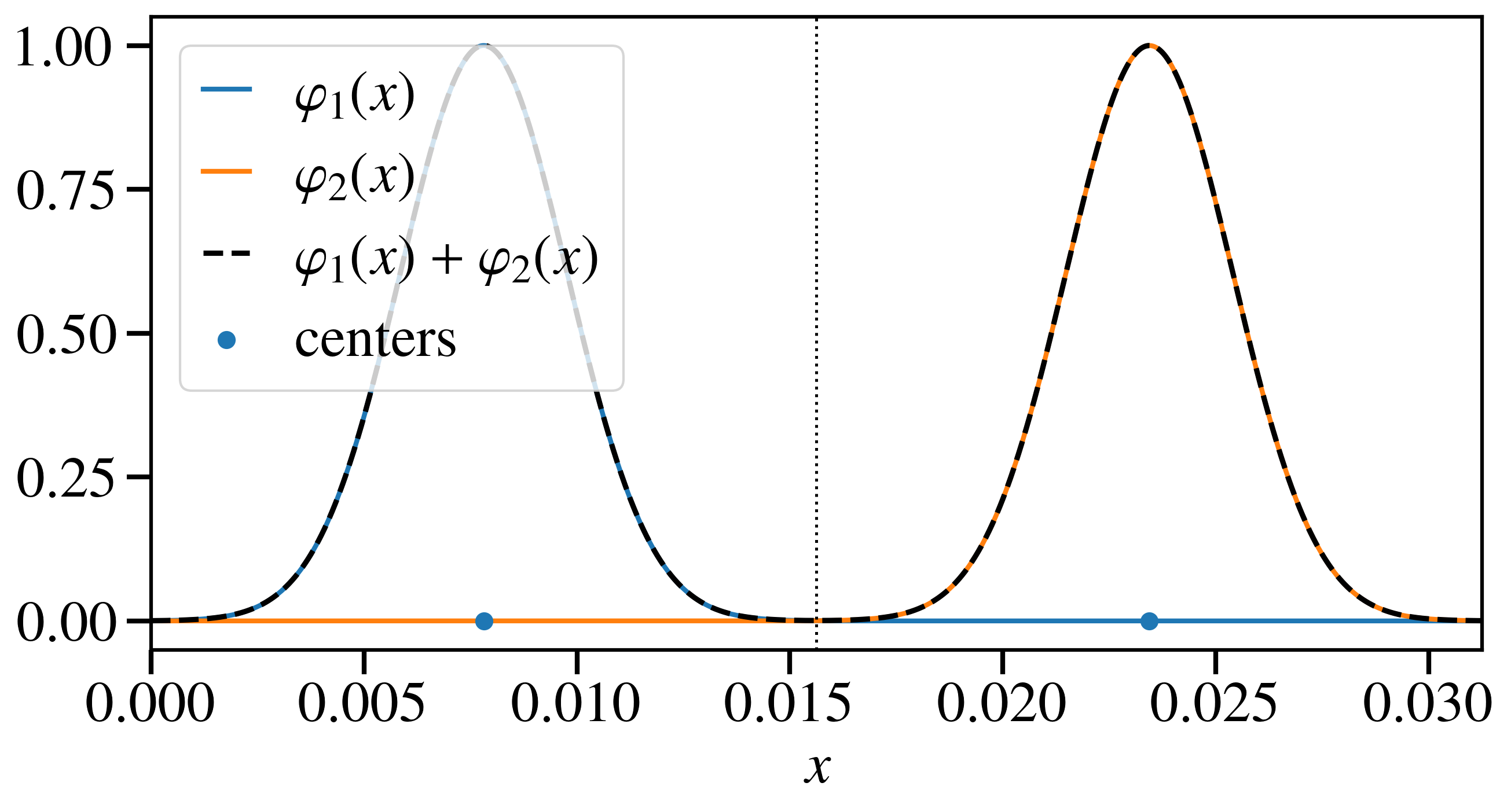}
        \caption{}
        \label{fig:base_phi_raw}
    \end{subfigure}
    \hfill
    \begin{subfigure}[b]{0.48\textwidth}
        \centering
        \includegraphics[width=\linewidth]{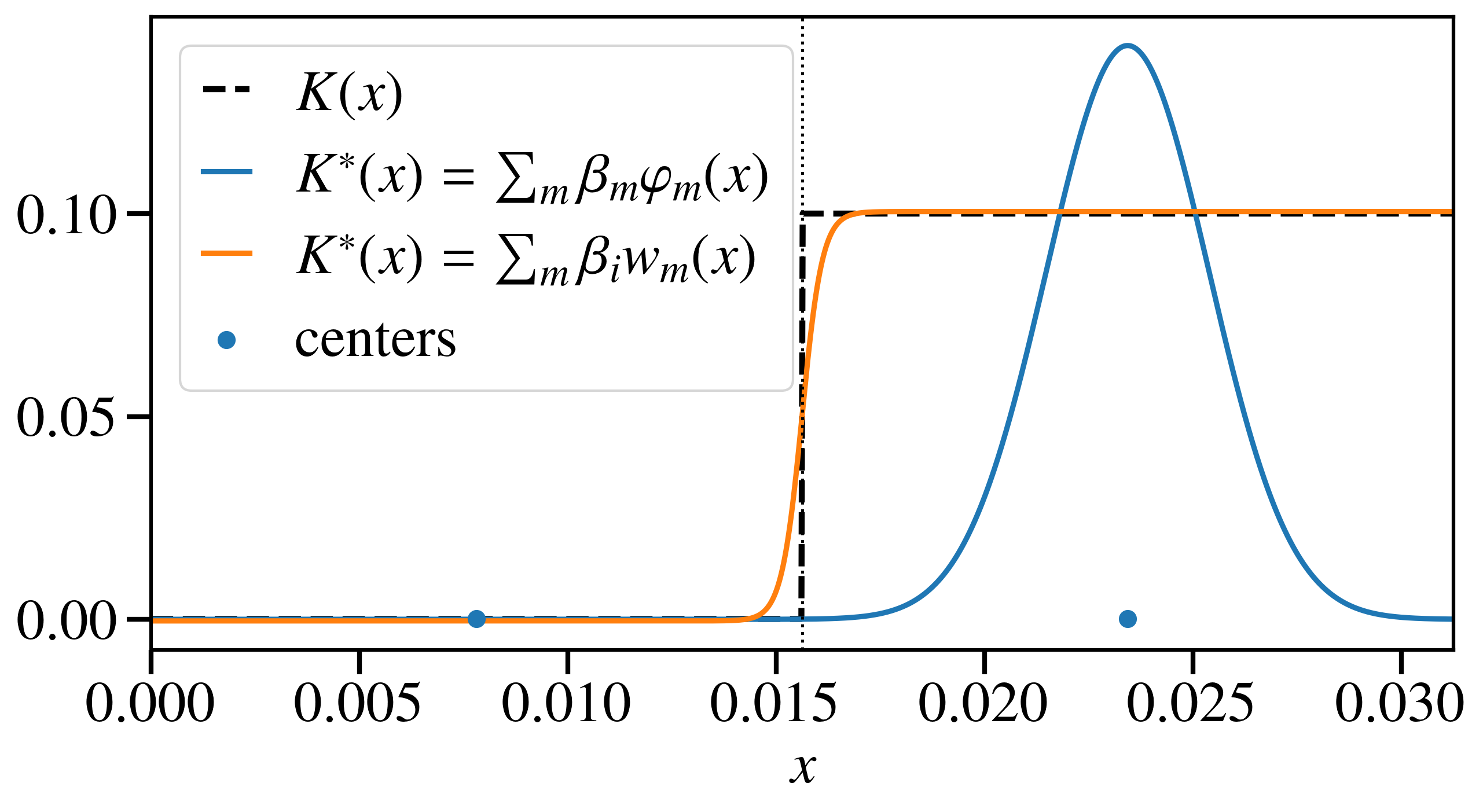}
        \caption{}
        \label{fig:base_rbf_comparison}
    \end{subfigure}
    \caption{Example 1. (a) Unnormalized Gaussian basis functions $\varphi_1(x)$ and $\varphi_2(x)$ and their sum $\varphi_1+\varphi_2$.
    (b)  Comparison of the unnormalized and Shepard-normalized reconstructions of $K(x)$.}
    %Corresponding $K(x)$ reconstructions using raw and Shepard-normalized RBFs.}
    \label{fig:base_combined}
\end{figure}

\subsubsection{Example 2. Uniform refinement with Shepard-normalized RBFs}

Using the same function $K(x)$ defined in \eqref{eqn:2.4_step}, we investigate the effect of uniformly increasing the number of Shepard-normalized RBFs on the reconstruction accuracy. The function $K(x)$ is approximated with different numbers of Shepard-normalized RBFs using the normalized construction described in Section~\ref{subsec:shepard}, where the coefficients $\beta_m \in \mathbb{R}$ are obtained by solving the Elastic Net regression problem described in equation~\eqref{elastic}. The regularization parameters are set to $\lambda_1 = 4.59 \times 10^{-4}$ and $\lambda_2 = 4.64 \times 10^{-6}$. We consider uniform center placement with M=2,4,8, and 16 and examine three kernel widths: $\sigma$, $\sigma/2$, and $\sigma/4$, where $\sigma=0.0019$. 

Figure{\color{red}~}\ref{fig:elasticnet_rbf} shows the resulting Shepard-normalized RBF reconstructions with $\sigma=0.0019$. As M increases, the approximation converges rapidly in the smooth regions away from the interface. However, near the jump at $x_0=0.015625$, uniform refinement produces spurious oscillations on both sides of the jump.    
Figure~\ref{fig:rbf_convergence_uniform_vs_adaptive} presents the convergence behavior of the RBF approximation for the function $K(x)$ under uniform refinement. Results are shown for several fixed kernel widths $\sigma$. The relative $L^2(\Omega)$ error,

\begin{equation}\label{eq:L2error}
\frac{\|K-K^*\|_{L^2(\Omega)}}{\|K\|_{L^2(\Omega)}}
=
\frac{\left(\sum_{T\in\mathcal{T}_h}\int_T \bigl(K_T-K^*(\bx)\bigr)^2\,d\bx\right)^{1/2}}
{\left(\sum_{T\in\mathcal{T}_h}\int_T K_T^2\,d\bx\right)^{1/2}}.
\end{equation}
exhibits oscillatory behavior across all tested values of $\sigma$ and remains in the range $10^{-3}$ to $10^{-2}$.

\begin{figure}[!h]
    \centering

    \begin{subfigure}[!h]{0.48\textwidth}
        \centering
        \includegraphics[width=\linewidth]{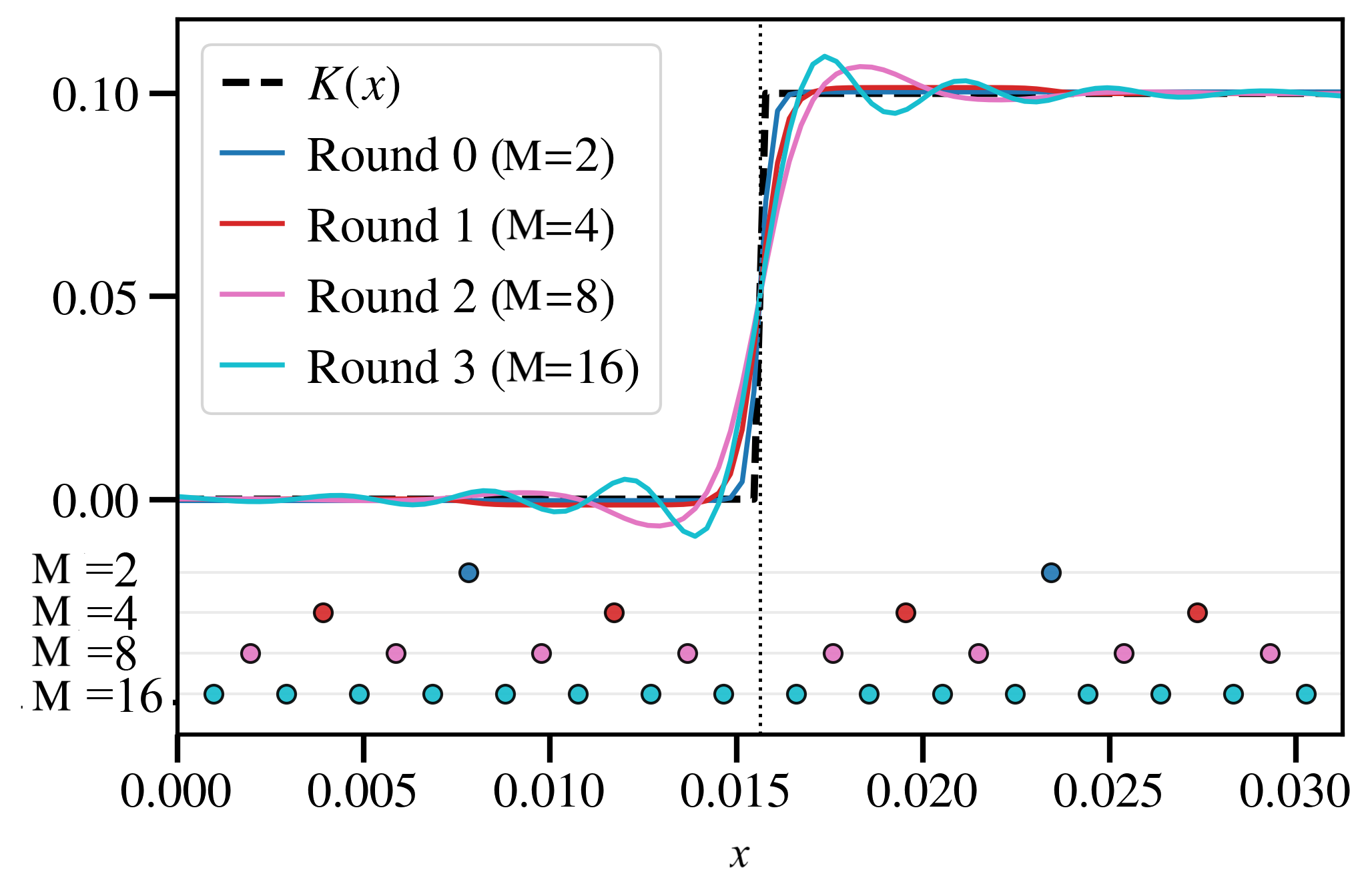}
        \caption{}
        \label{fig:elasticnet_rbf}
    \end{subfigure}
    \hfill
    \begin{subfigure}[!h]{0.48\textwidth}
        \centering
        \includegraphics[width=\linewidth]{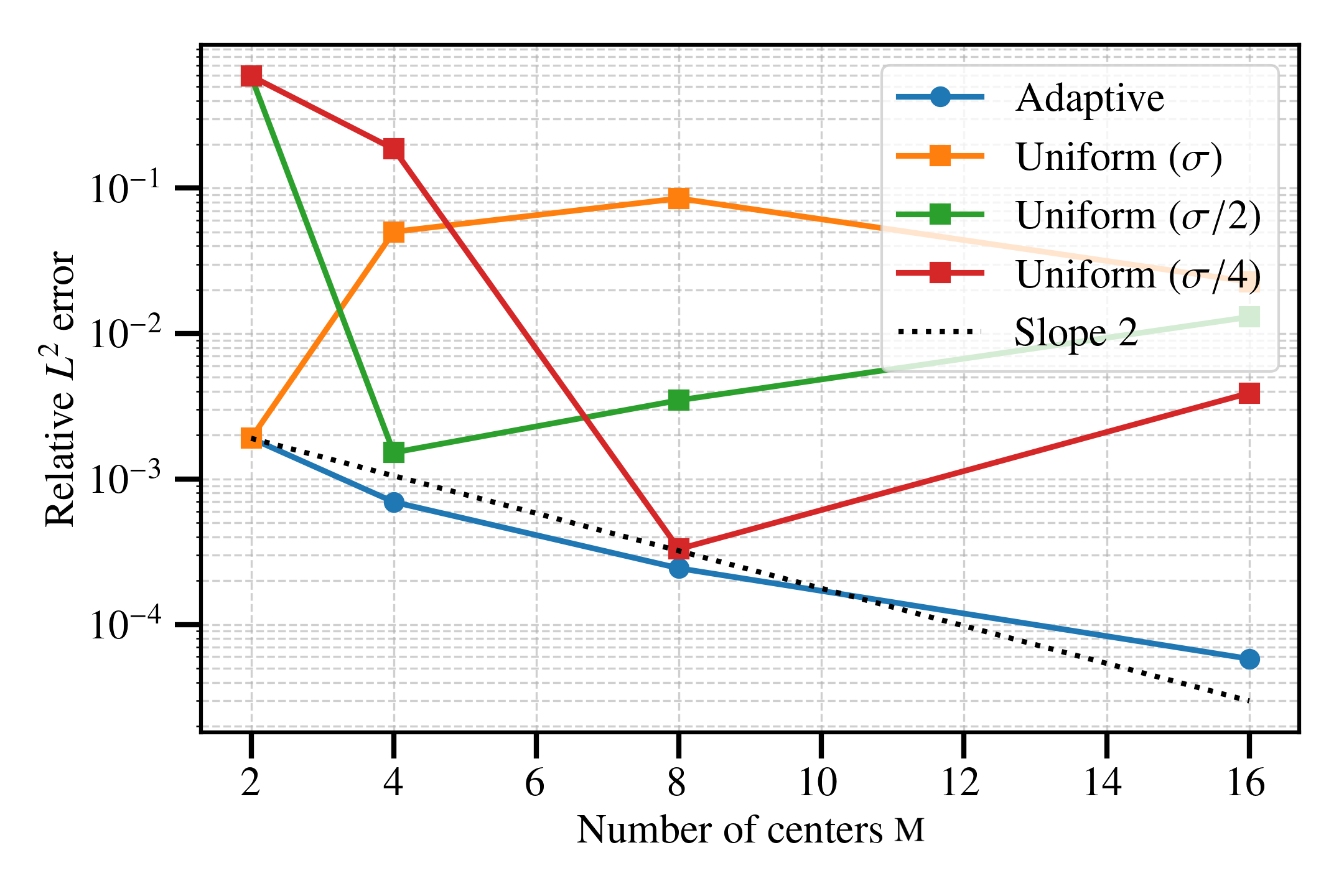}
        \caption{}
        \label{fig:rbf_convergence_uniform_vs_adaptive}
    \end{subfigure}

    \caption{Example 2. (a) Elastic Net approximation of the step function using M=2,4,8, and 16 RBFs. Circles indicate RBF centers. (b) Convergence of the relative $L^2(\Omega)$ error under uniform and adaptive RBF refinement. Uniform results are shown for multiple kernel widths $\sigma$.}
    \label{fig:step_uniform}
\end{figure}

\subsubsection{Example 3. Adaptive RBF refinement}

To address the oscillatory behavior observed under uniform refinement in the previous example, we now apply the proposed adaptive Elastic Net RBF strategy described in Algorithm~\ref{algo:adaptive}. Rather than uniformly enriching the basis across the domain, the adaptive procedure introduces new RBFs only in regions where the residual is large, with correspondingly reduced kernel widths $\sigma$. This targeted enrichment progressively sharpens the local resolution near the jump while keeping the total number of basis functions small.

Figure~\ref{fig:one_d_adaptive} illustrates the adaptive refinement process for the step function. As the refinement proceeds, additional RBF centers are concentrated near the jump at $x_0=0.015625$, leading to a sharp and stable reconstruction of the discontinuity.

\begin{figure}[!h]
    \centering

    \hfill
    \begin{subfigure}[!h]{0.48\textwidth}
        \centering
        \includegraphics[width=\linewidth]{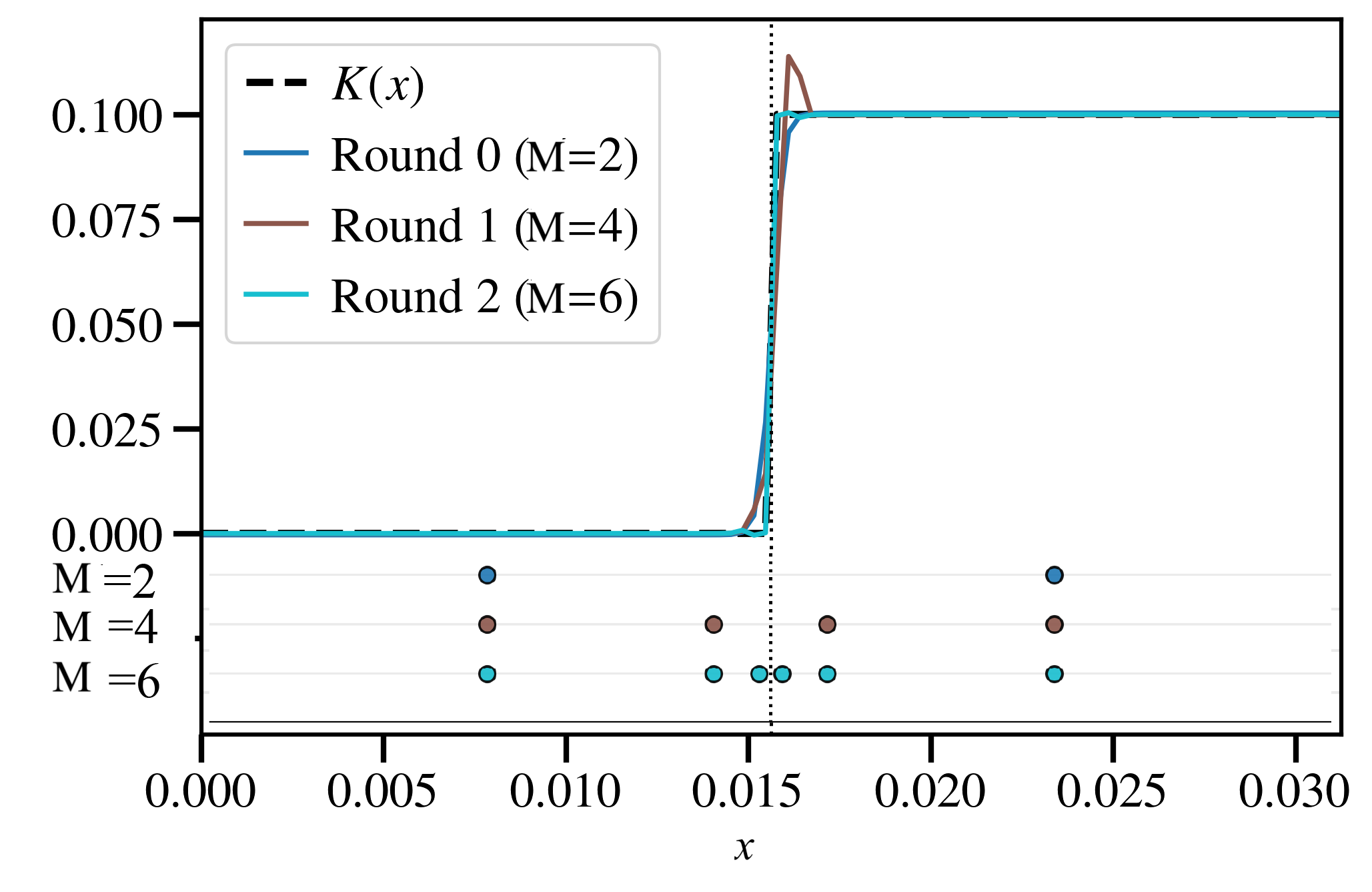}
        \caption{}
        \label{fig:one_d_adaptive}
    \end{subfigure}
    \hfill
    \begin{subfigure}[!h]{0.51\textwidth}
        \centering
        \includegraphics[width=\linewidth]{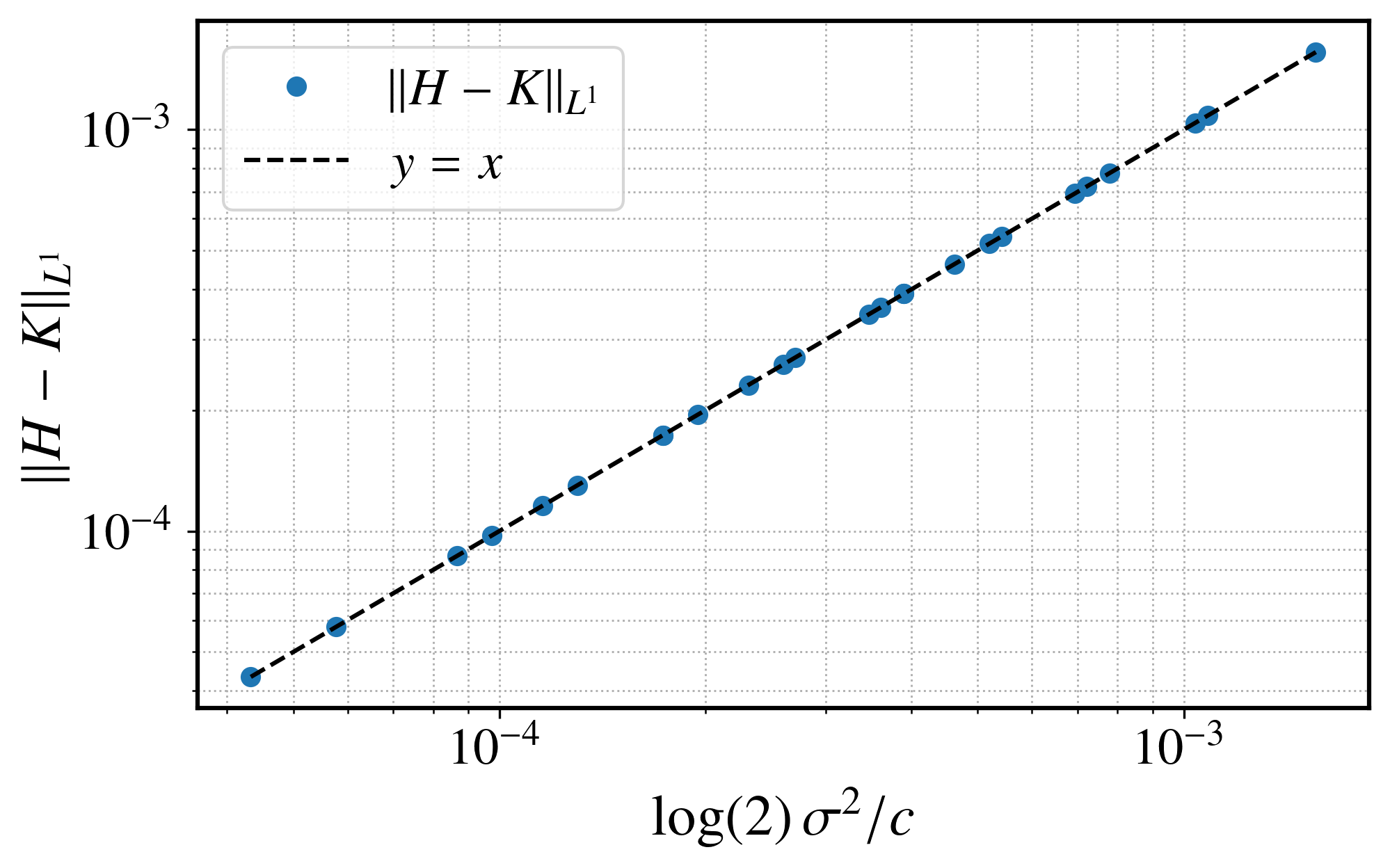}
        \caption{}
        \label{fig:L1_verification}
    \end{subfigure}

    \caption{Example 3.   (a) Adaptive Elastic Net RBF approximation of $K(x)$. (b) Numerical verification of the result
$\|H-K\|_{L^1(\mathbb{R})} = (\log 2)\sigma^2/c$.
Each marker corresponds to a different pair $(c,\sigma)$.
The dashed line $y=x$ represents perfect agreement.}
    \label{fig:step_uniform_vs_adaptive}
\end{figure}

The effectiveness of this adaptive strategy is also reflected in the convergence behavior. As shown earlier in the convergence plot (\ref{fig:rbf_convergence_uniform_vs_adaptive}), the adaptive refinement exhibits systematic error reduction as the number of centers increases, achieving accurate resolution of the jump with only a small number of RBFs. In this example, the sharp jump is well captured using just six centers, demonstrating that adaptive localization of basis functions is essential for efficient and stable approximation of $K(x)$.

\section{Analysis of RBF Approximation}

In this section we analyze the ability of radial basis functions (RBFs) with Shepard normalization to approximate discontinuous functions. In particular, we study the approximation of a Heaviside step function,
which represents a sharp interface across a hyperplane. We show that a two-center Gaussian RBF with Shepard normalization produces a logistic-type approximation whose $L^1$ error across the interface can be made arbitrarily small.

The analytical result is first derived in a general $d$-dimensional setting. Since the approximation depends only on the normal coordinate to the interface, the error reduces to a one-dimensional integral. We then verify this analytical formula numerically in the 1D setting.

\subsection{Step Function Approximation by RBFs with Shepard Normalization}

\begin{theorem}
    Let $H$ be the Heaviside step function on $\mathbb{R}$ using the half-maximum convention, i.e., $H(0) = 0.5$.
    For $\bv \in \mathbb{R}^d$ with $\|\bv\|=1$ and $b \in \mathbb{R}$, 
    let $H_d(\bx):= H(\langle \bv,\bx \rangle + b)$ be the $d$-dimensional Heaviside step function along the hyperplane $\{\bx \ | \ \langle \bv,\bx\rangle + b=0\}$.
    For any $\epsilon>0$, there exists a RBF with the Shepard normalization satisfying
    $$
    \int_{\mathbb R}|H(y)-K^*(y)|dy<\epsilon,
    $$
    where $y=\langle\bv,\bx\rangle+b$ and
$K^*(y)=K^*(\bx)$. This quantity represents the approximation error in the normal direction to the hyperplane $\{\bx:\langle\bv,\bx\rangle+b=0\}$, i.e., the error per unit measure along the interface.

    %$$
    %\int_{\mathbb{R}} |H_d(y\bv) - K^*(y\bv)| dy < \epsilon,
    %$$
    %which represents the error per unit length on the hyperplane being less than $\epsilon$.
    % Furthermore, while $K(0)=H(0)$ holds, the absolute error near the origin follows 
    % $|H(x)-K(x)| = 0.5 + O(|x|)$ as $|x| \to 0$.
\end{theorem}
\begin{proof}
    Consider the RBF with the Shepard normalization of the form:
    \begin{align*}
        K^*(\bx) = \beta_1 \frac{\exp(-\frac{\|\bx-\bc_1\|^2}{2\sigma_1^2})}{\exp(-\frac{\|\bx-\bc_1\|^2}{2\sigma_1^2})+\exp(-\frac{\|\bx-\bc_2\|^2}{2\sigma_2^2})}
        + \beta_2 \frac{\exp(-\frac{\|\bx-\bc_2\|^2}{2\sigma_2^2})}{\exp(-\frac{\|\bx-\bc_1\|^2}{2\sigma_1^2})+\exp(-\frac{\|\bx-\bc_2\|^2}{2\sigma_2^2})}.
    \end{align*}
    Let $\sigma=\sigma_1=\sigma_2>0$, $\bc_1=c\bv$, $\bc_2 = \gamma c\bv$
    where $\gamma$ is to be determined and $c > 0$.
    % Without loss of generality, let $b \ge 0$.
    Note that $\bc_1-\bc_2 = (1-\gamma)c\bv$ and
    $
    \|\bc_1\|^2-\|\bc_2\|^2 = (1-\gamma^2)\|c\bv\|^2.
    $
    % If $1+b/\|\bv\|^2 > 0$, 
    By letting $\gamma = -1-\frac{2b}{c\|\bv\|^2}$, we have
    \begin{align*}
        \frac{\exp(-\frac{\|\bx-\bc_1\|^2}{2\sigma_1^2})}{\exp(-\frac{\|\bx-\bc_1\|^2}{2\sigma_1^2})+\exp(-\frac{\|\bx-\bc_2\|^2}{2\sigma_2^2})}
        % = \frac{1}{1 + \exp(-\frac{2}{\sigma^2}\langle \bx, \bc \rangle)}.
        &= \frac{1}{1 + \exp(\frac{-2\langle \bx, \bc_1-\bc_2 \rangle + \|\bc_1\|^2-\|\bc_2\|^2}{2\sigma^2})}
        = \frac{1}{1 + \exp(\frac{-2(1-\gamma)\langle \bx, c\bv \rangle + (1-\gamma^2)\|c\bv\|^2}{2\sigma^2})}
        \\
        &=\frac{1}{1 + \exp(\frac{-2(b+c)}{\sigma^2}\left[\langle \bx, \bv \rangle +b\right])}
    \end{align*}
    Choose $c$ such that $c + b > 0$.
    By letting $\beta_1 = 1$ and $\beta_2 = 0$, 
    we have $K^*(\bx) = \frac{1}{1 + \exp(-\frac{2(b+c)}{\sigma^2}\left[\langle \bx, \bv \rangle +b\right])}$.
    % {\color{red}Sign inconsistency, we either change the sign of c or flip the coefficient $\beta$} 
    It can be checked that for any $\bx$ such that $\langle \bv, \bx \rangle + b\ne 0$, 
    $|H_d(\bx)-K^*(\bx)|= \frac{1}{1 + \exp(\frac{2(b+c)}{\sigma^2}|\langle \bx, \bv \rangle +b|)}$, which gives
    the error per unit length on the hyperplane of
    \begin{equation}\label{eq:analytic_formula}
    \int_{\mathbb{R}} |H(y) - K^*(y)| dy := \log 2 \cdot \frac{\sigma^2}{c+b},
    \end{equation}
    where $K^*(y) = \frac{1}{1 + \exp(-\frac{2(b+c)}{\sigma^2}y)}$.
    By letting $\frac{\sigma^2}{c+b}$ sufficiently small, the desired result is obtained.
    % Note that the antiderivative of $|H(x)-K(x)|$ is given by 
    % $E(x) = -\frac{1}{k}\log(1+\exp(-kx))$ where $k= \frac{2c}{\sigma^2}$.
    % Hence,
    % \begin{align*}
    %     \int_{-r}^r |H(x)-K(x)|dx 
    %     % &= 2\int_0^r |H(x)-K(x)|dx = 2(E(r)-E(0)) \\
    %     % &= \frac{2}{k}\left(\log 2 - \log(1+\exp(-kr)) \right) \\
    %     &= \log \frac{2}{1+\exp(-\frac{2c}{\sigma^2}r)} \cdot \frac{\sigma^2}{c}.
    % \end{align*}
    % The remaining part is readily followed by the Taylor series expansion of the absolut error at $x=0$.
\end{proof}

\subsection{Numerical Verification}
Since both $H_d(\bx)$ and $K^*(\bx)$ depend only on the scalar normal coordinate $y := \langle \bv,\bx\rangle + b,$ the multidimensional approximation reduces to a one-dimensional profile. Defining
$$
H(y) := H_d(\bx), \qquad K(y) := K^*(\bx),
$$
the error reduces to the integral
\begin{equation}\label{error}
\|H-K\|_{L^1(\mathbb R)}
=
\int_{\mathbb R} |H(y)-K(y)|\,dy,
\end{equation}
where $K(y) = \frac{1}{1+\exp\!\left(-\frac{2(b+c)}{\sigma^2}y\right)}.$ From the analytical derivation we obtain 
\begin{equation}\label{eq:1d_error}
    \|H(y)-K(y)\|_{L^1(\mathbb R)}
=
(\log 2)\frac{\sigma^2}{c+b}.
\end{equation}
To verify this formula numerically, we compute the integral $\int_{\mathbb R}|H(y)-K(y)|dy$ for several choices of
$(c,\sigma)$ and compare the numerical values with the theoretical
prediction.

Figure~\ref{fig:L1_verification} plots the numerical error against the
analytical value $(\log 2)\sigma^2/(c+b)$. Each point corresponds to a
different parameter pair $(c,\sigma)$ and we set $b=0$. The dashed line $y=x$ indicates
perfect agreement between theory and computation.

%To validate the analytical result in equation (\ref{eq:analytic_formula}), we compute $L^1$ error $\|H-K\|_{L^1(\mathbb{R})}$ numerically for multiple choices of
%$(c,\sigma)$ and compare it with the theoretical value
%$(\log 2)\sigma^2/c$.

%Figure~\ref{fig:L1_verification} plots the numerically computed error
%against the analytical prediction. Each point corresponds to a distinct
%pair $(c,\sigma)$. The dashed line represents $y=x$, indicating perfect
%agreement. The numerical and theoretical values coincide to approximately
%$10^{-9}$ relative accuracy.
% verification plot

\section{Numerical Experiments}

This section presents numerical experiments to evaluate the accuracy,
stability, and computational efficiency of the proposed PAM framework
for reconstructing heterogeneous permeability fields. All experiments are implemented in Python. The steady-state Darcy problem is solved using the FEniCS finite element library~\cite{logg2012automated, alnaes2015fenics}, while the permeability reconstruction and post-processing are performed using standard scientific computing libraries.

Although the proposed PAM framework is general and applicable to other
PDEs with heterogeneous coefficients, we use the Darcy model as a
representative example in the numerical experiments. Let
$\Omega \subset \mathbb{R}^2$ be a bounded domain whose boundary is decomposed into disjoint Dirichlet and Neumann parts
$\partial\Omega = \Gamma_D \cup \Gamma_N, ~ \Gamma_D \cap \Gamma_N = \emptyset.
$ The Darcy problem is given by
\begin{subequations}\label{eq:darcy_model}
\begin{align}
-\nabla \cdot (K(\bx)\nabla p) &= f && \text{in } \Omega, \\
p &= p_D && \text{on } \Gamma_D, \\
K(\bx)\nabla p \cdot n &= q && \text{on } \Gamma_N .
\end{align}
\end{subequations}
Here, $p := p(\bx): \Omega \to \mathbb{R}$ denotes the pressure,
$f$ is a source/sink term, and $K(\bx)$ is the discontinuous
piecewise-constant permeability field.

The weak formulation of \eqref{eq:darcy_model} reads as follows:
find $p \in H^1(\Omega)$ with $p = p_D$ on $\Gamma_D$ such that
\begin{equation}
\int_{\Omega} K(\bx)\nabla p \cdot \nabla v \, d\bx
=
\int_{\Omega} f v \, d\bx
+
\int_{\Gamma_N} q v \, ds
\qquad \forall v \in V_0,
\end{equation}
where
$$
V_0 = \{ v \in H^1(\Omega) : v = 0 \text{ on } \Gamma_D \}.
$$

In the numerical simulations, the pressure equation is discretized
using the continuous Galerkin finite element method with first-order
basis functions. Although it is possible to extend the formulation to
discontinuous Galerkin~\cite{riviere2008discontinuous} or enriched Galerkin finite element methods~\cite{lee2016locally} to
ensure local mass conservation~\cite{scovazzi2017analytical,arbogast2002implementation,dawson2004compatible}, here we restrict our discussion to the
continuous Galerkin method for simplicity of presentation.
The reconstructed permeability field
$K^*(\bx)$ obtained from the PAM method is then used in place of the
original piecewise-constant permeability field.

We consider two representative permeability fields on
$\Omega=[0,1]^2$, discretized using a $32\times32$ uniform mesh of
elements, as shown in Figure~\ref{fig:case_of_K}. These fields are
treated as fixed test cases throughout the section and are referred to
as \emph{Case~1} and \emph{Case~2}, respectively.

The experiments are organized as follows. We first study the effect of
different initial RBF dictionaries under uniform refinement. Next, we
apply the adaptive enrichment strategy to localize basis functions near
regions of rapid variation and quantify the resulting improvements in
accuracy. We then examine the scalability of the approach using
parallel execution. Finally, we evaluate the impact of the recovered
continuous permeability field on the solution of the Darcy equation.

\begin{figure}[!h]
    \centering
    \includegraphics[width=0.42\linewidth]{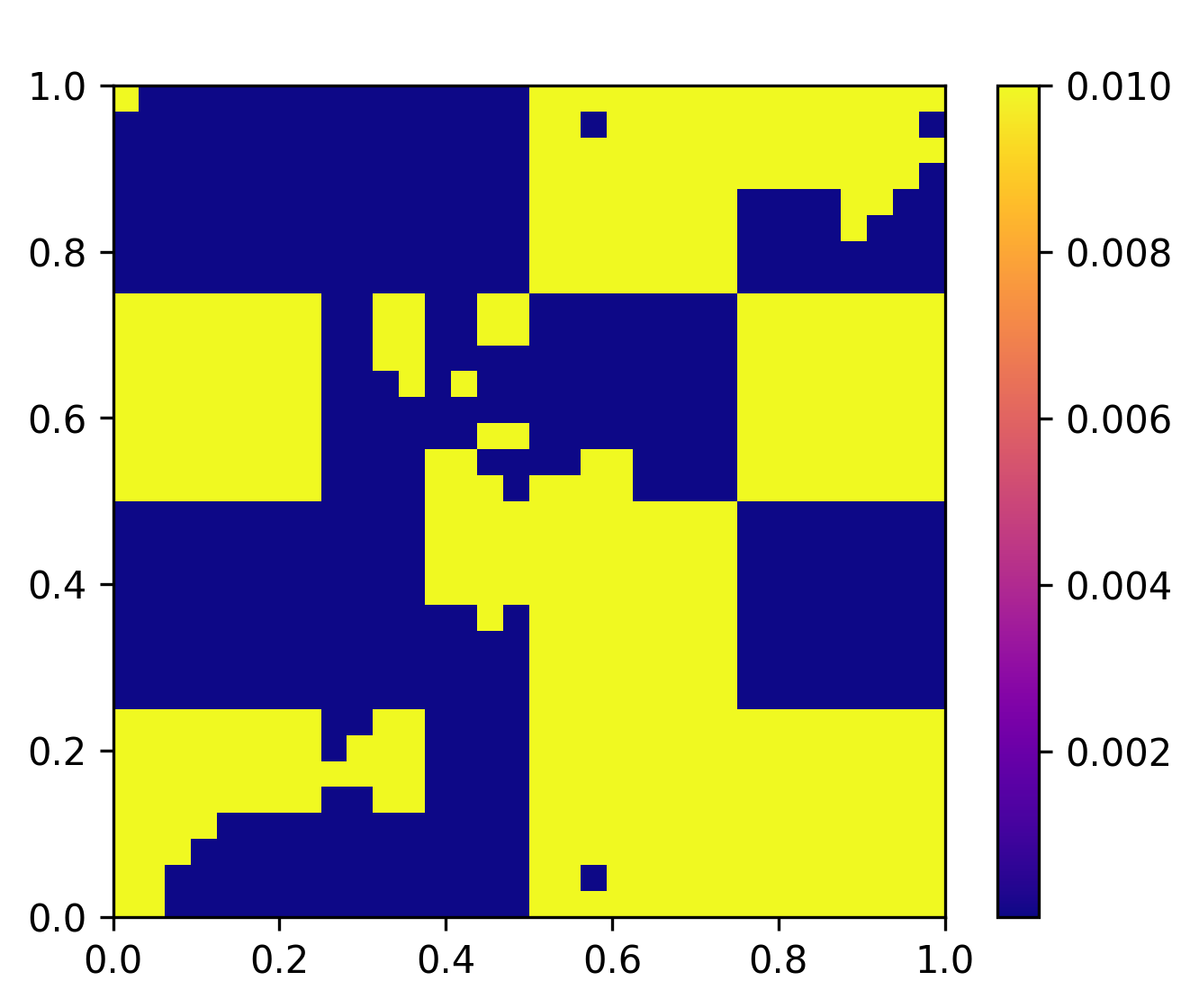}
    \includegraphics[width=0.43\linewidth]{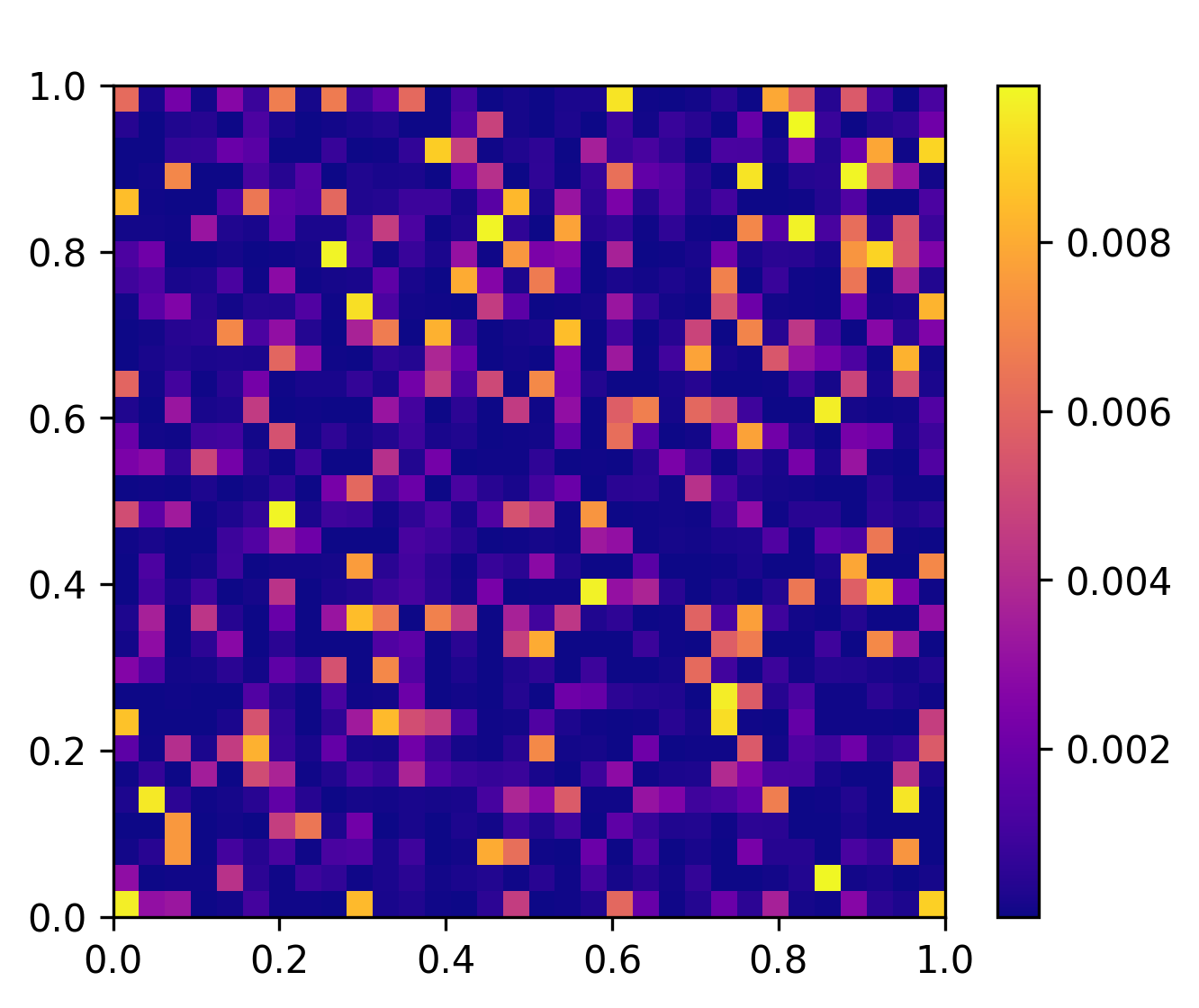}
    \caption{
    Representative two-dimensional permeability fields used in the numerical
    experiments:
    (left) Case~1 permeability field;
    (right) Case~2 permeability field.
    }
    \label{fig:case_of_K}
\end{figure}

\subsection{Experiment 1. Effect of the initial RBF dictionary under uniform refinement}
\label{sec:case1_dictionary}
First, we study how the initial RBF dictionary affects permeability reconstruction for
the two test fields introduced in Figure~\ref{fig:case_of_K} (Case~1 and Case~2).
In this experiment, we use a single subdomain ($N=1$, $\omega_1=\Omega$) on
$\Omega=[0,1]^2$. We construct an initial dictionary of Shepard-normalized Gaussian RBFs whose
centers are placed on a uniform $g\times g$ lattice over $\Omega$ with corresponding width $\sigma$. We vary the resolution parameter $g$ (and hence $\sigma$) to quantify the
accuracy-cost trade-off under uniform refinement. For each $g$, the Elastic Net parameters $(\lambda_1,\lambda_2)$ are held fixed across both permeability test cases to ensure a fair comparison.

Table~\ref{tab:recovery_results_combined} reports the $L^2(\Omega)$ error (see {\eqref{eq:L2error}}) and the CPU time. As $g$ increases, accuracy improves systematically for both test fields, while
runtime increases with the dictionary size ($g^2$ basis functions). Case~1 exhibits faster error decay under uniform refinement than Case~2, reflecting the additional difficulty posed by sharp interfaces and {high-contrast} regions.
Overall, the Shepard-normalized Gaussian dictionary provides a stable baseline for both permeability patterns and serves as a reference for the adaptive
enrichment strategy studied next.

\begin{table}[!h]
\centering
\begin{tabular}{|c|c|c|c|c|}
\hline
\multirow{2}{*}{{Dictionary $(g,\sigma)$}} &
\multicolumn{2}{c|}{{\ $\|K-K^*\|_{L^2(\Omega)}$}} &
\multicolumn{2}{c|}{{CPU time (s)}} \\ \cline{2-5}
 & {Case 1} & {Case 2} & {Case 1} & {Case 2}\\ \hline
$(64,\,0.0078)$  & $8.57 \times10^{-5}$ & $7.30 \times10^{-4}$  & 56 & 56  \\ \hline
$(32,\,0.031)$   & $1.55 \times10^{-3}$ & $1.92 \times10^{-3}$  & 14 & 20  \\ \hline
$(16,\,0.0625)$  & $2.28 \times10^{-3}$ & $2.18 \times10^{-3}$  & 5  & 5   \\ \hline
\end{tabular}
\caption{Experiment 1. Reconstruction accuracy and runtime for the two test permeability
fields (Case~1 and Case~2) under uniform RBF refinement. Increasing $g$ (and
reducing $\sigma$) improves accuracy at increased computational cost.}
\label{tab:recovery_results_combined}
\end{table}

\subsubsection{Experiment 1.1: Adaptive RBF refinement}
\label{sec:case2_adaptive}

We now evaluate the proposed \emph{adaptive RBF refinement strategy} described in
Algorithm~\ref{algo:adaptive}. The goal is to assess whether residual-driven
enrichment can achieve the same or better reconstruction accuracy than uniform
refinement while using substantially fewer basis functions.

As in the previous experiment, we consider a single subdomain ($\omega_1=\Omega$) on
$\Omega=[0,1]^2$, discretized by a $32\times32$ mesh of elements. The adaptive
procedure is initialized with a $32\times32$ uniform RBF dictionary
($1024$ centers) and proceeds by iteratively enriching the basis only in regions
where the local residual indicator is largest.

At each adaptive round, elementwise residual indicators are computed, and a fixed number of high-error elements are marked for
refinement. New Gaussian RBFs with reduced width ($\eta_q=1/2$) are then introduced locally
according to the refinement rule specified in
Section~\ref{sec:adaptive_rbf}. In all adaptive experiments, a fixed number of $K_{\mathrm{top}}=204$ elements
(corresponding to approximately $20\%$ of the $32\times32$ mesh) are marked for refinement at each adaptive round. The RBF coefficients are updated by solving the
Elastic Net problem after each enrichment step.

Table~\ref{tab:adaptive_summary} summarizes the refinement history for both test
permeability fields. For comparison, the uniform refinement strategy in Case~1 required a dense
$64\times64$ dictionary ($4096$ centers) to achieve $L^2$ errors of
approximately $8.6\times10^{-5}$ (Case~1 permeability) and
$7.3\times10^{-4}$ (Case~2 permeability). In contrast, the adaptive method reaches
comparable or better accuracy using significantly fewer basis functions.

In particular, for the Case~1 permeability, the adaptive scheme attains a $L^2$ error of $5.6\times10^{-5}$ after three refinement rounds using $2860$ centers{, approximately $30\%$ fewer than} the uniform configuration.
For the Case~2 permeability, five adaptive rounds yield an error of
$1.9\times10^{-5}$ with $4084$ centers, outperforming the uniform approach in accuracy.

These results demonstrate that residual-guided adaptive enrichment effectively
allocates basis functions to spatially important regions such as sharp
interfaces and localized high-contrast zones, allowing the proposed method to
achieve high accuracy.

\begin{table}[!h]
\centering
\begin{tabular}{|c|c|c|c|}
\hline
{Field Type} & 
{Adaptive Round} & 
{\# Centers} & 
{$\|K-K^*\|_{L^2(\Omega)}$}  \\ \hline
Case 1 & Round 0 & 1024 & $1.55 \times 10^{-3}$  \\ 
       & Round 1 & 1636 & $6.66 \times10^{-4}$  \\ 
       & Round 2 & 2248 & $1.11 \times10^{-4}$  \\
       & Round 3 & 2860 & $5.56 \times10^{-5}$  \\
        \hline \hline
Case 2 & Round 0 & 1024 & $1.92 \times 10^{-3}$  \\ 
       & Round 1 & 1636 & $7.17 \times10^{-4}$  \\ 
       & Round 2 & 2248 & $2.97 \times10^{-4}$  \\
       & Round 3 & 2860 & $1.31 \times10^{-4}$  \\
       & Round 4 & 3472 & $4.82 \times10^{-5}$  \\
       & Round 5 & 4084 & $1.94 \times10^{-5}$  \\
       \hline
\end{tabular}
\caption{Experiment 1.1. Adaptive refinement history for the two test permeability fields (Case~1 and Case~2). At each adaptive round, $K_{\mathrm{top}}=204$ elements are marked for refinement and locally enriched with additional RBFs.}

\label{tab:adaptive_summary}
\end{table}

Figure~\ref{fig:Kstar_and_centers_two_row} illustrates the evolution of the
adaptive reconstruction for the Case~2 permeability field.
The top row shows the recovered continuous permeability $K^*(\bx)$ at selected
adaptive rounds, while the bottom row displays the corresponding distribution of
RBF centers.

\begin{figure}[!h]
\centering

% ---------------- Row 1: K* predictions ----------------
\begin{subfigure}[!h]{0.32\textwidth}
  \centering
  \includegraphics[width=\linewidth]{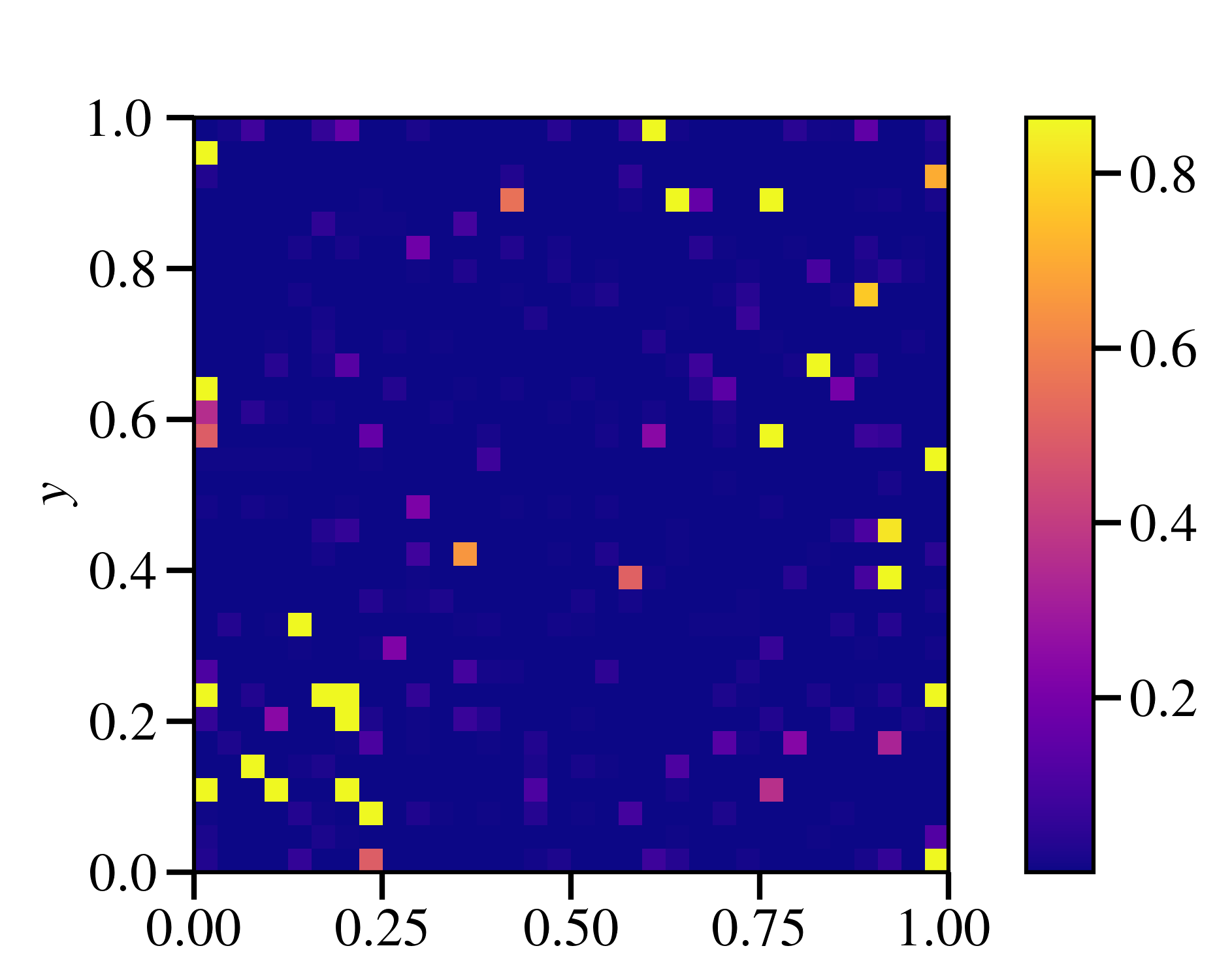}
  \caption{$K^*(\bx)$ (round 0)}
\end{subfigure}\hfill
\begin{subfigure}[!h]{0.32\textwidth}
  \centering
  \includegraphics[width=\linewidth]{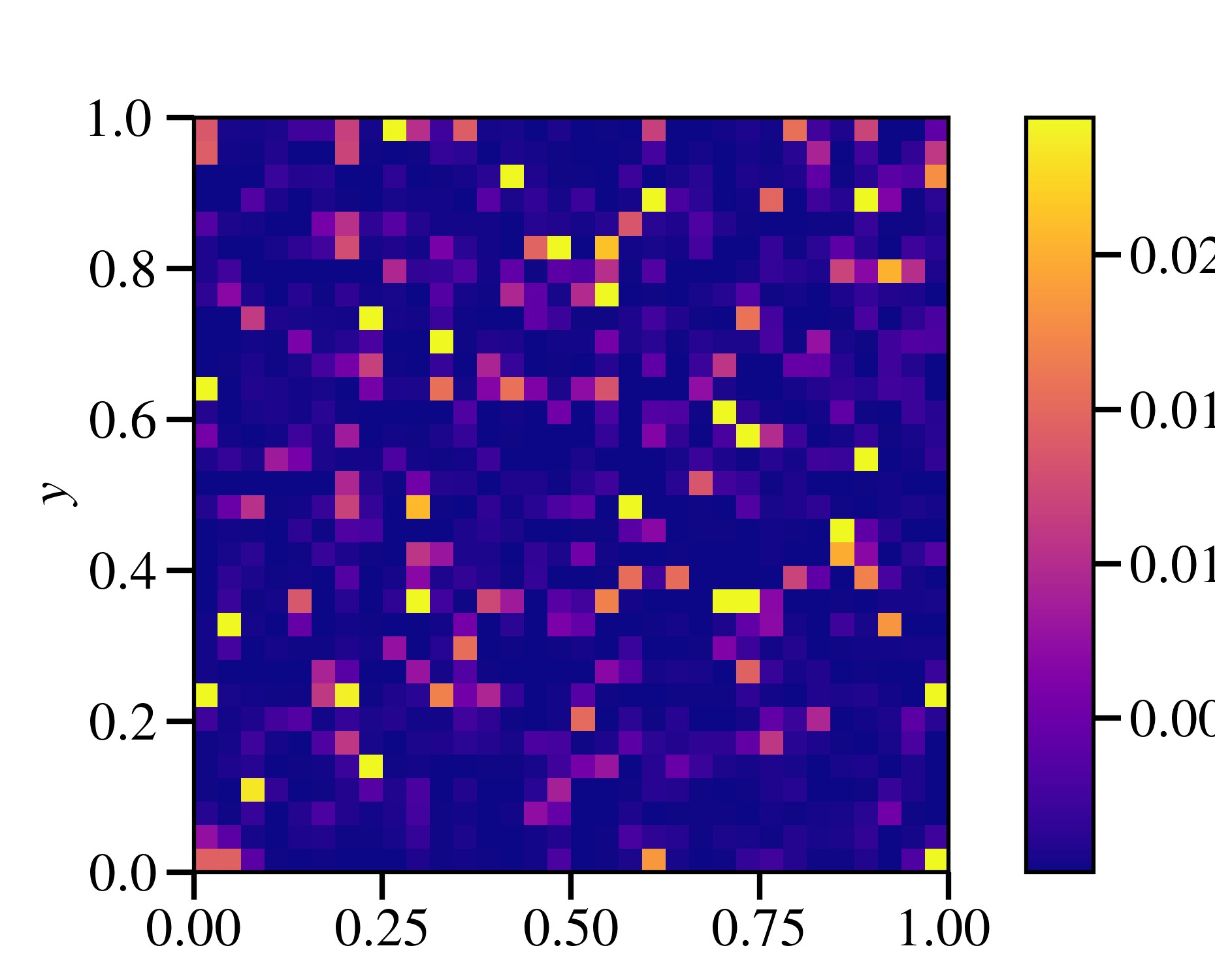}
  \caption{$K^*(\bx)$ (round 3)}
\end{subfigure}\hfill
\begin{subfigure}[!h]{0.32\textwidth}
  \centering
  \includegraphics[width=\linewidth]{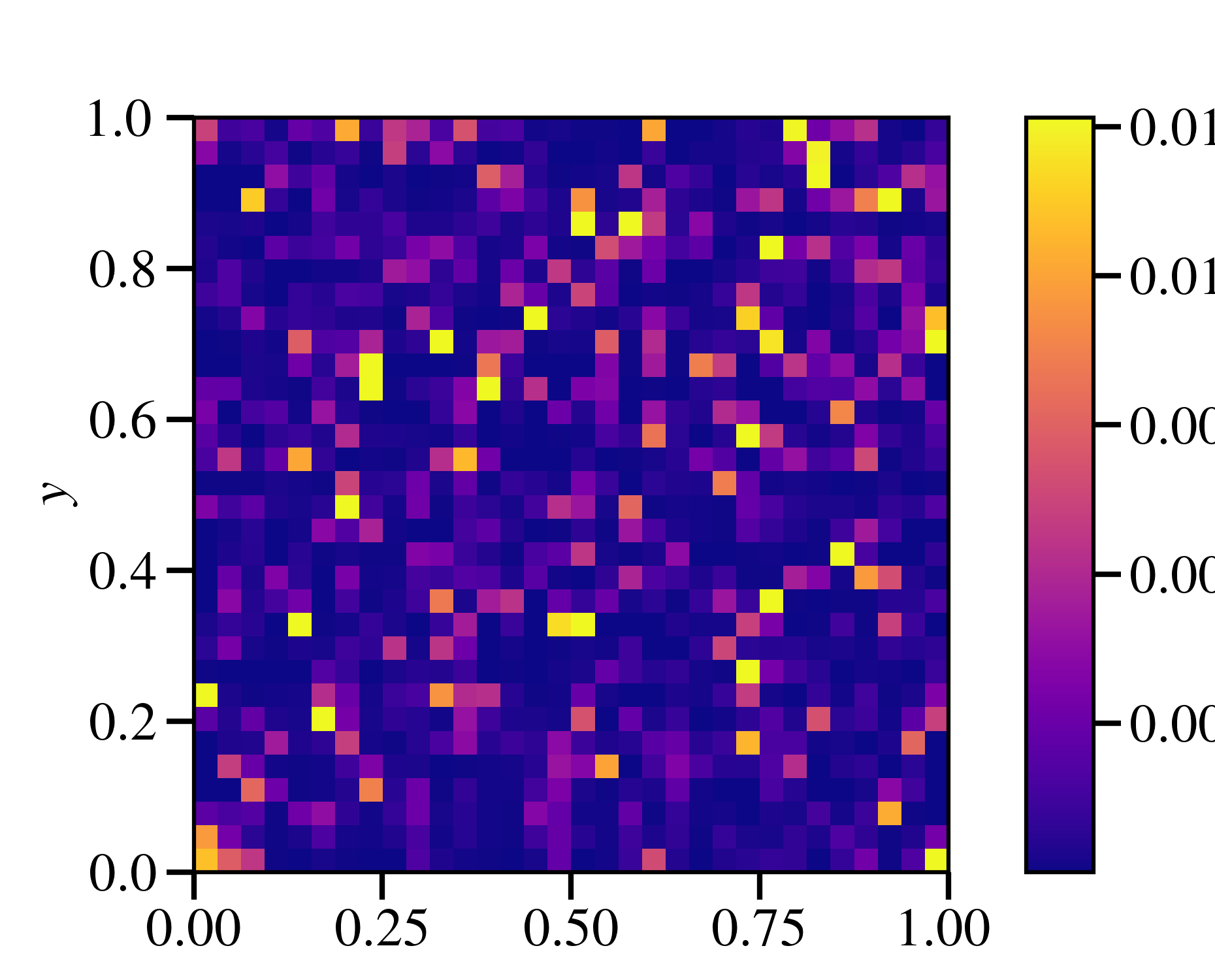}
  \caption{$K^*(\bx)$ (round 5)}
\end{subfigure}

%\vspace{0.35em}

% ---------------- Row 2: center locations ----------------
\begin{subfigure}[!h]{0.32\textwidth}
  \centering
  \includegraphics[width=\linewidth]{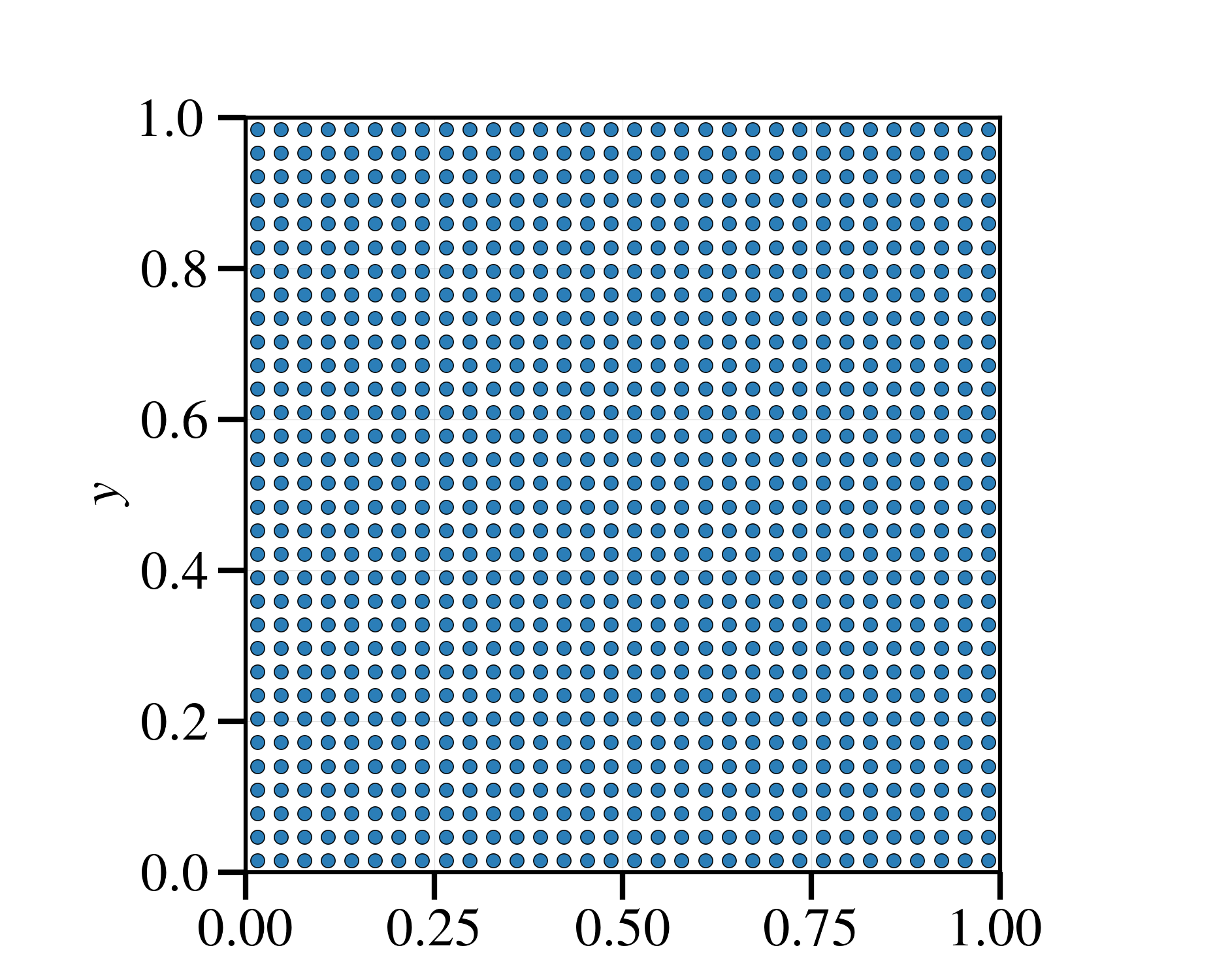}
  \caption{Centers (round 0)}
\end{subfigure}\hfill
\begin{subfigure}[!h]{0.32\textwidth}
  \centering
  \includegraphics[width=\linewidth]{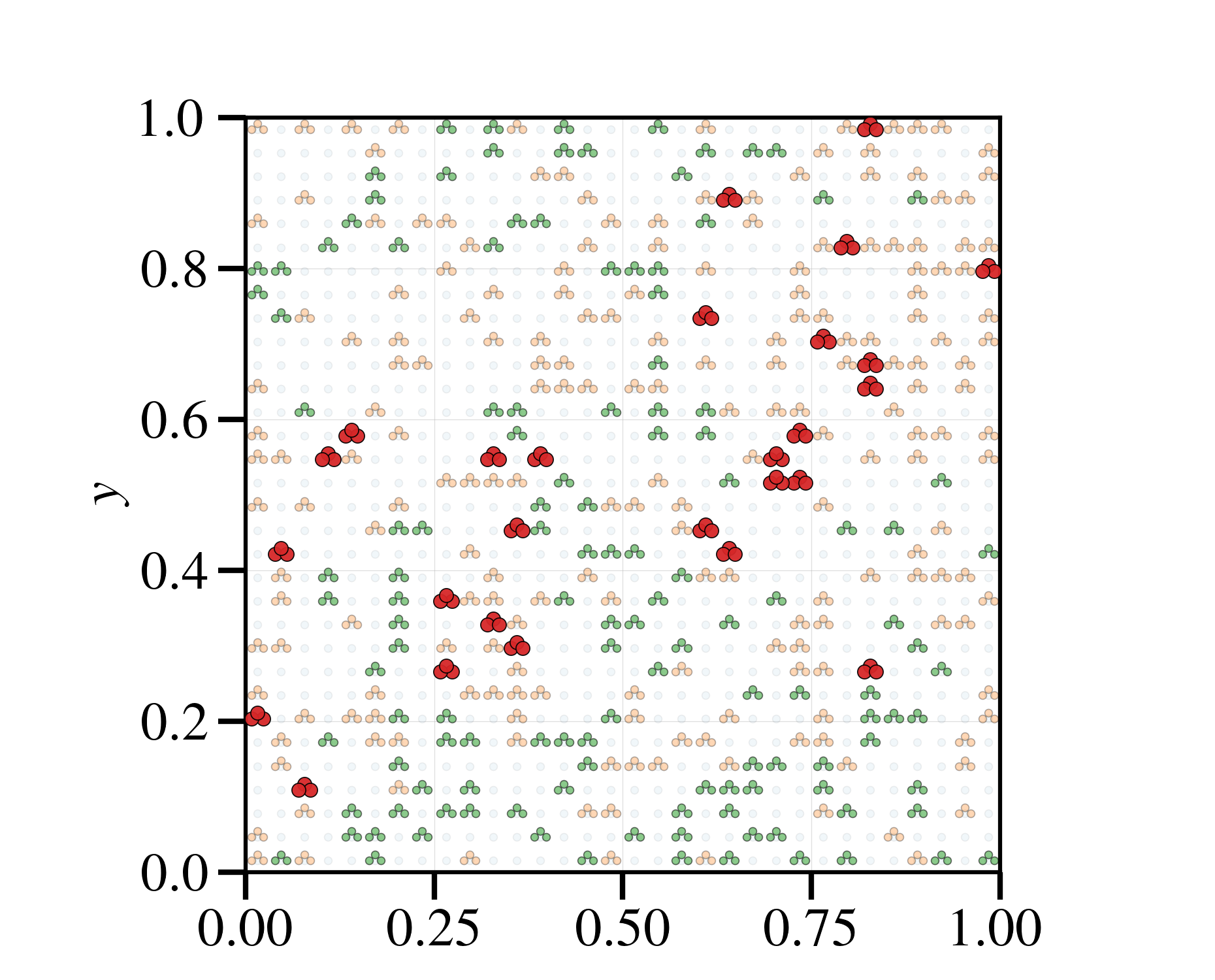}
  \caption{Centers (round 3)}
\end{subfigure}\hfill
\begin{subfigure}[!h]{0.32\textwidth}
  \centering
  \includegraphics[width=\linewidth]{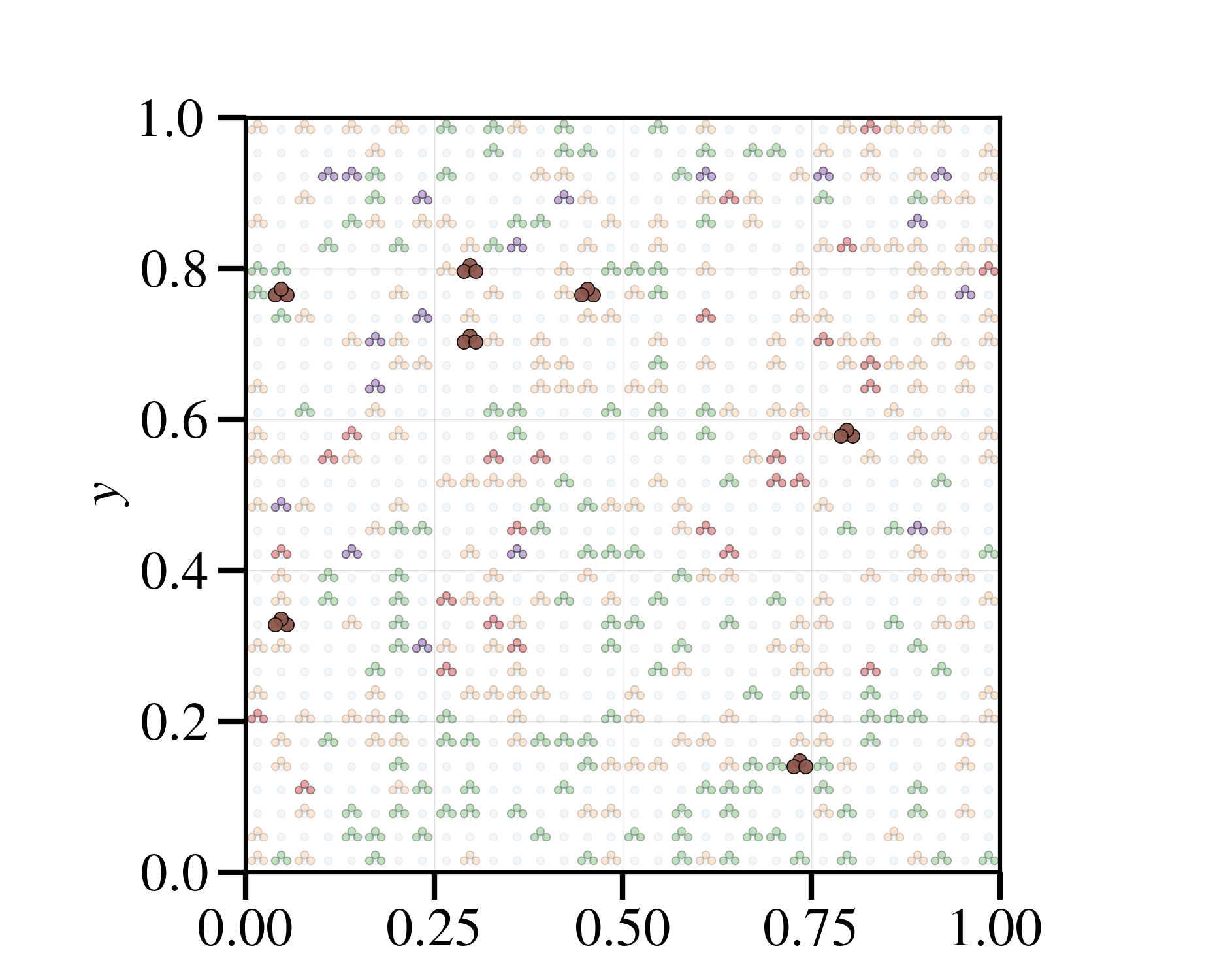}
  \caption{Centers (round 5)}
\end{subfigure}

\caption{Experiment  1.1. Recovered permeability $K^*(\bx)$ (top row) and the corresponding evolution of RBF center locations (bottom row). In the center plots, previously selected centers are shown with lighter markers, while newly added centers in each accepted round are highlighted.}
\label{fig:Kstar_and_centers_two_row}
\end{figure}

\subsubsection{Experiment 1.2: Parallel recovery via domain decomposition}

In this experiment, we examine the parallel performance of the proposed framework by decomposing a fixed computational domain into multiple nonoverlapping subdomains and performing permeability recovery independently on each subdomain.

We consider the same domain $\Omega$, and partition it into $N_s$ nonoverlapping subdomains.
On each $\omega_i$, the local permeability surrogate $K^*_{\omega_i}$ is
constructed independently using Algorithm~\ref{algo:adaptive}. The global reconstructed field $K^*$ is obtained by assembling the $K^*_{\omega_i}$ across all subdomains. Each subdomain problem involves its own local data distribution and RBF dictionary, leading to locally defined regression matrices. As a result, the Elastic Net parameters $(\lambda_1,\lambda_2)$ are selected independently for each $\omega_i$.

To avoid overlap and redundant computation, we adopt a half-open domain decomposition in which each element belongs to exactly one subdomain. This strategy ensures that no data are duplicated across subdomains and enables
embarrassingly parallel execution with minimal inter-process communication.

Table~\ref{tab:recovery_results_combined_parallel} reports the recovery accuracy and CPU time for different domain decompositions. The single domain ($1\times1$) configuration serves as a global baseline. When the domain is decomposed into $2\times2$ subdomains, each local problem
becomes smaller and better conditioned, resulting in a substantial reduction in
computational time. Despite the independent local solves, the global reconstruction accuracy remains
comparable to the baseline case, confirming that the half-open decomposition preserves accuracy. Figure~\ref{fig:kstar_parallel_tiles} visualizes the reconstructed piecewise continuous
permeability field $K^*(\bx)$ obtained using different domain decompositions.

\begin{table}[!h]
\centering
\begin{tabular}{|c|c|c|c|c|c|}
\hline
\multirow{2}{*}{{Subdomains}} & 
\multirow{2}{*}{{Centers in $\omega_i$}} & 
\multicolumn{2}{c|}{$\|K-K^*\|_{L^2(\Omega)}$} & 
\multicolumn{2}{c|}{{CPU (s)}} \\ \cline{3-6}
 & & {Case 1} & {Case 2} & {Case 1} & {Case 2} \\ \hline
$1 \times 1$ & 1024 & $1.55 \times 10^{-3}$  & $1.92 \times 10^{-3}$ & 14 & 20   \\ \hline
$2 \times 2$ & 256 & $1.42 \times 10^{-3}$ & $1.88 \times 10^{-3}$ & 4 & 5 \\ \hline
$2 \times 2$ & 1024 & $1.34 \times 10^{-5}$    & $4.48\times 10^{-4}$ & 12  & 17 \\ \hline
\end{tabular}
\caption{Experiment  1.2. Recovery accuracy for Perlin and Heterogeneous permeability fields using the half-open parallel domain decomposition. 
The first configuration ($1\times1$) provides the global baseline, while the $2\times2$ configuration demonstrates comparable accuracy with a fourfold speed-up due to parallel computation. 
Differences in local $\lambda$ values across subdomains arise from geometric and data density variations, not from numerical instability.}
\label{tab:recovery_results_combined_parallel}
\end{table}

Figure~\ref{fig:kstar_parallel_tiles} visualizes the parallel recovery of the
reconstructed piecewise continuous permeability field $K^*(\bx)$ using two and four
subdomains.

\begin{figure}[!h]
    \centering
    \begin{minipage}{0.48\textwidth}
        \centering
        \includegraphics[width=\textwidth]{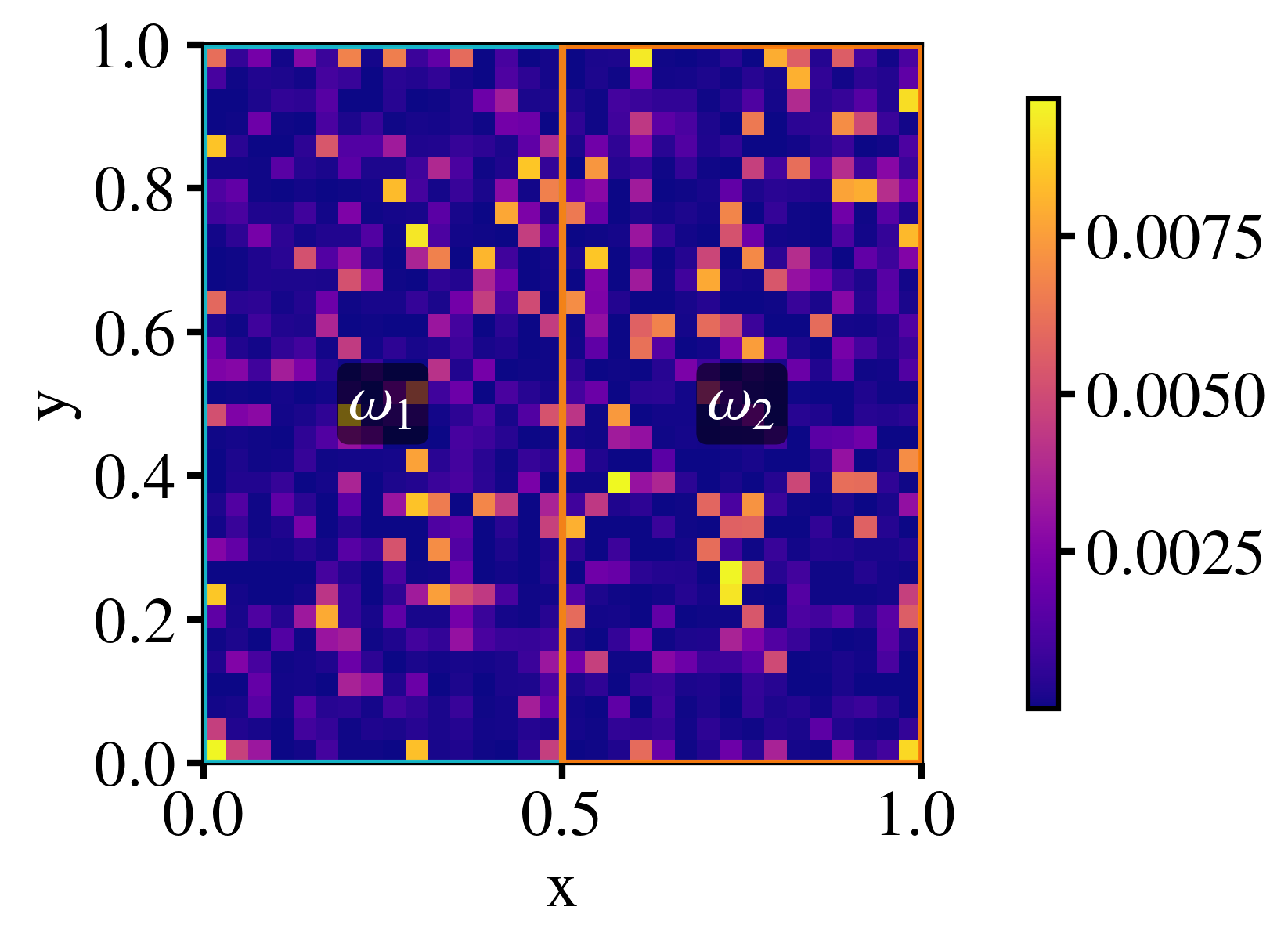}
        \subcaption{$2\times 1$ domain decomposition}
        \label{fig:kstar_parallel_1x2}
    \end{minipage}\hfill
    \begin{minipage}{0.48\textwidth}
        \centering
        \includegraphics[width=\textwidth]{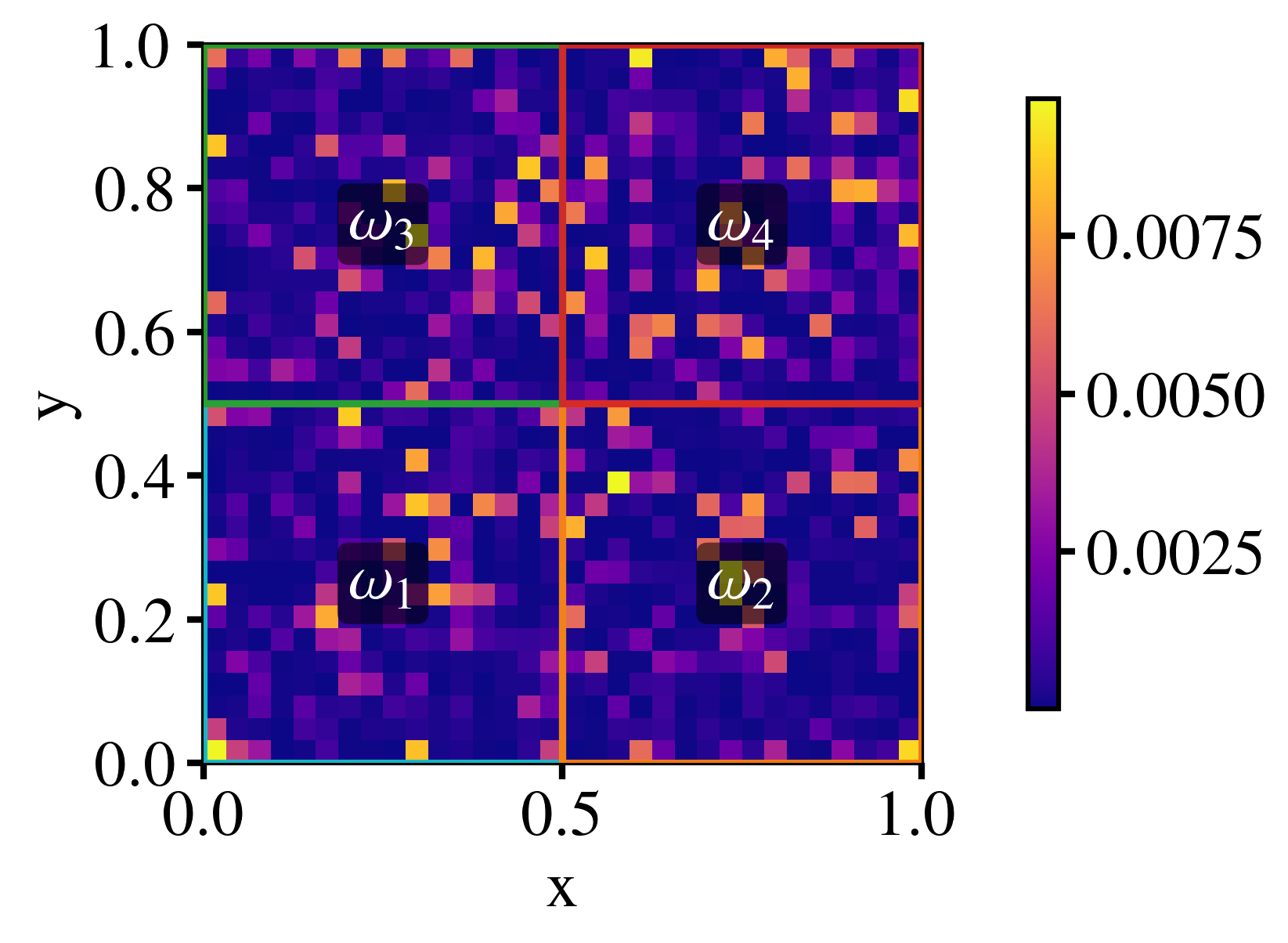}
        \subcaption{$2\times 2$ domain decomposition}
        \label{fig:kstar_parallel_2x2}
    \end{minipage}
    \caption{Experiment  1.2. Parallel recovery of the reconstructed permeability field $K^*(\bx)$using half-open domain decomposition. Left: two subdomains. Right: four subdomains.}

    \label{fig:kstar_parallel_tiles}
\end{figure}

\subsubsection{Pressure computation using the continuous recovered permeability}

We now investigate how the reconstructed continuous permeability
$K^*(\bx)$ affects the solution of the forward Darcy problem.
Using the model problem introduced in \eqref{eq:darcy_model}, we consider
the domain $\Omega=[0,1]^2$ with zero source term ($f=0$). For this
experiment, Dirichlet boundary conditions are imposed on the left and
right boundaries of the domain,
$
p = 1 ~~\text{on } x=0, ~
p = 0 ~~ \text{on } x=1,
$
while homogeneous Neumann conditions are prescribed on the remaining
boundaries.

For reference, the pressure $p$ is computed using the true permeability $K(\bx)$
on a sufficiently fine mesh. The reconstructed pressure $p^*$ is obtained by
solving the same Darcy problem with the continuous permeability $K^*(\bx)$.
Accuracy is quantified by the relative $L^2(\Omega)$ error
\begin{equation}
E_p^{\mathrm{rel}}
=
\frac{\|p - p^*\|_{L^2(\Omega)}}{\|p\|_{L^2(\Omega)}} .
\label{eq:pressure_L2_error}
\end{equation}

Figure~\ref{fig:perm_pressure_pairing} presents the permeability and pressure plots for both test cases.
For Case 1 (top row), the reconstructed permeability $K^*(\bx)$ closely matches the true permeability $K(\bx)$, and the corresponding pressure $p^*$ accurately reproduces the reference solution, yielding $E_p^{\mathrm{rel}} \approx 5.27\times10^{-2}$. For Case 2 (bottom row), which features sharper interfaces, the continuous reconstruction again leads to a stable and consistent pressure solution, with relative error
$E_p^{\mathrm{rel}} \approx 2.66\times10^{-2}$.

These results demonstrate that the recovered continuous permeability not only captures the heterogeneous coefficient structure but also preserves the essential flow behavior when used in the forward Darcy solver.

\begin{figure}[!h]
\centering

% ===================== Case 1: Perlin =====================
\begin{subfigure}[!h]{0.48\textwidth}
  \centering
  \includegraphics[width=\linewidth]{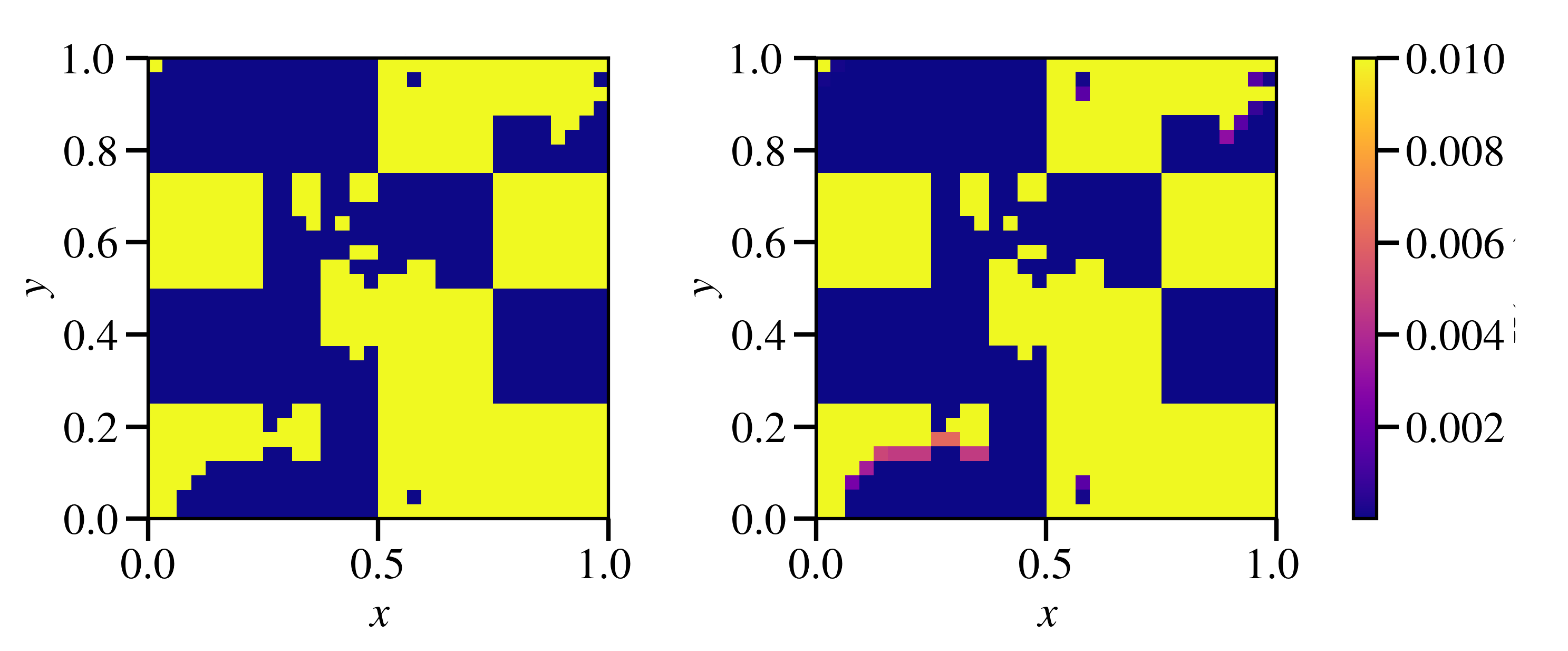}
  \caption{Case 1: $K$ (left) and $K^*$ (right)}
  \label{fig:perlin_perm}
\end{subfigure}\hfill
\begin{subfigure}[!h]{0.48\textwidth}
  \centering
  \includegraphics[width=\linewidth]{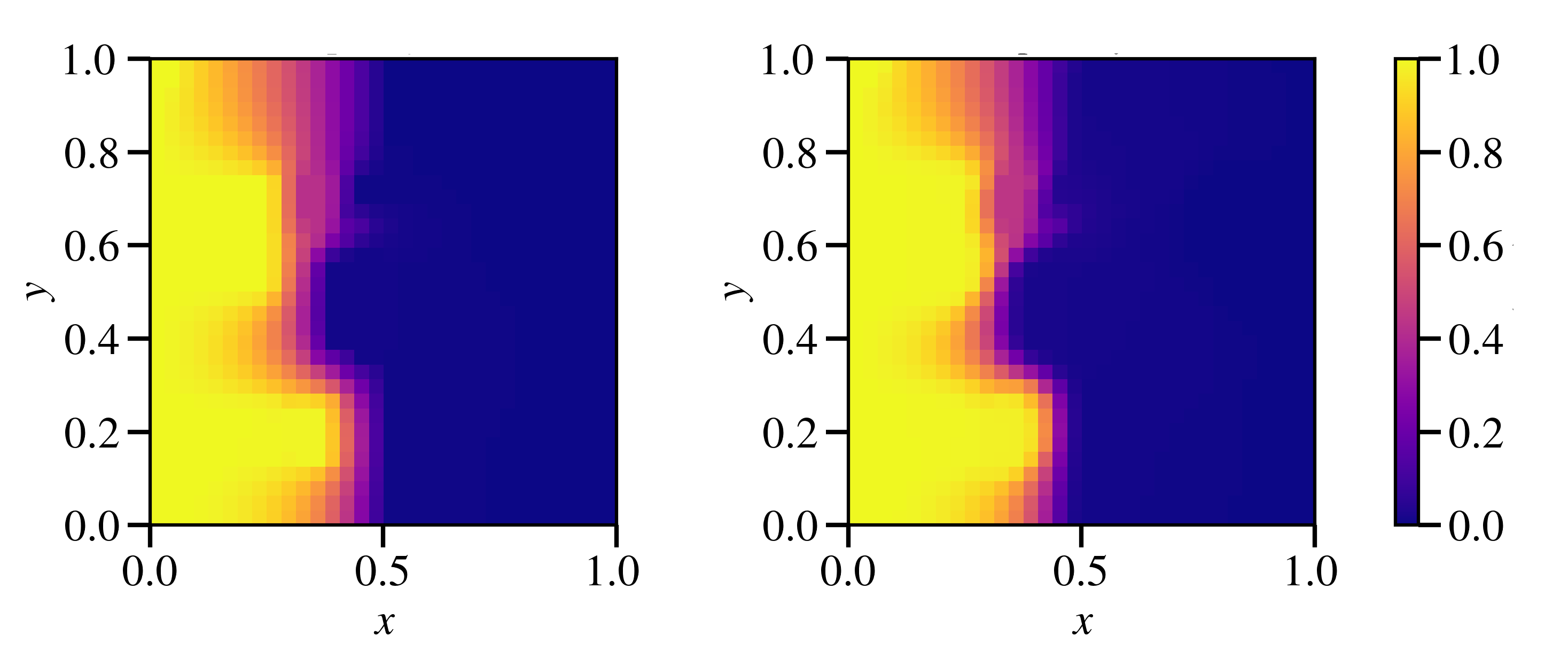}
  \caption{Case 1: $p$ (left) and $p^*$ (right)}
  \label{fig:perlin_pressure}
\end{subfigure}

\vspace{0.75em}

% ===================== Case 2: Box =====================
\begin{subfigure}[!h]{0.48\textwidth}
  \centering
  \includegraphics[width=\linewidth]{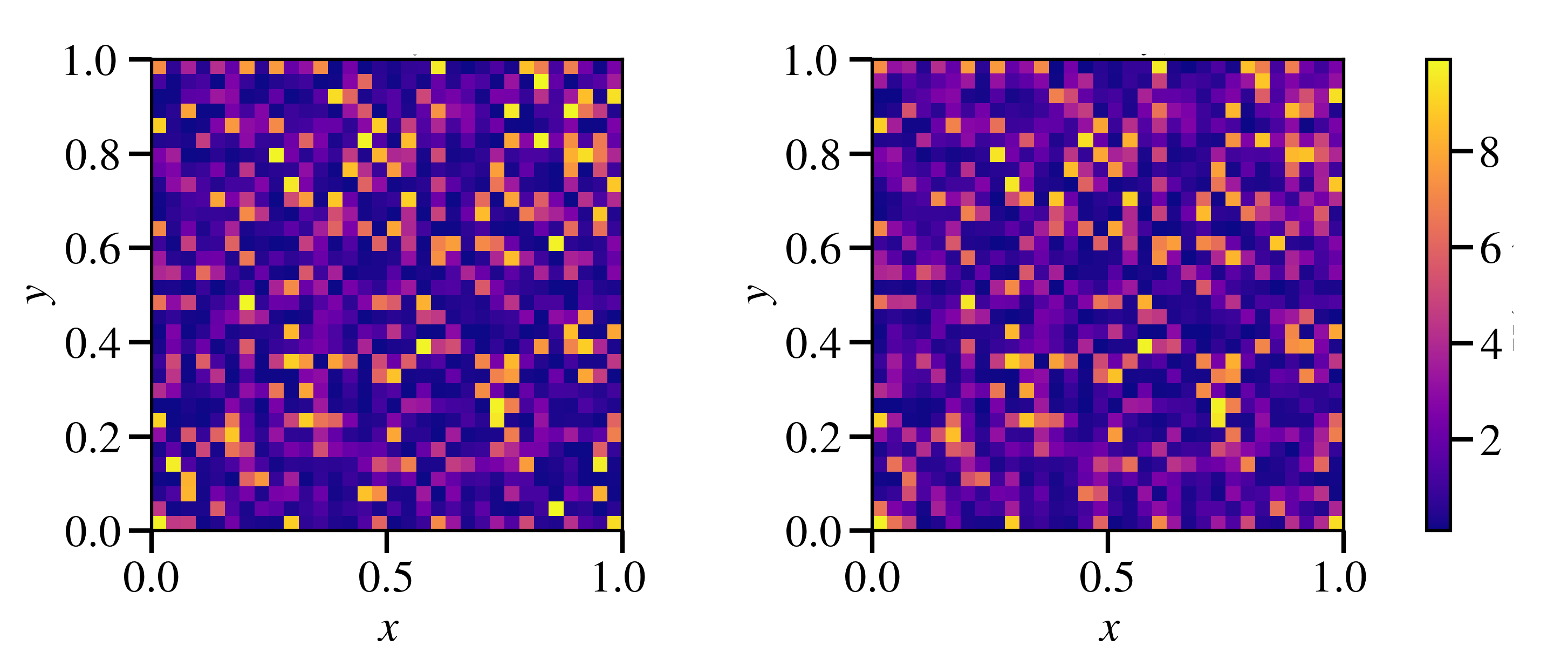}
  \caption{Case 2: $K$ (left) and $K^*$ (right)}
  \label{fig:box_perm}
\end{subfigure}\hfill
\begin{subfigure}[!h]{0.48\textwidth}
  \centering
  \includegraphics[width=\linewidth]{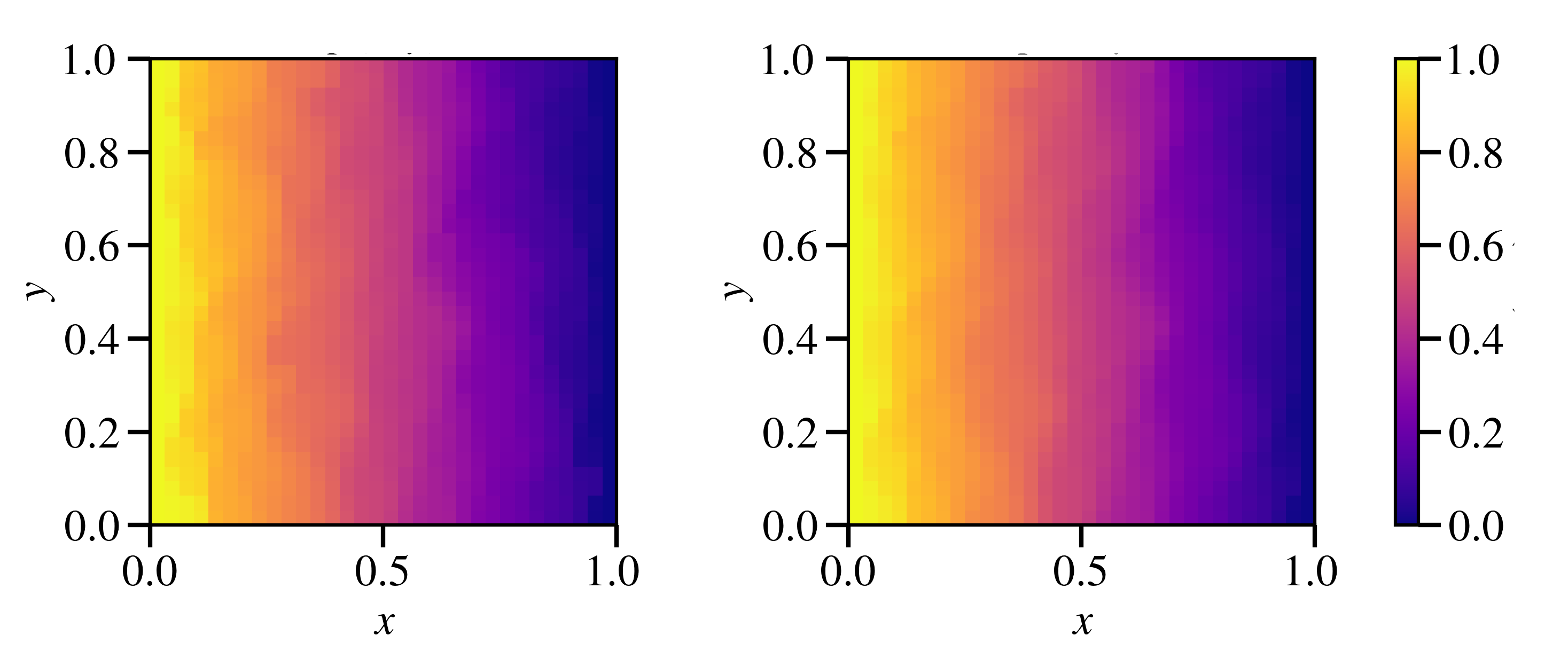}
  \caption{Case 2: $p$ (left) and $p^*$ (right)}
  \label{fig:box_pressure}
\end{subfigure}

\caption{
Permeability and pressure plots for both test cases.
Top row: Case~1 .
Bottom row: Case~2.
In each case, the pressure computed using the reconstructed continuous
permeability $K^*(\bx)$ closely matches the reference solution obtained
with the true permeability.
}
\label{fig:perm_pressure_pairing}
\end{figure}

Since the reconstructed permeability $K^*(\bx)$ is a continuous function, it can be evaluated on arbitrary finite element meshes without recomputing or modifying the permeability approximation.
This allows the same recovered coefficient to be reused across different spatial resolutions in the forward Darcy solver. Figure~\ref{fig:pressure_multires_Kstar} illustrates pressure solutions computed
with the same $K^*(\bx)$ on coarse ($16\times16$) and fine ($64\times64$) meshes for both test cases.

\begin{figure}[!h]
    \centering
    
    % --- Perlin: 16x16 and 64x64 ---
    \begin{minipage}{0.48\textwidth}
        \centering
        \includegraphics[width=\textwidth]{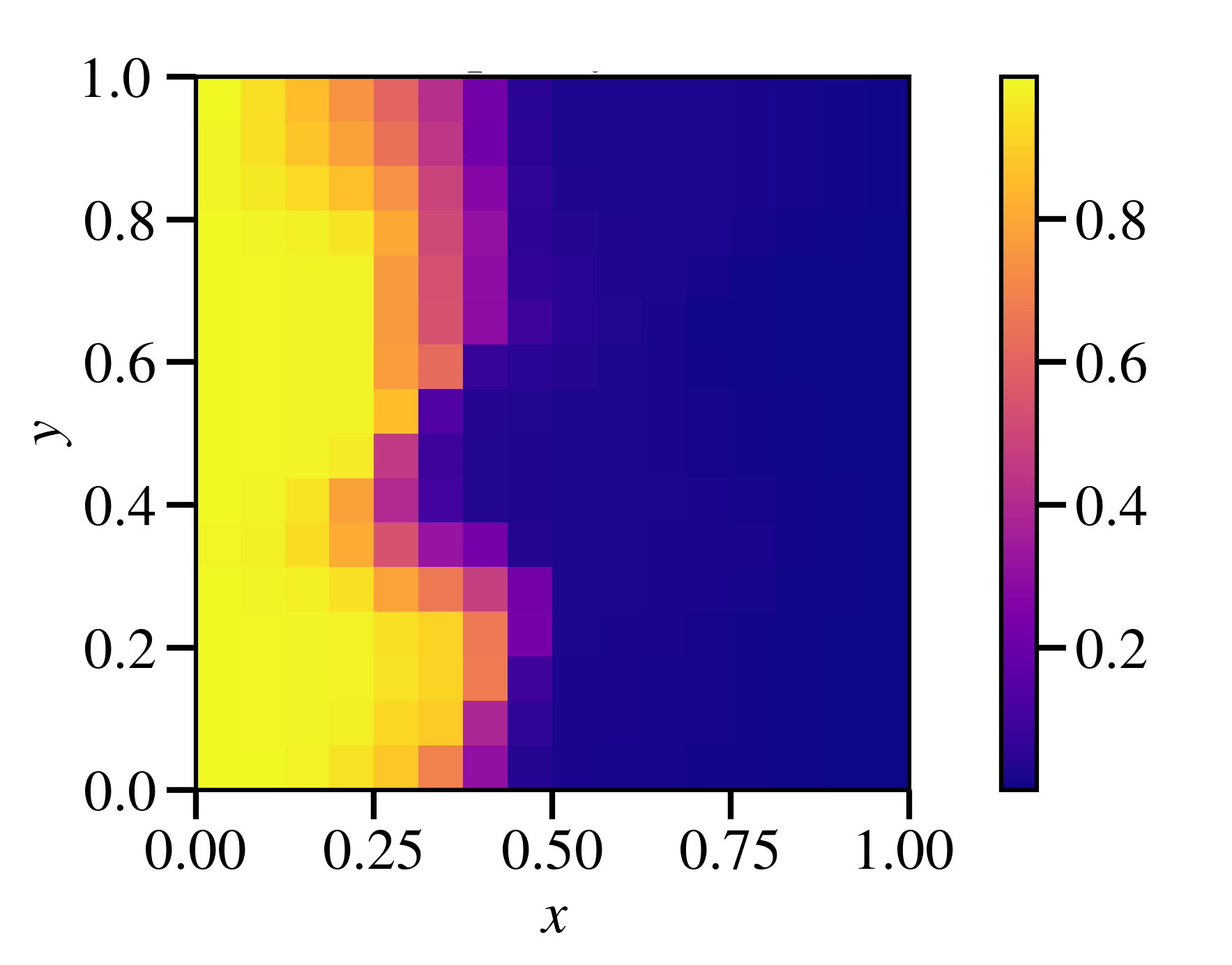}
        {\small (a) Case~1, $16\times16$ mesh}
    \end{minipage}
    \hfill
    \begin{minipage}{0.48\textwidth}
        \centering
        \includegraphics[width=\textwidth]{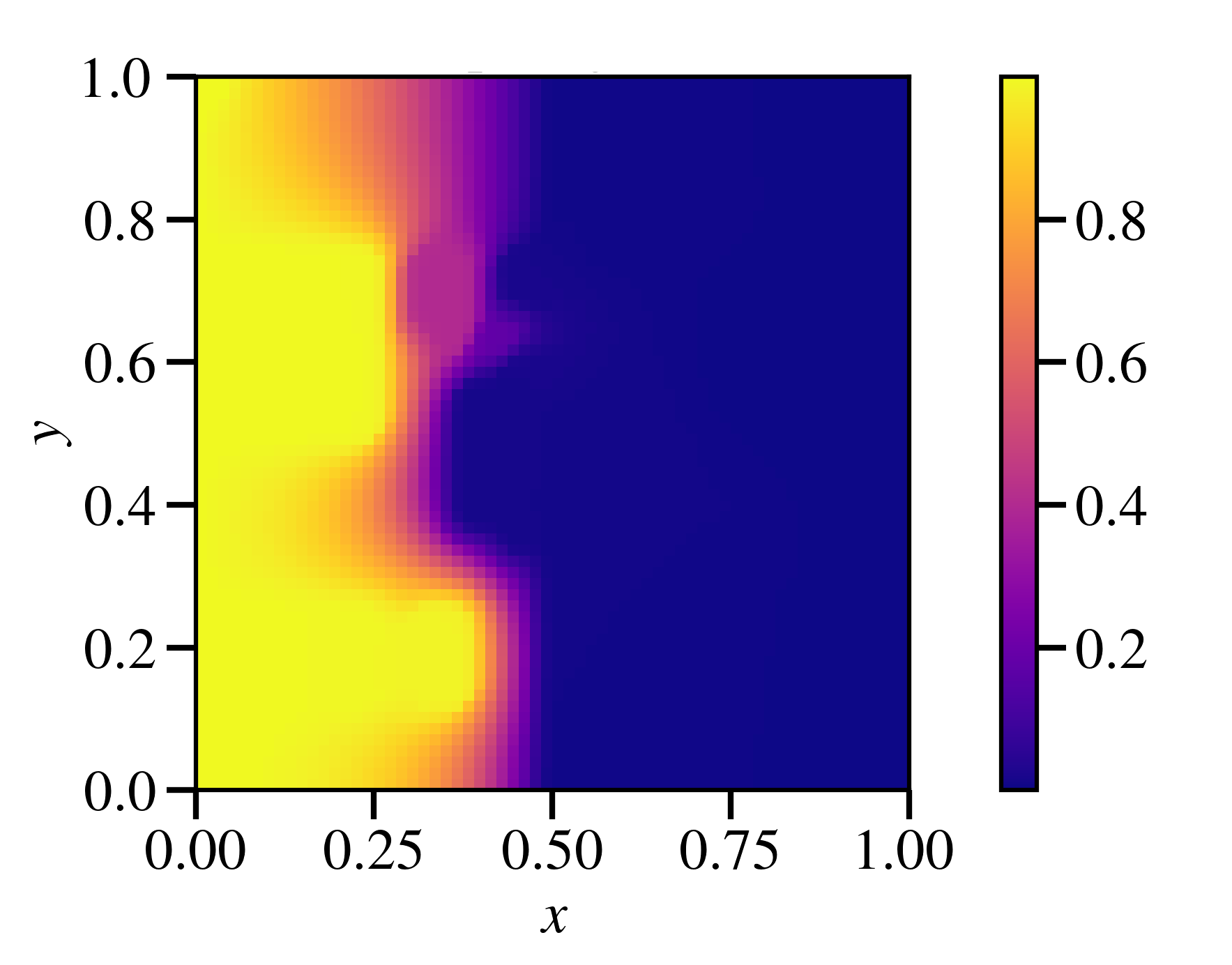}
        {\small (b) Case~1, $64\times64$ mesh}
    \end{minipage}

    \vspace{1em}

    % --- Box: 16x16 and 64x64 ---
    \begin{minipage}{0.48\textwidth}
        \centering
        \includegraphics[width=\textwidth]{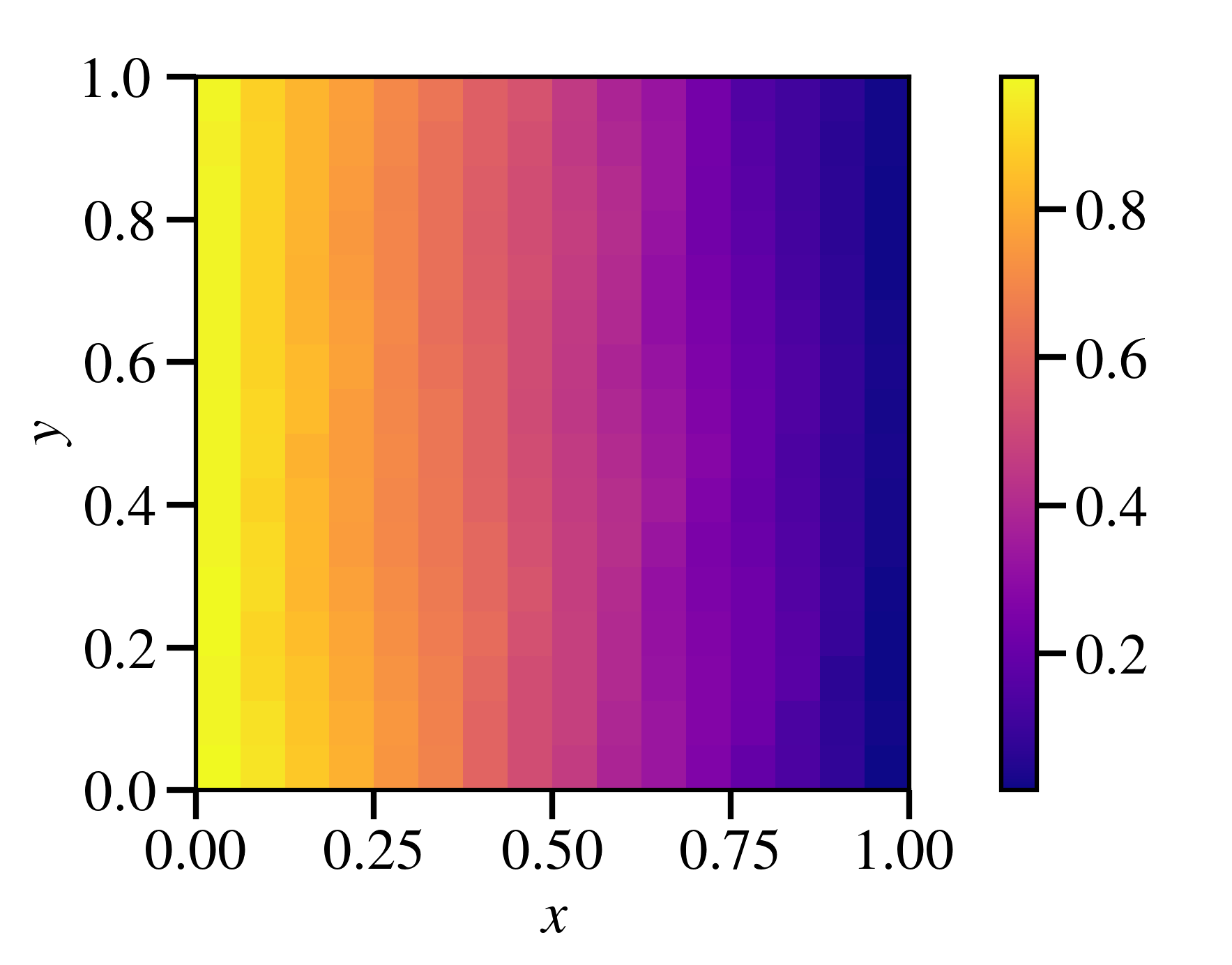}
        {\small (c) Case~2, $16\times16$ mesh}
    \end{minipage}
    \hfill
    \begin{minipage}{0.48\textwidth}
        \centering
        \includegraphics[width=\textwidth]{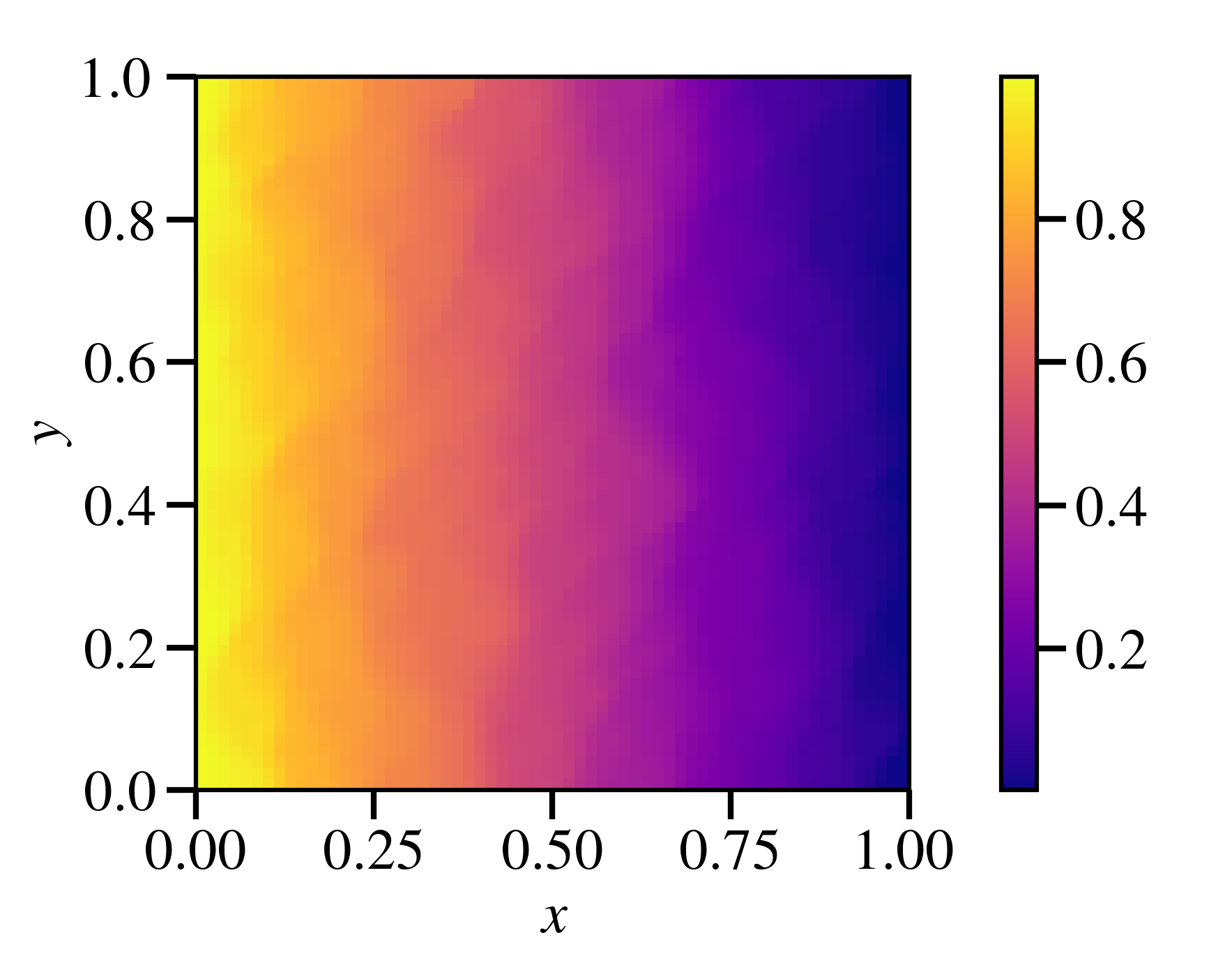}
        {\small (d) Case~2, $64\times64$ mesh}
    \end{minipage}

    \caption{
    Pressure fields computed using the same continuous reconstructed permeability
    $K^*(\bx)$ on different mesh resolutions.
    Top row: Case~1 permeability field. Bottom row: Case~2 permeability field.
    Only the mesh resolution changes; the permeability model is fixed.
    }
    \label{fig:pressure_multires_Kstar}
\end{figure}

Further, we study how the corresponding pressure solution behaves under mesh refinement. For a sequence of uniform meshes with mesh size $h=1/N_x$, we solve the Darcy problem given in \eqref{eq:darcy_model} with $K^*(\bx)$, 
and compare against the reference pressure $p_h$ computed on a fine
$64\times64$ mesh using the true permeability.

Figures~\ref{fig:conv_perlin} and~\ref{fig:conv_box} show the relative $L^2$
pressure error versus $h$ for Case~1 and Case~2. In both cases, the error
decreases approximately linearly,
$$
E_p^{\mathrm{rel}}(h)\approx C\,h,
$$
indicating approximately first-order convergence of the pressure solution. This behavior
confirms that the continuous recovered permeability $K^*(\bx)$ yields a stable
and convergent Darcy solver across mesh refinements.

\begin{figure}[!h]
    \centering
    \begin{subfigure}[!h]{0.48\textwidth}
        \centering
        \includegraphics[width=\textwidth]{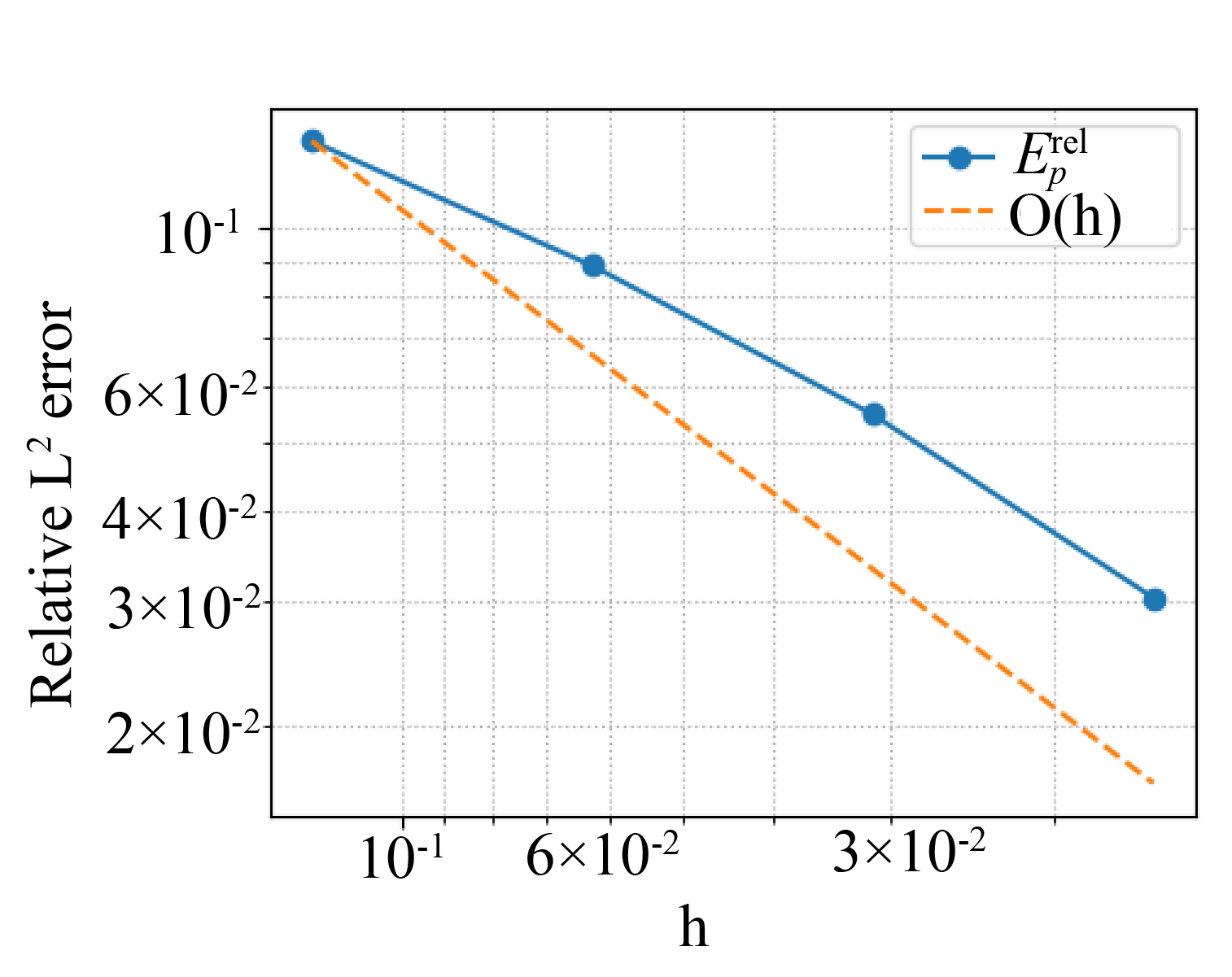}
        \caption{}
        \label{fig:conv_perlin}
    \end{subfigure}
    \hfill
    \begin{subfigure}[!h]{0.48\textwidth}
        \centering
        \includegraphics[width=\textwidth]{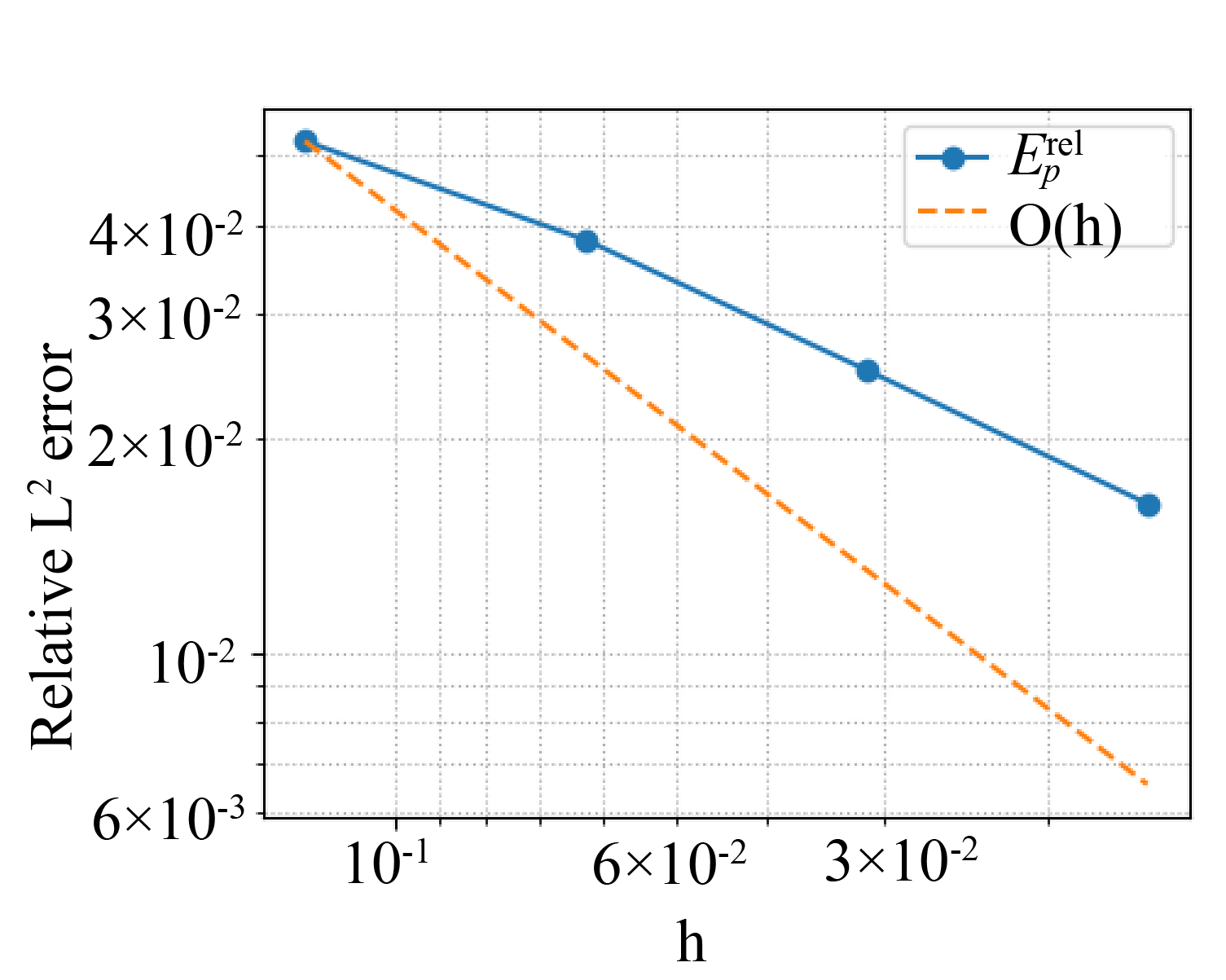}
        \caption{}
        \label{fig:conv_box}
    \end{subfigure}

    \caption{ Convergence of the pressure solution obtained using the continuous reconstructed
permeability $K^*(\bx)$ under mesh refinement:
(a) Case~1 permeability;
(b) Case~2 permeability.
}

    \label{fig:conv_both}
\end{figure}

\subsection{Experiment 2. SPE10 Benchmark Permeability Field}
\label{sec:application_spe10}

We now apply the proposed adaptive RBF framework with Elastic Net regression to a large-scale, highly heterogeneous permeability field taken from the SPE10 benchmark. Specifically, we consider a two-dimensional horizontal slice of the SPE10 model,
discretized on a $60 \times 220$ rectangular grid.  
The reference permeability $K(\bx)$ is provided as a cellwise constant field with
strong contrast and multiscale channelized structures, making this benchmark a
challenging test for coefficient recovery methods.

Due to the size and heterogeneity of the domain, the computational domain
$\Omega$ is partitioned into multiple nonoverlapping subdomains
$\{\omega_i\}$. On each subdomain, the permeability is reconstructed independently using
Algorithm~\ref{algo:adaptive}. In this experiment, we set $K_{\mathrm{top}} = 660$ per round. The initial kernel width is set to $\sigma=0.00159$.

Figure~\ref{fig:spe10_perm_truth} shows the reference SPE10 permeability field $K(\bx)$, while Figure~\ref{fig:spe10_perm_kstar} presents its continuous reconstruction $K^*(\bx)$ obtained using the proposed method. The reconstructed field accurately reproduces the dominant high- and low-permeability
regions and captures the structures present in the benchmark data. These results indicate that the proposed approach generalizes beyond synthetic
test problems to standard benchmark datasets.

\begin{figure}[ttbp]
    \centering

    \begin{subfigure}[!h]{0.49\textwidth}
        \centering
        \includegraphics[width=\linewidth]{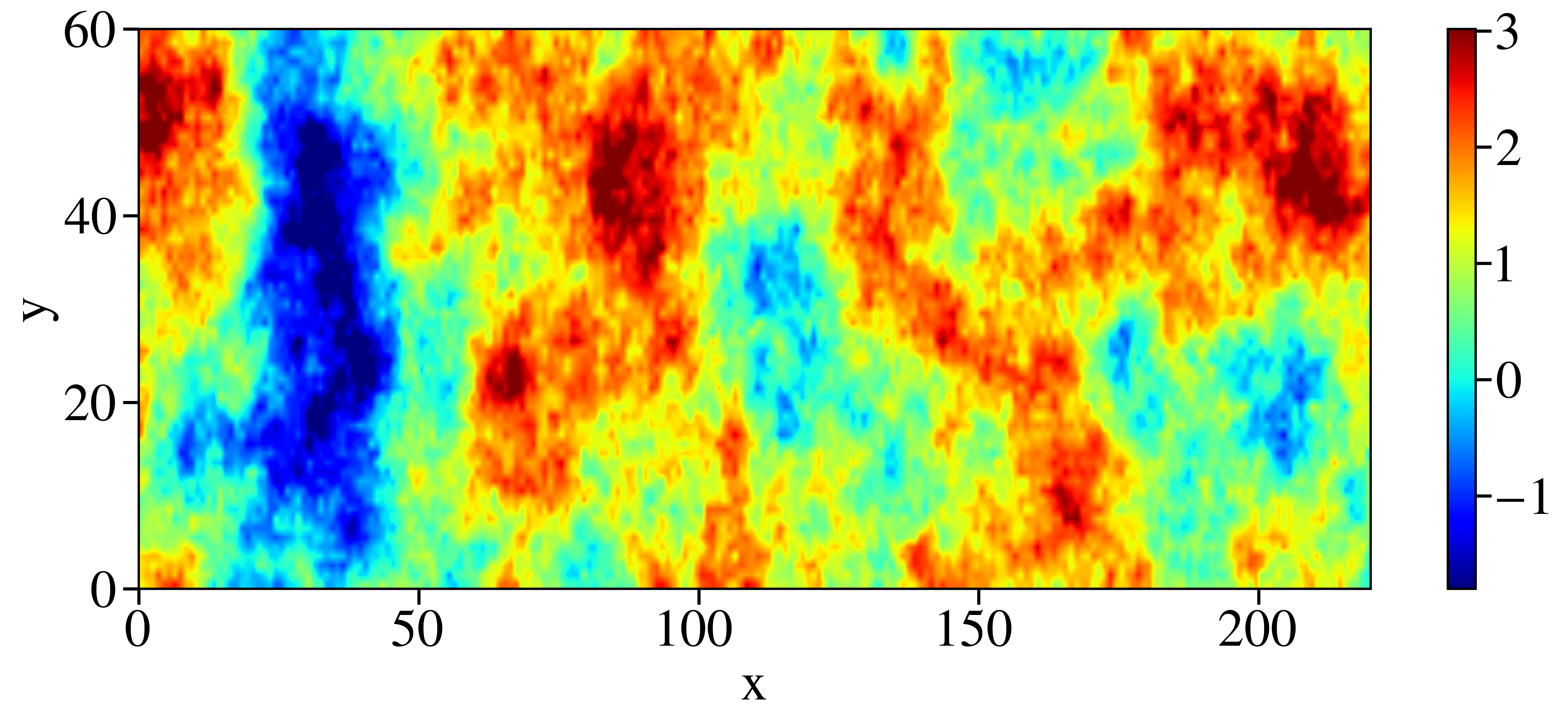}
        \caption{}
        \label{fig:spe10_perm_truth}
    \end{subfigure}
    \hfill
    \begin{subfigure}[!h]{0.49\textwidth}
        \centering
        \includegraphics[width=\linewidth]{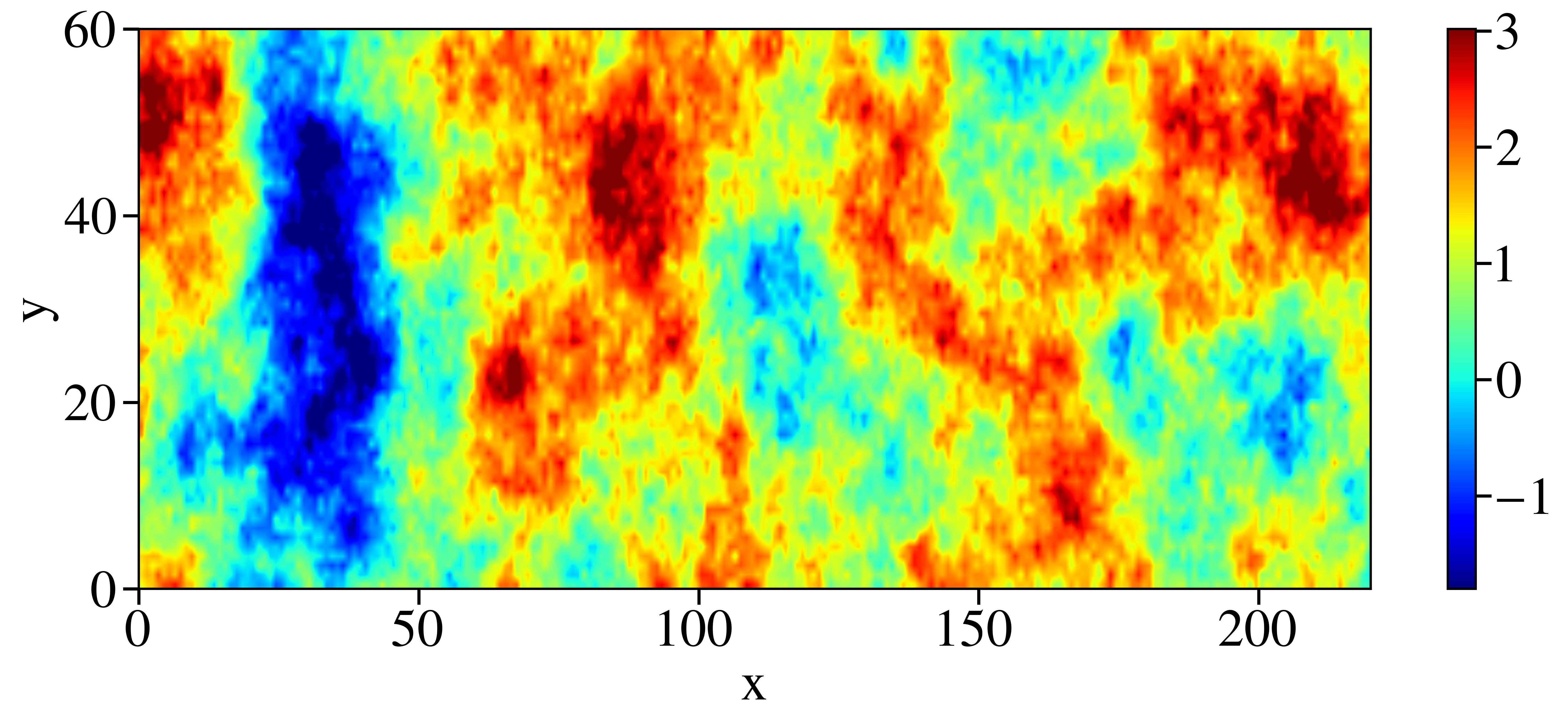}
        \caption{}
        \label{fig:spe10_perm_kstar}
    \end{subfigure}

    \caption{Experiment 2. Permeability fields from the SPE10 dataset on the $60 \times 220$ grid shown in $\log_{10}$ scale. 
    (a) Original permeability data $K(\bx)$. 
    (b) Continuous reconstruction $K^*(\bx)$ obtained using the PAM framework.}
    
    \label{fig:spe10_permeability}
\end{figure}

\begin{figure}[ttbp]
    \centering

    \begin{subfigure}[!h]{0.49\textwidth}
        \centering
        \includegraphics[width=\linewidth]{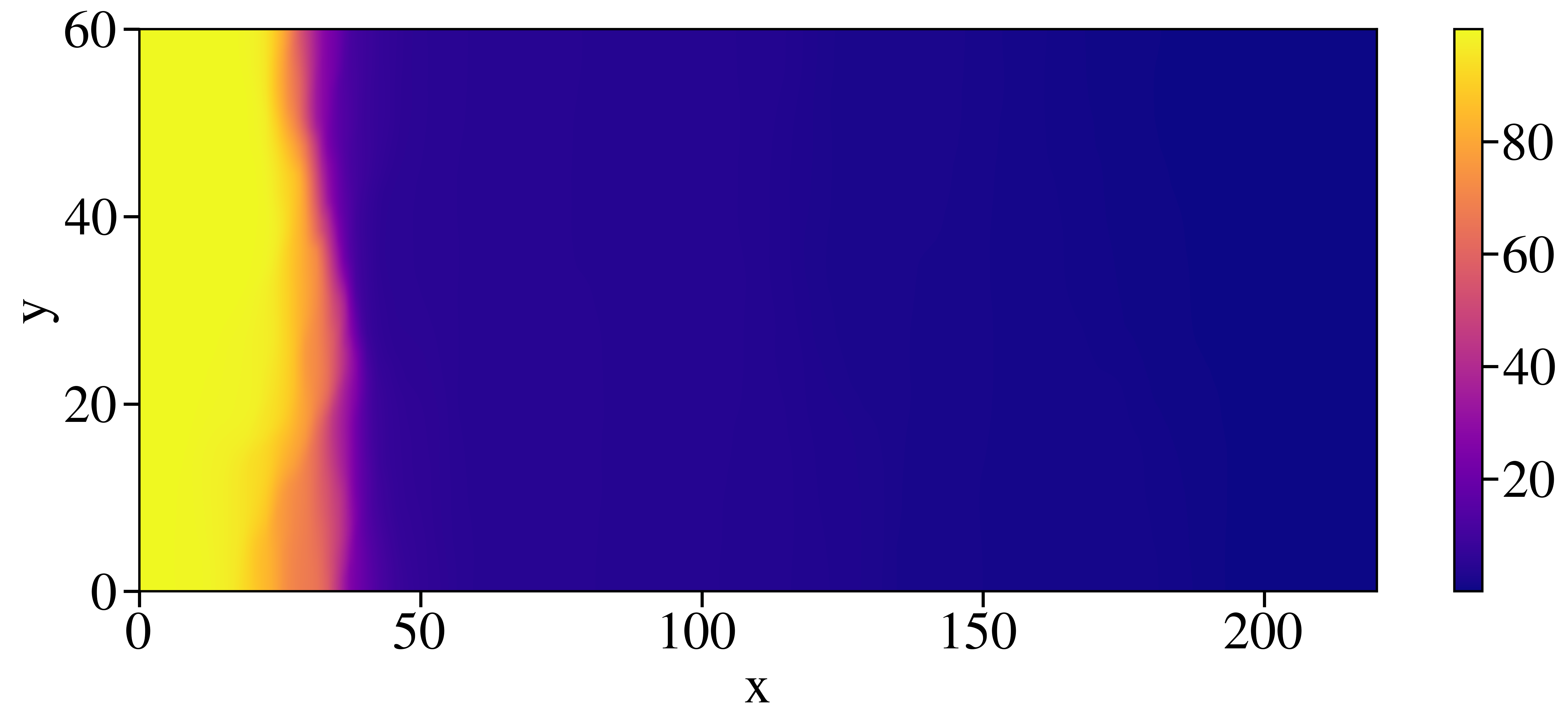}
        \caption{}
        \label{fig:spe10_pressure_true}
    \end{subfigure}
    \hfill
    \begin{subfigure}[!h]{0.49\textwidth}
        \centering
        \includegraphics[width=\linewidth]{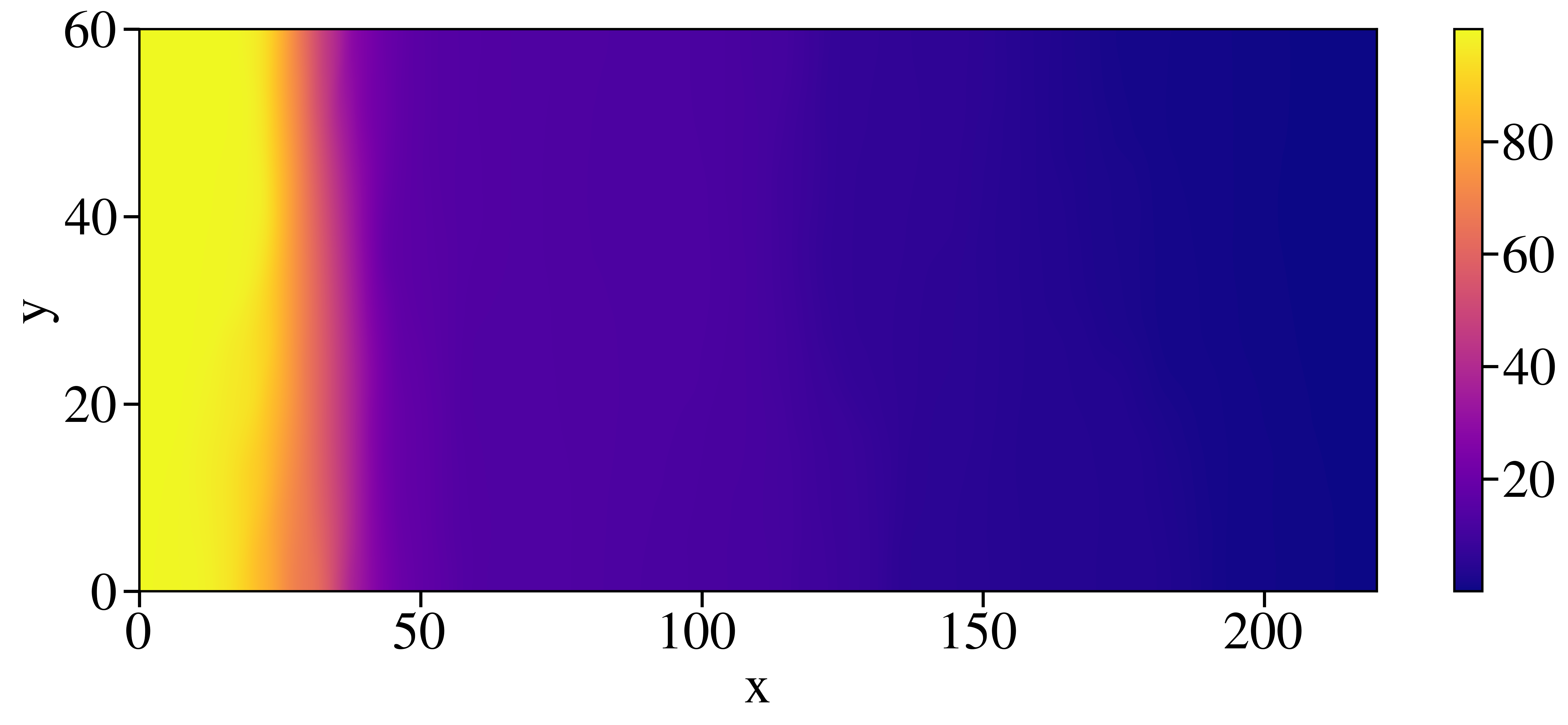}
        \caption{}
        \label{fig:spe10_pressure_kstar}
    \end{subfigure}

    \caption{Experiment 2. Pressure solutions obtained on the $60 \times 220$ grid using the SPE10 permeability data. 
    (a) Solution computed with the original permeability data $K(\bx)$. 
    (b) Solution computed with the reconstructed continuous permeability $K^*(\bx)$. 
    }

    \label{fig:spe10_pressure}
\end{figure}

To study the impact of the reconstructed permeability on flow predictions, we solve the steady-state Darcy problem (\ref{eq:darcy_model}) on the rectangular domain $\Omega$ associated with the SPE10 slice. Dirichlet boundary conditions are imposed on the inflow and outflow boundaries,
with $p=100$ at $x=0$ and $p=0$ at $x=220$, while homogeneous Neumann (no-flow) conditions are applied on the remaining boundaries. This configuration induces a predominantly left-to-right flow and is commonly
used in SPE10 benchmarking studies. Figures~\ref{fig:spe10_pressure_true} and~\ref{fig:spe10_pressure_kstar} compare the pressure fields obtained using the reference permeability $K(\bx)$ and the reconstructed permeability $K^*(\bx)$. The relative $L^2(\Omega)$ error between the corresponding pressure solutions
is approximately $1.8\times10^{-1}$.

To further examine the impact of spatial discretization on the pressure solution obtained from the reconstructed permeability, we next consider pressure fields computed on coarser and finer meshes using the same continuous permeability $K^*(\bx)$.
Figure~\ref{fig:pstar_multires} illustrates how mesh resolution affects the sharpness of the pressure transition and the presence of numerical diffusion, while preserving the overall flow structure.

\begin{figure}[ttbp]
    \centering

    \begin{subfigure}[!h]{0.49\textwidth}
        \centering
        \includegraphics[width=\linewidth]{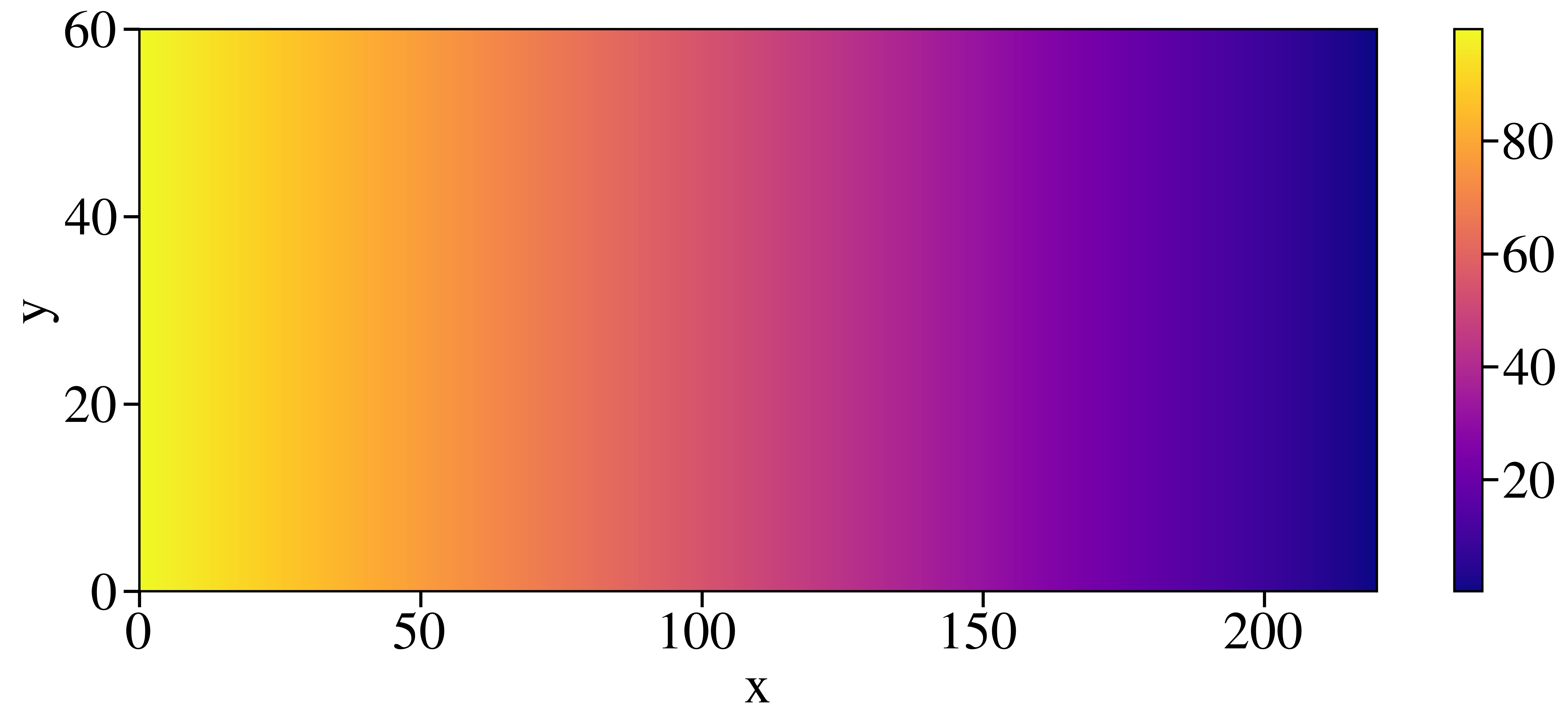}
        \caption{}
        \label{fig:spe10_pressure_coarse}
    \end{subfigure}
    \hfill
    \begin{subfigure}[!h]{0.49\textwidth}
        \centering
        \includegraphics[width=\linewidth]{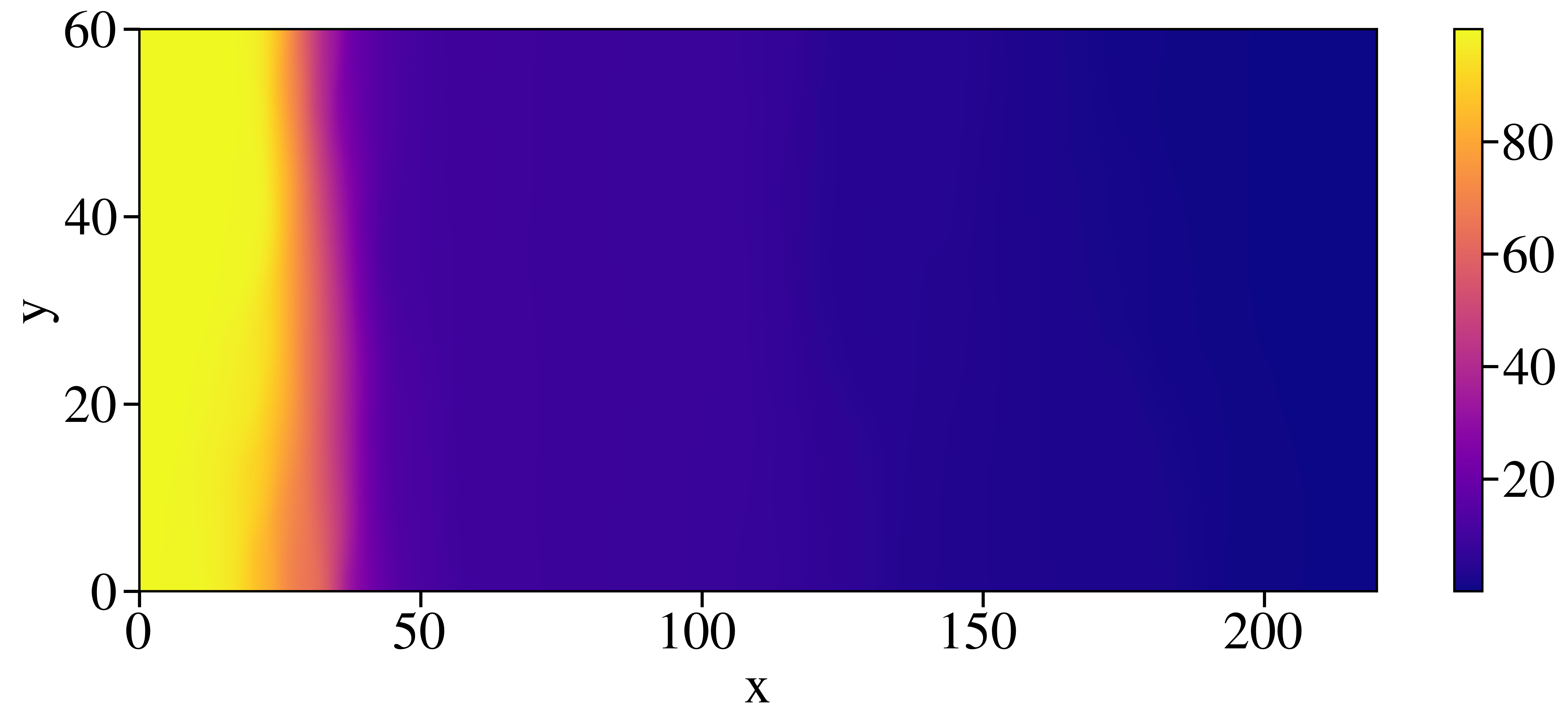}
        \caption{}
        \label{fig:spe10_pressure_fine}
    \end{subfigure}

    \caption{Experiment 2. Visualization of pressure fields on two different mesh resolution.
    (a): pressure $p^*(\bx)$ obtained on mesh of size $110\times30$.
    (b): pressure $p^*(\bx)$ obtained on mesh of size $440\times120$. 
    }

    \label{fig:pstar_multires}
\end{figure}

\subsubsection{Pressure computation on different meshes and geometries using the continuous recovered permeability}

The reconstructed permeability $K^*(\bx)$ provides a continuous representation of the originally piecewise constant SPE10 data. As a result, $K^*(\bx)$ can be evaluated on any computational mesh, regardless of resolution or geometry. To demonstrate this flexibility, we solve the Darcy pressure equation \ref{eq:darcy_model} using $K^*(\bx)$ on three representative configurations involving different meshes and domain geometries.

We first consider the original rectangular domain and solve the pressure on a non-uniform mesh obtained by introducing local refinement in selected regions. Figure~\ref{fig:geom1} shows the resulting mesh and the corresponding pressure field. The solution naturally adapts to the locally refined regions while maintaining a consistent global flow pattern. 

\begin{figure}[!h]
\centering
\begin{subfigure}[!h]{0.5\textwidth}
    \centering
    \includegraphics[width=\linewidth]{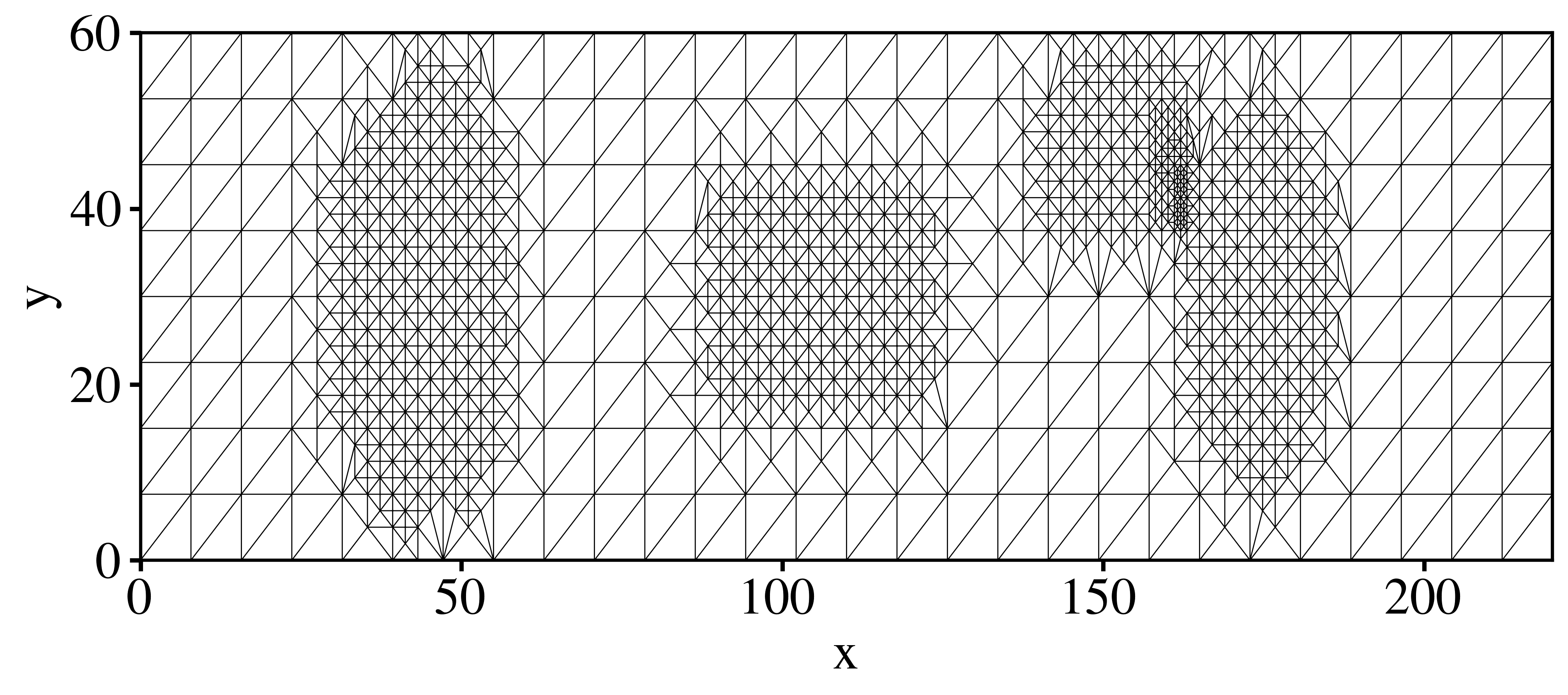}
    \caption{}
\end{subfigure}
\hfill
\begin{subfigure}[!h]{0.48\textwidth}
    \centering
    \includegraphics[width=\linewidth]{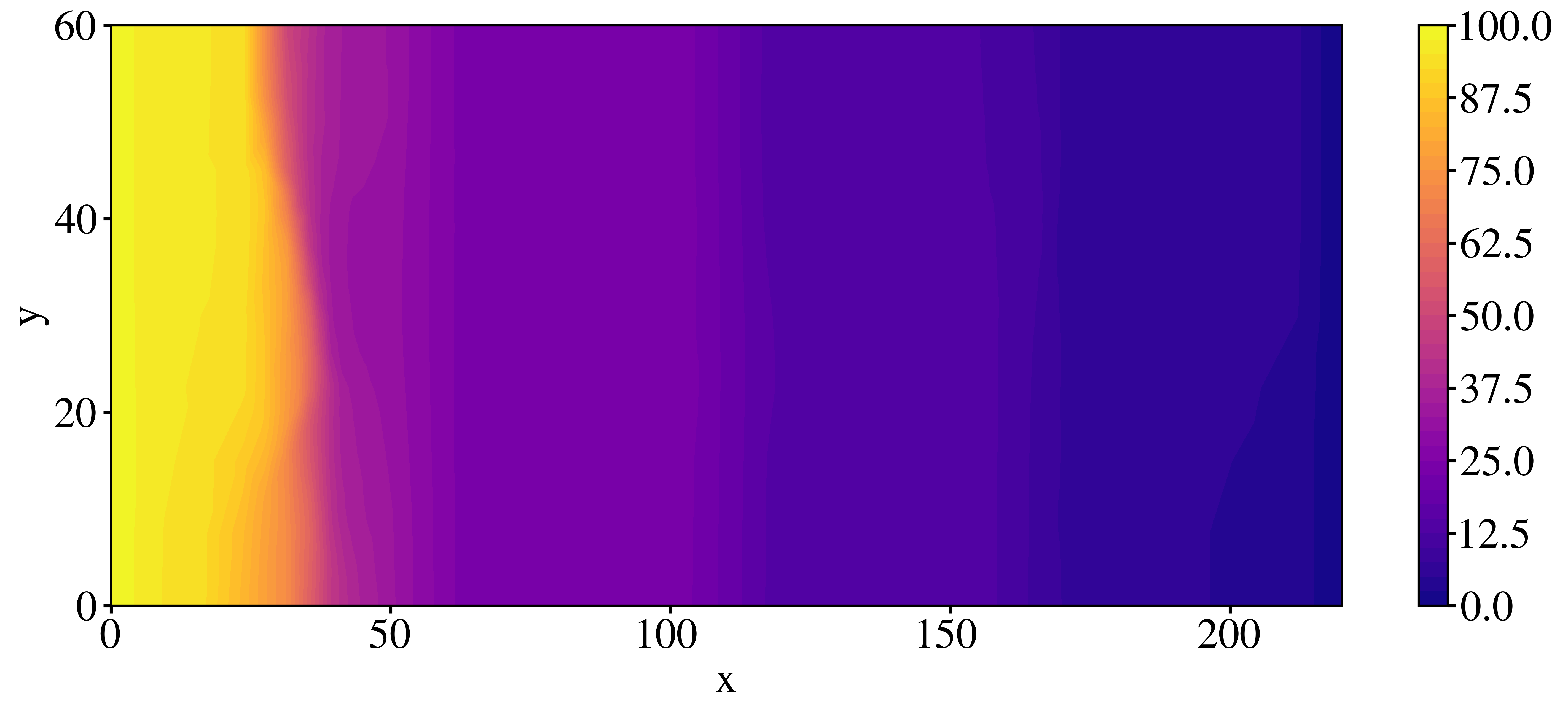}
    \caption{}
\end{subfigure}
\caption{
Geometry 1: Rectangular domain with locally refined mesh.
(a) Computational mesh.
(b) Pressure field $p^*(\bx)$ computed using $K^*(\bx)$.
}
\label{fig:geom1}
\end{figure}

Next, we consider a different domain geometry. The same permeability field $K^*(\bx)$ is evaluated directly on this domain. Figure~\ref{fig:geom2} presents the mesh and the corresponding pressure solution. The flow field adapts to the domain configuration while remaining physically consistent, illustrating that the continuous representation of $K^*(\bx)$ enables direct simulation on different geometries without modification of the permeability data.

\begin{figure}[ttbp]
\centering
\begin{subfigure}[!h]{0.5\textwidth}
    \centering
    \includegraphics[width=\linewidth]{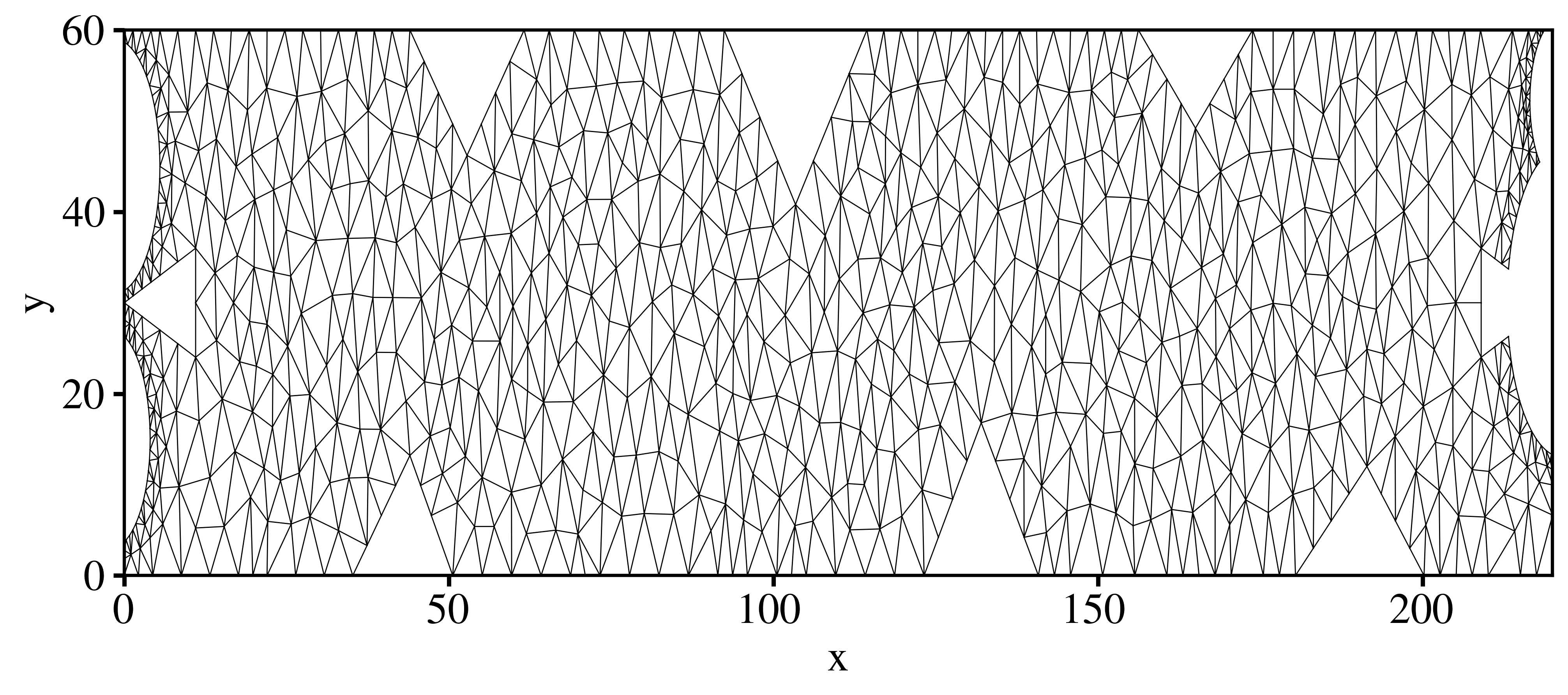}
    \caption{}
\end{subfigure}
\hfill
\begin{subfigure}[!h]{0.48\textwidth}
    \centering
    \includegraphics[width=\linewidth]{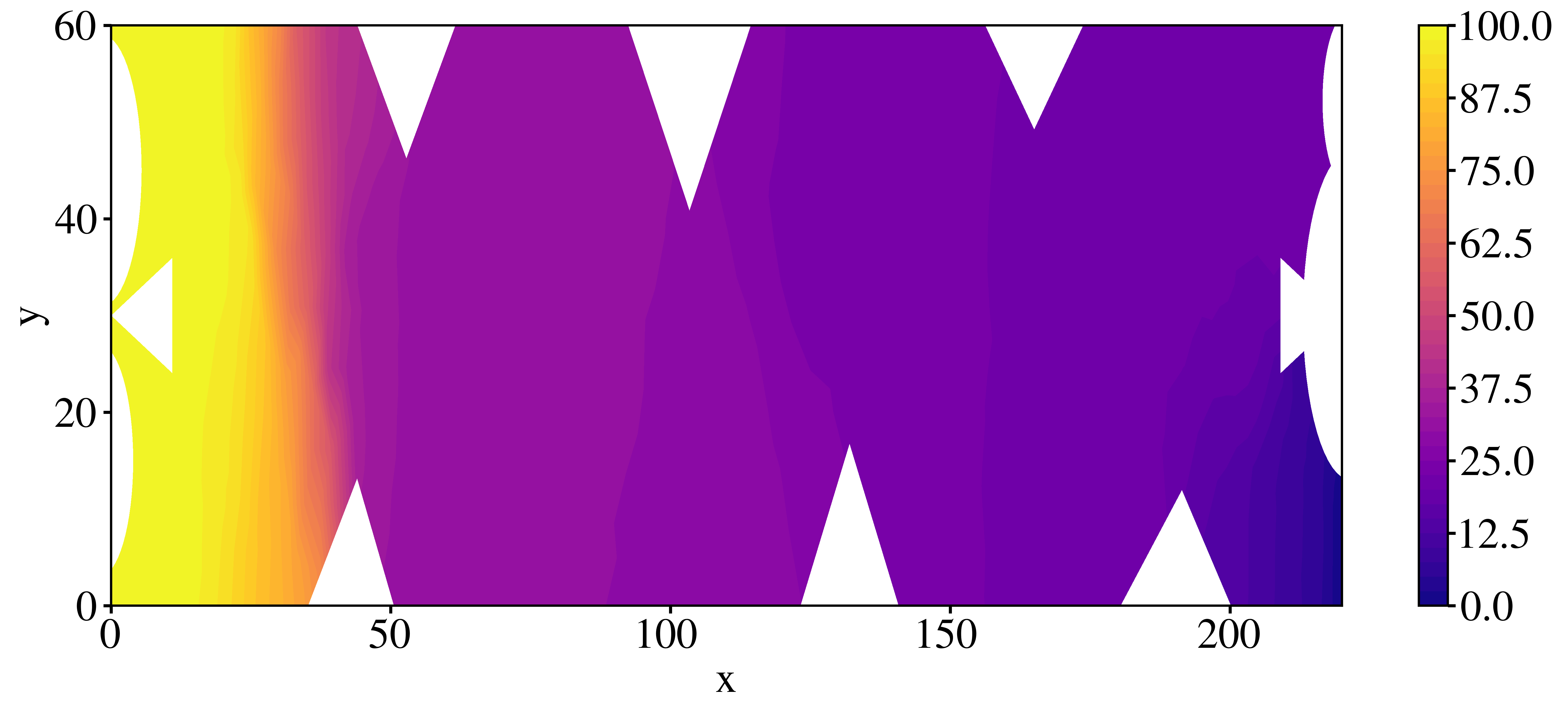}
    \caption{}
\end{subfigure}
\caption{
Geometry 2: Irregular domain with boundary cuts.
(a) Computational mesh.
(b) Pressure field $p^*(\bx)$ computed using $K^*(\bx)$.
}
\label{fig:geom2}
\end{figure}

Finally, we consider a rectangular domain containing multiple circular holes. Dirichlet boundary conditions are imposed on the left and right boundaries, while homogeneous Neumann conditions are enforced on the boundaries of the holes. Figure~\ref{fig:geom3} shows the mesh and the resulting pressure field. The solution captures the flow behavior around the circular obstacles, with streamlines bending around the inclusions. 

\begin{figure}[!h]
\centering
\begin{subfigure}[!h]{0.5\textwidth}
    \centering
    \includegraphics[width=\linewidth]{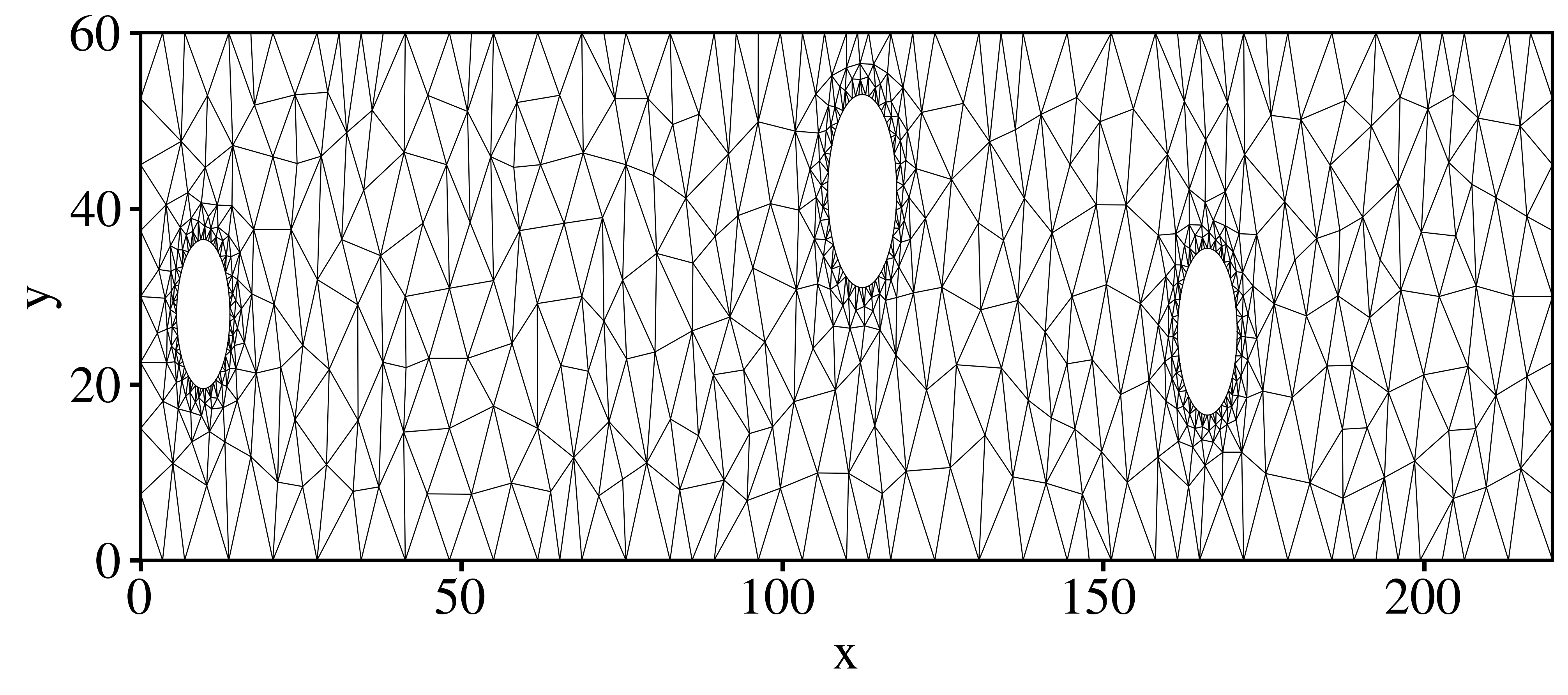}
    \caption{}
\end{subfigure}
\hfill
\begin{subfigure}[!h]{0.48\textwidth}
    \centering
    \includegraphics[width=\linewidth]{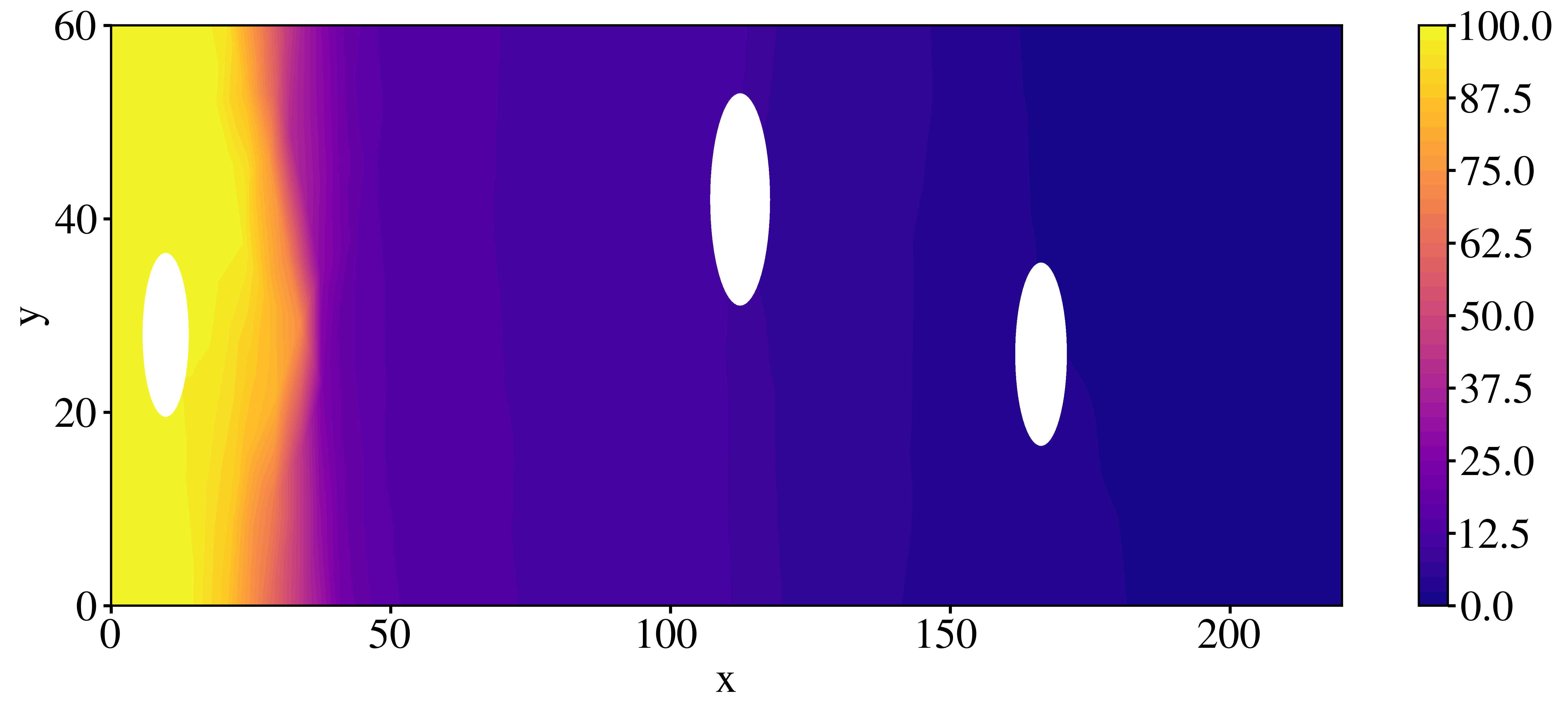}
    \caption{}
\end{subfigure}
\caption{
Geometry 3: Rectangular domain with circular holes.
(a) Computational mesh.
(b) Pressure field $p^*(\bx)$ computed using $K^*(\bx)$.
}
\label{fig:geom3}
\end{figure}

Overall, these examples show that the continuous permeability $K^*(\bx)$ can be used directly on different meshes and geometries, enabling flexible simulations with standard finite element methods.

\section{Conclusion}

In this work, we introduced PAM, a parallel and adaptive mesh-free framework for constructing continuous, closed-form approximations of discontinuous permeability data. The method combines localized radial basis function (RBF) representations, Shepard normalization, and sparse regression to accurately approximate piecewise constant fields while preserving sharp transitions.

The proposed approach is inherently mesh-independent: once constructed, the resulting surrogate can be evaluated consistently across different geometries, discretizations, and numerical schemes without the need for repeated interpolation or projection. The use of localized basis functions together with residual-driven adaptive refinement enables efficient resolution of discontinuities without requiring uniform enrichment across the domain.

Numerical experiments, including challenging heterogeneous permeability fields such as the SPE10 benchmark, demonstrate that the method captures sharp interfaces, achieves accurate approximations across multiple resolutions, and scales effectively through parallel subdomain computations.

Overall, PAM provides a flexible and scalable framework for approximating discontinuous fields by continuous representations, offering a robust alternative to traditional interpolation and mesh-dependent approaches in heterogeneous porous media.

\section*{Acknowledgments}
The work of K. Chawla and S. Lee was supported by the U.S. Department of Energy, Office of Science, Energy Earthshots Initiatives under Award Number DE-SC 0024703.

\bibliographystyle{elsarticle-num} 
\bibliography{references}
\end{document}